\documentclass[11pt,reqno,english,empty]{amsart}
\usepackage{array}
\usepackage{frcursive}
\usepackage[hypertex,
  colorlinks=true,
 linkcolor=blue,
 citecolor=blue,
]{hyperref}
\usepackage[T1]{fontenc}
\usepackage{mathrsfs}
\usepackage{amsmath,amsthm,amssymb}
\usepackage{latexsym}
\usepackage{enumerate}
\usepackage{mathrsfs}
\usepackage{stmaryrd}
\usepackage{amsopn}
\usepackage{amsmath}
\usepackage{amssymb}
\usepackage{amsfonts}
\usepackage{amsbsy}
\usepackage{amscd,indentfirst,epsfig}
\usepackage{amsfonts,amsmath,latexsym,amssymb,verbatim,amsbsy}
\usepackage{amsthm}
\usepackage{colordvi}
\usepackage{pstricks}

\def\tend{\rightarrow}
\def\R{\mathbb R}
\def\S{\mathbb S}

\def\N{\mathbb N}

\def\C{\mathcal C}
\def\H{\mathcal H}

\def\T{\mathcal T}

\def \B{\mathcal{B}}
\def\s{\sigma}
\def\a{\alpha}

\def\t{\theta}
\def\l{\lambda}
\def\L{L^1(\mathbb{S}^{d-1})}

\def\g{\gamma}
\def \u {\mathbf{u}}

\def\d{\mathrm{d}}

\def\Q{\mathcal{Q}}

\newtheorem{theo}{Theorem}[section]
\newtheorem{prop}[theo]{Proposition}
\newtheorem{cor}[theo]{Corollary}
\newtheorem{lem}[theo]{Lemma}
\newtheorem{defi}[theo]{Definition}
\newtheorem{rmq}[theo]{Remark}
\def \leq {\leqslant}
\def \geq {\geqslant}
\def \LLn {L_n(h)}
\def \LLN {L_n(h_1)}
\def \LLM {L_n(h_2)}

\numberwithin{equation}{section}

\def\beq{\begin{equation}}
\def\eeq{\end{equation}}
\def\beqn{\begin{equation*}}
\def\eeqn{\end{equation*}}
\def\bea{\begin{eqnarray}}
\def\eea{\end{eqnarray}}
\def\bean{\begin{eqnarray*}}
\def\eean{\end{eqnarray*}}
\def\bary{\begin{array}}
\def\eary{\end{array}}
\def\bal {ballistic annihilation}

\title[On ballistic annihilation]
{\textbf{Existence of self-similar profile for a kinetic annihilation model revisited
}}

\author{Véronique Bagland \& Bertrand Lods }

\address{\textbf{Véronique Bagland}, Clermont Universit\'e, Universit\'{e} Blaise Pascal, Laboratoire de Math\'{e}matiques, CNRS UMR 6620,  BP 10448, F-63000 Clermont-Ferrand,
France.}\email{Veronique.Bagland@math.univ-bpclermont.fr}

\address{\textbf{Bertrand Lods}, Universit\`{a} degli
Studi di Torino \& Collegio Carlo Alberto, Department of Economics and Statistics, Corso Unione Sovietica, 218/bis, 10134 Torino, Italy.}\email{lods@econ.unito.it}

\begin{document}

\maketitle

\begin{abstract}
We show the existence of a self-similar solution for a modified Boltzmann equation describing probabilistic ballistic annihilation.  Such a model describes a system of hard spheres such that, whenever two particles meet, they either annihilate with probability $\alpha \in (0,1)$ or they undergo an elastic collision with probability $1 - \alpha$. For such a model, the number of particles, the linear momentum and the kinetic energy are not conserved. We show that, for $\alpha$ smaller than some explicit threshold value $ \alpha_1$, a self-similar solution exists.

\medskip


\end{abstract}

\tableofcontents

\section{Introduction}

In the physics literature, various kinetic models have been proposed in the recent years in order to test the relevance of non-equilibrium statistical
mechanics for systems of \textit{reacting particles}. Such models are very challenging in particular for the derivation of suitable hydrodynamic models because of the  lack of collisional invariants.
We investigate in the present paper a recent model, introduced in \cite{Ben-Naim,coppex04, coppex05, Kaprivsky, Piasecki,Trizac} to describe the so-called \textbf{\textit{probabilistic ballistic annihilation}}.  Such a model describes a system of (elastic) hard spheres that interact in the following way: particles move freely (ballistically) between collisions while, whenever two particles meet, they either annihilate with probability $\alpha \in (0,1)$ (and both the interacting particles disappear from the system), or they undergo an elastic collision with probability $1 - \alpha$. For such a model, not only the kinetic energy is not conserved during binary encounters, but also the number of particles and the linear
momentum. Notice that, originally only  pure annihilation has been considered \cite{Ben-Naim,Kaprivsky} (corresponding to $\alpha=1$). Later on, a more elaborate model has been built which allows to recover the classical Boltzmann equation for hard spheres in the limit $\alpha = 0$. Notice that such a Boltzmann equation for ballistic annihilation in the special (and unphysical) case of Maxwellian molecules has already been studied in the mid-80's \cite{spiga,santos}  and was referred to as \textit{Boltzmann equation with removal}.

The present paper is the first mathematical investigation of the physical model of probabilistic ballistic annihilation for the physical relevant hard spheres interactions, with the noticeable exception of the results of  \cite{theile} which prove  the validity of the spatially homogeneous Boltzmann equation for pure annihilation (i.e. whenever $\alpha=1$). We shall in particular prove the existence of special self-similar profile for the associated equation. Before entering into details of our results, let us introduce more precisely the model we aim to investigate.

\subsection{The Boltzmann equation for \bal}  In a kinetic framework, the behavior of a system of hard  spheres which annihilate with probability $\alpha \in (0,1)$ or collide elastically with probability $1-\alpha$ can be described (in a spatially homogeneous situation) by the so-called velocity distribution $f(t,v)$ which represents the probability density of particles with velocity $v \in \R^d$ $(d \geq 2)$ at time $t \geq 0.$  The time-evolution of the
one-particle distribution function \(f(t,v)\), \(v\in\R^d\),
\(t>0\) satisfies the following
\begin{equation}\label{BE}
\partial_t f(t,v)=(1-\alpha)\Q(f,f)(t,v) -\alpha \Q_-(f,f)(t,v)=\mathbb{B}(f,f)(t,v)
\end{equation}
where $\Q$ is the quadratic Boltzmann collision operator defined by the bilinear symmetrized form
 \begin{equation*}\label{bilin}
 \Q(g,f)(v) = \frac{1}{2}\,\int _{\R^d \times \S^{d-1}} \B(v-v_*,\s)
         \left(g'_* f' + g' f_* '- g_* f - g f_* \right) \, \d v_* \, \d\sigma,
 \end{equation*}
where we have used the shorthands $f=f(v)$, $f'=f(v')$, $g_*=g(v_*)$ and
$g'_*=g(v'_*)$ with post-collisional velocities $v'$ and $v'_*$  parametrized by
 \begin{equation*}\label{eq:rel:vit}
 v'  =  \frac{v+v_*}{2} + \frac{|v-v_*|}{2}\;\s,   \qquad
v'_* =  \frac{v+v_*}{2} - \frac{|v-v_*|}{2}\;\s,   \qquad \s  \in  {\S}^{d-1}
 \end{equation*}
and the collision kernel is given by
$$\B(v-v_*,\s)=\Phi( |v-v_*|)b(\cos \theta)$$
where $\cos \theta=\left\langle \frac{v-v_*}{|v-v_*|}, \s \right\rangle.$ Typically, for the model we have in mind, we shall deal with
$$\Phi(|v-v_*|)=|v-v_*|$$
and constant $b(\cdot)$ corresponding to hard spheres interactions which is the model usually considered in the physics literature \cite{Maynar1,Maynar2,Trizac}. We shall also consider more general kernel, typically, we shall assume that
\begin{equation}\label{Phig}\Phi(|v-v_*|)=|v-v_*|^\gamma \qquad  \gamma \in (0,1]\end{equation}
and
$$\|b\,\|_{\L}:=|\mathbb{S}^{d-2}|\int_{-1}^1 b(t)(1-t^2)^{(d-3)/2}\d t < \infty$$
where $|\mathbb{S}^{d-2}|$ is the area of $(d-2)$-dimensional unit sphere. Without loss of generality, we will assume in all the paper that
 $$\|b\,\|_{\L}=1.$$
 Notice that, for constant angular cross-section, this amounts to choose $b(\cdot)=1/|\mathbb{S}^{d-1}|$. A very special model is the one of so-called Maxwellian molecules which corresponds to $\gamma=0$. The model of Maxwellian molecules has been studied mathematically in \cite{santos, spiga} and we will discuss this very special case in Appendix \ref{appendixB}.

The above collision operator $\Q(f,f)$ splits as $\Q(f,f)=\Q_+(f,f)-\Q_-(f,f)$ where the gain part $\Q_+$ is given by
$$ \Q_+(f,f)(v) =  \int _{\R^d \times {\S}^{d-1}}  \B(v-v_*,\s) f'_* f'\, \d v_* \,  \d\sigma$$
while the loss part $\Q_-$ is defined as
$$\Q_-(f,f)(v)=f(v)L(f)(v), \qquad \text{ with } \qquad L(f)(v)=\int _{\R^d \times  {\S}^{d-1}}\B(v-v_*,\s) f_*\, \d v_* \, \d\sigma.$$
One has
$$\mathbb{B}(f,f):=(1-\alpha)\Q(f,f) -\alpha \Q_-(f,f)=(1-\alpha)\Q_+(f,f) -\Q_-(f,f).$$
Formally, if $f(t,v)$ denotes a nonnegative solution to \eqref{BE} then, no macroscopic quantities are conserved. For instance, the number density
$$n(t)=\int_{\R^d}f(t,v)\d v$$
and the kinetic energy
$$E(t)=\int_{\R^d}|v|^2\,f(t,v)\d v$$
are continuously decreasing since, multiplying \eqref{BE} by $1$ or $|v|^2$ and integrating with respect to $v$, one formally obtains
$$\dfrac{\d}{\d t}n(t)=-\alpha \int_{\R^d}\Q_-(f,f)(t,v)\d v \leq 0$$
while
$$\dfrac{\d}{\d t}E(t)=-\alpha \int_{\R^d} |v|^2\Q_-(f,f)(t,v)\d v \leq 0.$$
It is clear therefore that \eqref{BE} does not admit any nontrivial steady
solution and, still formally, $f(t,v) \to 0$ as $t \to \infty.$

\subsection{Scaling solutions}  Physicists expect that solutions to \eqref{BE}
should approach for large times a self-similar solution $f_H$ to \eqref{BE} of
the form
\begin{equation}\label{scalingfpsi}f_H(t,v)=\lambda(t)\,\psi_H(\beta(t)v)\end{equation}
for some suitable scaled functions $\lambda(t), \beta(t) \geq 0$  with $\lambda(0)=\beta(0)=1$
and some nonnegative function  $\psi_H=\psi_H(\xi)$ such that
\beq\label{bas}
 \psi_H \ \equiv \hspace{-3.5mm} / \hspace{2mm} 0 \qquad \mbox{ and } \qquad
\int_{\R^d} \psi_H(\xi) \, (1+|\xi|^2)\, \d\xi <\infty.
\eeq
The first step in the proof of the above statement is actually the existence
of the profile $\psi_H$ and \textit{\textbf{this is the aim of the present paper}}.

Using the scaling properties of the Boltzmann collision operators $\Q_\pm$,
one checks easily that
$$\mathbb{B}(f_H,f_H)(t,v) = \lambda^2(t)\beta^{-(d+\gamma)}(t)
\mathbb{B}(\psi_H,\psi_H)(\beta(t)v)  \qquad \forall v \in \R^d.$$
Then,  $f_H(t,v)$ is a solution to \eqref{BE} if and
only if $\psi_H(\xi)$ is a solution to the rescaled problem
$$\dfrac{\dot{\lambda}(t)\beta^{d+\gamma}(t)}{\lambda^2(t)}\psi_H(\xi)
+\dfrac{\dot{\beta}(t)\beta^{d+\gamma-1}(t)}{\lambda(t)}\xi \cdot \nabla_\xi \psi_H(\xi)
=\mathbb{B}(\psi_H,\psi_H)(\xi)$$
where the dot symbol stands for the time derivative. Since $\psi_H$ does not depend on time $t$, there exist some constants $\mathbf{A}$ and $\mathbf{B}$  such that
\begin{equation}\label{ABlambdabeta}\mathbf{A}=\dfrac{\dot{\lambda}(t)\beta^{d+\gamma}(t)}{\lambda^2(t)}, \qquad
\mathbf{B}=\dfrac{\dot{\beta}(t)\beta^{d+\gamma-1}(t)}{\lambda(t)}\end{equation}
Thereby, $\psi_H$ is a solution to
\begin{equation}
\label{tauT}
\mathbf{A}\psi_H(\xi)+\mathbf{B}\xi \cdot \nabla_\xi \psi_H(\xi)
=\mathbb{B}(\psi_H,\psi_H)(\xi).
\end{equation}
Actually, one sees easily that the coefficients $\mathbf{A}$ and $\mathbf{B}$ depend on the profile $\psi_H$. Indeed, integrating first \eqref{tauT} with respect to $\xi$ and then multiplying  \eqref{tauT} by $|\xi|^2$ and integrating again with respect to $\xi$ one sees that \eqref{bas} implies that
$$\mathbf{A}=-\frac{\alpha}{2} \int_{\R^d}
\left(\frac{d+2}{\int_{\R^d} \psi_H(\xi_*)\, \d\xi_*}
-\frac{d\,|\xi|^2}{\int_{\R^d} \psi_H(\xi_*)\, |\xi_*|^2\, \d\xi_*}\right)
\Q_-(\psi_H ,\psi_H )(\xi)\d \xi$$
and
$$\mathbf{B} =-\frac{\alpha}{2}\int_{\R^d}
\left(\frac{1}{\int_{\R^d} \psi_H(\xi_*)\, \d\xi_*}
-\frac{|\xi|^2}{\int_{\R^d} \psi_H(\xi_*)\, |\xi_*|^2\, \d\xi_*}\right)
\Q_-(\psi_H ,\psi_H )(\xi)\d \xi.$$
 Let us note that $\mathbf{A}$ and $\mathbf{B}$ have no sign. However,

$$0 < d\mathbf{B}-\mathbf{A}=\frac{\alpha}{\displaystyle\int_{\R^d} \psi_H(\xi_*)\, \d\xi_*}
  \displaystyle \int_{\R^d} \Q_-(\psi_H,\psi_H)(\xi)\, \d \xi\,,$$
and
$$0 < (d+2)\mathbf{B}-\mathbf{A}= \dfrac{\alpha}{\displaystyle\int_{ \R^d} \psi_H(\xi_*)\,|\xi_*|^2 \, \d\xi_*} \displaystyle \int_{\R^d} |\xi|^2\Q_-(\psi_H,\psi_H)(\xi)\, \d \xi\,.$$
  Solving \eqref{ABlambdabeta}, one obtains the expressions of $\beta$ and $\lambda$. Namely, since $\lambda(0)=\beta(0)=1$,
$$\begin{cases} \beta(t)&=\left(1+\left((d+\gamma)\mathbf{B}-\mathbf{A}\right)t\right)^{\frac{\mathbf{B}}{(d+\gamma)\mathbf{B}-\mathbf{A}}}\\
\lambda(t)&=\left(1+\left((d+\gamma)\mathbf{B}-\mathbf{A}\right)t\right)^{\frac{\mathbf{A}}{(d+\gamma)\mathbf{B}-\mathbf{A}}} \qquad t \geq 0\end{cases}$$
where we notice that $(d+\gamma)\mathbf{B}-\mathbf{A} > 0.$

We now observe that, with no loss of generality, one may assume that
\begin{equation}\label{init}\int_{\R^d} \psi_H(\xi)\,\d\xi=1 \qquad
\mbox{ and } \qquad \int_{\R^d} \psi_H(\xi)\,|\xi|^2\, \d\xi=\frac{d}{2}.
\end{equation}
Indeed, if $\psi_H$ denotes a solution to \eqref{tauT} satisfying
\eqref{init} then, for any $\beta=(\beta_1,\beta_2) \in (0,\infty)^2$, the function
$\psi_{H,\beta}$ defined by
$$\psi_{H,\beta}(\xi)=\beta_1\left(\frac{d\beta_1}{2\beta_2}\right)^{\frac{d}{2}}
\psi_H\left(\sqrt{\frac{d\beta_1}{2\beta_2}}\; \xi\right)$$ is a solution
to \eqref{tauT} with mass $\beta_1$ and energy $\beta_2$. Assuming \eqref{init} and introducing
$$n_H(t)=\int_{\R^d}f_H(t,v)\d v, \qquad E_H(t)=\int_{\R^d}|v|^2 f_H(t,v)\d
v,$$
one obtains
\begin{equation}\label{haff}\begin{cases}n_H(t)&=\left(1+\left((d+\gamma)\mathbf{B}-\mathbf{A}\right)t\right)^{-\,\frac{d\mathbf{B}-\mathbf{A}}{(d+\gamma)\mathbf{B}-\mathbf{A}}}\\

 E_H(t)&=\frac{d}{2} \left(1+\left((d+\gamma)\mathbf{B}-\mathbf{A}\right)t\right)^{-\;\frac{(d+2)\mathbf{B}-\mathbf{A}}{(d+\gamma)\mathbf{B}-\mathbf{A}}} \qquad t \geq 0\end{cases}\end{equation}

\textit{\textbf{The main objective of the present work is to prove the existence of a self-similar profile  $\psi_H$ satisfying \eqref{tauT},
\eqref{init}.}} Notice that the existence of such a self-similar profile was taken for granted in several works in the physics community \cite{Maynar1,Maynar2,Trizac} but no rigorous justification was available up to now. Our work  aims to fill this blank, giving in turn the first rigorous mathematical ground justifying the analysis performed in the \textit{op. cit.}

\subsection{Notations} Let us introduce the notations we shall use
in the sequel. Throughout the paper we shall use the notation
$\langle \cdot \rangle = \sqrt{1+|\cdot|^2}$. We denote, for any
$\eta \in \R$, the Banach space
\[
     L^1_\eta(\R^d) = \left\{f: \R^d \to \R \hbox{ measurable} \, ; \; \;
     \| f \|_{L^1_\eta} := \int_{\R^d} | f (v) | \, \langle v \rangle^\eta \d v
     < + \infty \right\}.
\]

More generally we define the weighted Lebesgue space $L^p_\eta
(\R^d)$ ($p \in [1,+\infty)$, $\eta \in \R$) by the norm
$$\| f \|_{L^p_\eta} = \left[ \int_{\R^d} |f (v)|^p \, \langle v
        \rangle^{p\eta} \d  v \right]^{1/p} \qquad 1 \leq p < \infty$$
        while $\| f \|_{L^\infty_\eta} =\mathrm{ess-sup}_{v \in \R^d} |f(v)|\langle v\rangle^\eta$ for $p=\infty.$

We shall also use weighted Sobolev spaces $\mathbb{H}^s_\eta(\R^d)$ ($s \in \R$, $\eta \in \R$). When $s \in \N$, they  are defined by the norm
$$\| f \|_{\mathbb{H}^s_\eta} = \left( \sum_{|\ell|\leq s} \| \partial^\ell f \|^2_{L^2_\eta} \right)^{1/2}$$
where for $\ell\in\N^d$,  $\partial^\ell=\partial_{\xi_1}^{\ell_1} \ldots \partial_{\xi_d}^{\ell_d}$ and $|\ell|=\ell_1+\ldots +\ell_d$. Then, the definition is extended to real positive values of $s$ by interpolation.  For negative value of $s$, one can define $\mathbb{H}^{s}_{-\eta}(\R^{d})$ as the dual space of $\mathbb{H}^{-s}_{\eta}(\R^{d})$, i.e.
$$\|f\|_{\mathbb{H}^{s}_{\eta}}=\sup\left\{\left|\int_{\R^{d}}f(v)\,g(v)\d v\right|\,;\; \,\|g\|_{\mathbb{H}^{-s}_{-\eta}(\R^{d})}\leq 1\right\} \qquad \forall s < 0,\quad \eta \in\R.$$
We also define the space $\mathcal{C}([0,T],w-L^{1}(\R^{d}))$ of continuous functions from $[0,T]$ to the space $L^{1}(\R^{d})$ where the latter is endowed with its weak topology. 

\subsection{Strategy and main results} To prove the existence of a steady state $\psi_H$, solution to \eqref{tauT}, we shall use a dynamical approach as in
\cite{BagLau07,BisiCarrLods08,EMRR05,GPV04,MiMouh06}. It then amounts to
finding a steady state to  the \textit{annihilation equation}
\begin{equation}\label{BEscaled}
\partial_t \psi(t,\xi) + \mathbf{A}_\psi(t)\,\psi(t,\xi) + \mathbf{B}_\psi(t) \,\xi \cdot \nabla_\xi \psi(t,\xi)=\mathbb{B}(\psi,\psi)(t,\xi)
\end{equation}
supplemented with some nonnegative initial condition
\beq\label{CI}
\psi(0,\xi)=\psi_0(\xi),
\eeq
where $\psi_0$ satisfies
\beq\label{massenergie}
\int_{\R^d} \psi_0(\xi)\, \d\xi=1, \qquad
\int_{\R^d} \psi_0(\xi)\,|\xi|^2\,  \d\xi=\frac{d}{2},
\eeq
while
$$\mathbf{A}_\psi(t)=-\frac{\alpha}{2}
\int_{\R^d}\left(d+2-2|\xi|^2\right)\Q_-(\psi,\psi)(t,\xi)\d \xi,$$
and
$$\mathbf{B}_\psi(t)=-\frac{\alpha}{2d}\int_{\R^d}
\left(d-2|\xi|^2\right)\Q_-(\psi,\psi)(t,\xi)\d \xi.$$

Notice that \eqref{BEscaled} has to be seen only as a somewhat artificial generalization of \eqref{tauT}: we do not claim that \eqref{BEscaled} can be derived from \eqref{BE} nor that a solution $\psi$ to \eqref{BEscaled} is associated to a self-similar solution to \eqref{BE}. Again, the introduction of the new equation \eqref{BEscaled} is motivated only by the fact that any steady state of \eqref{BEscaled} is a solution to \eqref{tauT}.

We now describe the content of this paper. As explained above, the existence of the profile $\psi_H$ is obtained by finding a steady state to the annihilation equation \eqref{BEscaled}. As in previous works  \cite{BagLau07,BisiCarrLods08,EMRR05,GPV04,MiMouh06}, the proof relies on the application of a suitable version of Tykhonov fixed point theorem (we refer to \cite[Appendix A]{BagLau07} for a complete proof of it):
\begin{theo}[\textit{\textbf{Dynamic proof of stationary states}}]\label{GPV}
Let $\mathcal{Y}$ be a locally convex topological vector space and  $\mathcal{Z}$ a nonempty convex and  compact subset of\, ${\mathcal{Y}}$. If $(\mathcal{S}_t)_{t \ge 0}$ is  a 
continuous semi-group on $\mathcal{Z}$ such that $\mathcal{Z}$ is invariant under
the action of $\mathcal{S}_t$  (that is $\mathcal{S}_t z  \in\mathcal{Z}$ for any $z \in {\mathcal{Z}}$
and $t \ge 0$),
then there exists $z_0 \in \mathcal{Z}$ which is stationary under the action
of $\mathcal{S}_t$ (that is $\mathcal{S}_t z_0=z_0$ for any $t \ge 0$).
\end{theo}

In a more explicit way, our strategy is therefore to identify a topological vector space $\mathcal{Y}$ and a convex subset $\mathcal{Z} \subset \mathcal{Y}$ such that
\begin{enumerate}
\item for any $\psi_0 \in \mathcal{Z}$ there is a global solution $\psi \in   \mathcal{C}([0,\infty),\mathcal{Y })$ to  \eqref{BEscaled} that satisfies \eqref{CI};
\item the  solution $\psi$ is unique in $\mathcal{Y}$ and for any $\psi_0  \in \mathcal{ Z}$, one has  $\psi(t) \in \mathcal{Z}$ for any $t > 0$;
\item the set $\mathcal{Z}$ is compactly embedded into $\mathcal{Y}$;
\item solutions to \eqref{BEscaled} have to depend continuously on the initial datum.
\end{enumerate}

According to the above program, a crucial step in the above strategy is therefore to investigate the well-posedness of the Cauchy problem \eqref{BEscaled}-\eqref{CI} and next section is devoted to this point. The notion of solutions we consider here is as follows.

\begin{defi}\label{defi:sol} Given a nonnegative initial datum $\psi_0$ satisfying \eqref{massenergie}
and given $T >0$, a nonnegative function $\psi\::\:[0,T] \times \R^d \to \R$ is said to be a solution to the annihilation equation \eqref{BEscaled} if
$$\psi \in \mathcal{C}([0,T]\,;{ \,w-L^1(\R^d)})\cap L^\infty(0,T;L^1_{2}(\R^d)) \cap L^1(0,T;L^1_{2+\gamma}(\R^d))$$
 and satisfies \eqref{BEscaled} in the weak form:
\begin{multline}\label{weakformu}
\int_{\R^d} \psi(t,\xi)\varrho(\xi)\d\xi + \int_0^t \d s
\big[\mathbf{A}_\psi(s)-d \mathbf{B}_\psi(s)\big]\int_{\R^d}  \varrho(\xi)\,
\psi(s,\xi)\,\d\xi \\
= \int_0^t \d s \mathbf{B}_\psi(s) \int_{\R^d}
\psi(s,\xi)\, \xi \cdot \nabla_\xi \varrho(\xi) \d\xi
+ \int_{\R^d} \varrho(\xi)\psi_0(\xi) \d \xi
+ \int_0^t \d s\int_{\R^d} \mathbb{B}(\psi,\psi)(s,\xi)\varrho(\xi)\d\xi
\end{multline}
for any $\varrho \in \mathcal{C}^1_c(\R^d)$ and any $t \in (0,T).$
\end{defi}
Notice that the assumption $\psi \in L^1(0,T;L^1_{2+\gamma}(\R^d))$ is needed
in order to both the quantities $\mathbf{A}_\psi(t)$ and $\mathbf{B}_\psi(t)$
to be well defined.

\bigskip
Let us point out the similarities and the differences between \eqref{BEscaled} and the well-known Boltzmann equation. First, it follows from the definition of the coefficients $\mathbf{A}_\psi$ and $\mathbf{B}_\psi$ that the mass and the energy of solutions to \eqref{BEscaled} are conserved. However, there is no reason for the momentum to be preserved. Even if we assume that the initial datum has vanishing momentum
we are unable to prove that this propagates with time. It is also not clear whether there exists an entropy for \eqref{BEscaled}. Let us note on the other hand that since the coefficients $\mathbf{A}_\psi$ and $\mathbf{B}_\psi$ involve moments of order $2+\gamma$ of $\psi$, a crucial step will be to prove, via suitable \emph{a priori} estimates, that high-order moments of solutions are uniformly bounded, ensuring a good control of both $\mathbf{A}_\psi$ and $\mathbf{B}_\psi$. At different stages of this paper, this lack of {\it a priori} estimates and this necessary control of $\mathbf{A}_\psi$ and $\mathbf{B}_\psi$ complicate the analysis with respect to the Boltzmann equation. It also leads us to formulate some assumptions, some of which we hope to be able to get rid of in a future work. Let us now describe precisely what are the practical consequences of the aforementioned differences.
Since we are interested in the physically relevant model of hard spheres interactions, the cross section involved in the collision operator is unbounded. Consequently, the existence of a solution to \eqref{BEscaled} is obtained by applying a fixed point argument to a truncated equation and then passing to the limit. Such an approach is reminiscent from the well-posedness theory of the Boltzmann equation \cite{MiWe99} and relies on suitable \emph{a priori} estimates and stability result. In particular, such a stability result allows to prove in a unique step the above points (1) and (4) of the above program. We thereby prove the following theorem in Section \ref{sec:cauchy}.
\begin{theo}\label{well_posedness}
Let $p>1$. Let $\psi_0\in L^1_{2+\gamma}(\R^d)\cap L^p(\R^d)$
be a nonnegative distribution function satisfying \eqref{massenergie}.
Then,  there exists a nonnegative solution 
$$\psi
\in\C([0,\infty);{w-L^1}(\R^d)) \cap L^1_{\mathrm{loc}}((0,\infty),L^1_{2+2\gamma}(\R^d))\cap L^\infty_{\mathrm{loc}}((0,\infty),L^1_{2+\gamma}(\R^d))$$
to \eqref{BEscaled} such that $\psi(0,\cdot)=\psi_0$ and
$$\int_{ \R^d} \psi(t,\xi) \, \d\xi= 1, \qquad
\int_{ \R^d} \psi(t,\xi) \,|\xi|^2 \, \d\xi= \frac{d}{2} \qquad \forall t \geq 0.$$
Furthermore, if we  assume that $p=2$ and that $\psi_0$ also satisfies 
\beq\label{hypini}\psi_0 \in L^1_{9+d+\gamma+2\kappa}(\R^d) \cap  L^2_{\frac{9+d}{2}+ \kappa}(\R^d)\cap \mathbb{H}^1_{3+\frac{d+\gamma+\kappa}{2}}(\R^d)
\eeq
for some $\kappa>0$, such a solution is unique. 
\end{theo}
Notice that, with respect to classical existence results on Boltzmann equation (see e.g. \cite{MiWe99}), we need here to impose an additional $L^p$-integrability condition on the initial datum $\psi_0$. Such an assumption is needed in order to control the nonlinear drift term in \eqref{BEscaled} and especially to get bounds on the moments of order
$2+\gamma$ arising in the definition of $\mathbf{A}_\psi(t)$ and $\mathbf{B}_\psi(t)$, these bounds need to be uniform with respect to the truncation.
Moreover, as far as the uniqueness is concerned, we need additional regularity assumptions of Sobolev type in order to control the drift term in the equation satisfied by the difference of two solutions.

The previous result allows to identify the space $\mathcal{Y}$  in the above Theorem \ref{GPV} as $\mathcal{Y}=L^{1}(\R^{d})$ endowed with its weak topology and gives the existence of a semi-group for \eqref{BEscaled} and the next step is to finding a subset $\mathcal{Z}$ which is left invariant under the action of this semi-group and is a compact subset of $\mathcal{Y}$. Since $\mathcal{Z}$ has to be a weakly compact subset of $L^{1}(\R^{d})$, it is natural in view of Dunford-Pettis criterion to look for a subspace involving \textit{higher-order moments} of the solution $\psi(t)$
together with additional \textit{integrability conditions}. We are therefore first lead to prove uniform in time moment estimates for the solution $\psi(t)$. More precisely, the main result of Section \ref{sec_mom} is the following
\begin{theo}\label{theomom}
Let $p>1$. Let $\psi_0\in L^1_{2+\gamma}(\R^d)\cap L^p(\R^d)$ be a nonnegative distribution function satisfying \eqref{massenergie}. Let then $\psi
\in\C([0,\infty); w-L^1(\R^d)) \cap L^\infty_{\mathrm{loc}}((0,\infty),L^1_{2+\gamma}(\R^d))$
be a nonnegative solution to \eqref{BEscaled}-\eqref{CI}. Then, there exists $\alpha_0\in(0,1]$ such that for $0<\alpha<\alpha_0$, the solution $\psi$ satisfies
$$\sup_{t \geq 0} \int_{\R^d}\psi(t,\xi)\,|\xi|^{2+\gamma}\d\xi \leq \max\left\{\int_{\R^d}\psi_0(\xi)\,|\xi|^{2+\gamma}\d\xi,\overline{M}\right\},$$
for some explicit constant $\overline{M}$ depending only on $\alpha$, $\gamma$, $b(\cdot)$ and $d$.
\end{theo}

\begin{rmq} The parameter $\alpha_0$ appearing in the above theorem  is fully explicit. In the particular case of true hard spheres in dimension $d=3$, i.e. for constant collision kernel $b(\cdot)=1/4\pi$ and $\gamma=1$, one has  $\alpha_0=\frac{2}{7}.$ We refer to Proposition \ref{theoMom} \& Remark \ref{rmqalpha0} for more details.
\end{rmq}
The proof of the above result relies on a careful study of the moment system associated to the solution $\psi(t)$ to \eqref{BEscaled}-\eqref{CI}. Since we are dealing with hard spheres interactions, such a system is not closed but a sharp version of \textit{Povzner-type inequalities} allows to control higher-order moments in terms of lower-order ones.  The restriction on the parameter $\alpha \in (0,\alpha_0)$ arises naturally in the proof of the uniform in time bound of the moment of order $2+\g$ (see Proposition \ref{theoMom}).

At the end of Section \ref{sec_mom} we establish a lower bound for $L(\psi)$ where $L$ denotes the operator in the definition of $\Q_-$, namely
\beq\label{minL}L(\psi)(t,\xi)=\int_{\R^d} \psi(t,\xi_*)\,|\xi-\xi_*|^\gamma \,\d\xi_*\geq \mu_\alpha \langle \xi \rangle ^\gamma, \qquad \forall \xi \in \R^d,\:\:t \geq 0,
\eeq
for some positive constant $\mu_\alpha >0$ depending on $\gamma,d,\alpha$, $b(\cdot)$ and on $\int_{\R^d} \psi_0(\xi)|\xi|^\gamma\d\xi$. Note that this bound will be essential in Section
\ref{sec_LP} and that we need here to assume that $\psi_0$ is an \textit{\textbf{isotropic}} function.  Isotropy is indeed propagated by \eqref{BEscaled}. For the Boltzmann equation, this assumption is useless since such a bound may be obtained thanks to the entropy for elastic collisions (see \cite[Proposition 2.3]{MouhVill04}) or thanks to the Jensen inequality and vanishing momentum for inelastic collisions and $\gamma=1$ (see \cite[Eq. (2.7)]{MiMo3}). This naturally leads us to Section \ref{sec_LP} where we deal with propagation of higher-order Lebesgue norms and where
we obtain the following:
\begin{theo}\label{theoLp}
Let $\psi_0\in L^1_{2+\gamma}(\R^d)$ be a nonnegative distribution function satisfying \eqref{massenergie}. We assume furthermore that $\psi_0$ is an isotropic function, that is
\begin{equation}\label{HYP}\psi_0(\xi)=\overline{\psi_0}(|\xi|) \qquad \forall \xi \in \R^d.\end{equation}
Then, there is some explicit $\overline{\alpha} \in (0,1]$ such that, for $0 < \alpha < \overline{\alpha}$  there exists some explicit $p^\star_\a \in (1,\infty]$ such that, for any $p  \in (1,p_\a^\star)$,
$$\psi_0 \in L^p(\R^d) \implies \sup_{t \geq 0}\|\psi(t)\|_{L^p} \leq  \max\left\{\|\psi_0\|_{L^p},C_p(\psi_0)\right\}$$
for some explicit constant $C_p(\psi_0)>0$ depending only on  $\alpha$, $\gamma$, $b(\cdot)$, $p$, the dimension $d$ and $\int_{\R^d}\psi_0(\xi)|\xi|^\gamma\d\xi.$ Here above,  $\psi \in\C([0,\infty);{w-L^1(\R^d)}) \cap L^\infty_{\mathrm{loc}}((0,\infty),L^1_{2+\gamma}(\R^d))$
is a nonnegative solution to \eqref{BEscaled}-\eqref{CI}.
\end{theo}
\begin{rmq}  Just as in Theorem \ref{theomom}, the parameter $\overline{\alpha}$ is explicit: for true hard spheres in dimension $d=3$ one has $\overline{\alpha}=\frac{1}{4}.$ In this case, the parameter $p^\star_\a=\frac{3\alpha}{5\alpha-1}$ if $1/5 < \alpha < \overline{\alpha}$ while $p^\star=\infty$ if $\alpha \leq 1/5.$ See Remarks \ref{rmqastar}, \ref{rmqabar} \& \ref{rmq:Cpunif} for details.
\end{rmq}
The proof of the above result comes from a careful study of the equation for higher-order Lebesgue norms of the solution $\psi(t)$ combined with the above bound \eqref{minL} where we only consider \textit{\textbf{isotropic}} initial datum. Here again, one notices a restriction on the parameter $\alpha \in (0,\overline{\alpha})$ for the conclusion to hold. The fact that the constant $C_p(\psi_0)$ depends on the initial datum $\psi_0$ through (the inverse of) its moment $\int_{\R^d}\psi_0(\xi)|\xi|^\gamma\d\xi$ is no major restriction since we will be able to prove the propagation of lower bound for such a moment along the solution to \eqref{BEscaled} (see Sections \ref{sec_mom} and \ref{sec_LP} for details).

It remains now to show that weighted Sobolev bounds also propagate uniformly with time. We are able to do it only for physical case of hard-spheres interactions, i.e. whenever $\gamma=1$. 

\begin{theo}\label{theo:sob} Assume $\gamma=1$.  
Let $\psi_0\in L^1_{2+\gamma}(\R^d)$ be a nonnegative function satisfying 
\eqref{HYP} and 
$$ \psi_0 \in L^1_{\mathfrak{q}(\kappa)}(\R^d) \cap  L^2_{\frac{9+d}{2}+ \kappa}(\R^d)\cap \mathbb{H}^1_{\frac{d+7+\kappa}{2}}(\R^d)
$$
for some $\kappa>0$ where $\mathfrak{q}(\kappa)=\max\left\{\frac{9+d(d-2)}{2}+\kappa, 10+d+2\kappa\right\}.$
 Let 
 $$\psi \in{\C([0,\infty);w-L^1(\R^d)) \cap L^\infty_{\mathrm{loc}}((0,\infty),L^1_{2+\gamma}(\R^d))}$$
denote the nonnegative solution to \eqref{BEscaled}-\eqref{CI}. Then, there is some explicit $\alpha_1 \in (0,\min\{\alpha_0,\overline{\alpha}\}]$ such that, for $0 < \alpha < \alpha_1$ 
\beq\label{L2normunif}
\sup_{t \geq 0} \|\psi(t)\|_{L^2_{\frac{9+d}{2}+\kappa}}\leq \max\left\{ \|\psi_0\|_{L^2_{\frac{9+d}{2}+\kappa}}, C_{2,\kappa}(\psi_0) \right\}
\eeq
and 
\beq\label{H1normunif}
\sup_{t\geq 0} \|\nabla\psi(t)\|_{L^2_{\frac{d+7+\kappa}{2}}} \leq \max\left\{\|\nabla \psi_0\|_{L^2_{\frac{d+7+\kappa}{2}}}, C_{\mathrm{Sob}}(\psi_0)\right\},
\eeq
where $C_{2,\kappa}(\psi_0)$ depends on the bound in Theorem \ref{theomom}, 
$\|\psi_0\|_{L^1_{\mathfrak{q}(\kappa)}}$ and the bound in Theorem \ref{theoLp} (with $p=2$) whereas $C_{\mathrm{Sob}}(\psi_0)$ depends on  the same bounds and on the one in \eqref{L2normunif}
\end{theo}

Section \ref{sec_sob} is devoted to the proof of the above Theorem. 
The proofs of \eqref{L2normunif} and \eqref{H1normunif} rely on \eqref{minL} and on some well-known regularity properties of $\Q_+$. 
 We wish to emphasise the fact that the restriction $\gamma=1$ is coming from the propagation of (weighted) Sobolev norms. Namely, while local in time propagation of weighted Sobolev norms is true for any $\gamma \in (0,1]$ (see  Lemma \ref{lem:sob}), we are able to prove \emph{uniform in time} estimates only for $\gamma=1$. Notice that, for $\gamma \in (0,1)$, the main obstacle is coming from the \emph{loss term} $\Q_{-}.$

Combining the four above results with  Theorem \ref{GPV}  we obtain our main result, proven in Section \ref{sec_autosim}:
\begin{theo}\label{existence} Assume $\gamma=1$. For any $\alpha \in (0, \alpha_1)$ there exists a radially symmetric nonnegative $\psi_H \in L^1_{3}(\R^d) \cap L^2(\R^d)$ satisfying \eqref{tauT} and \eqref{init}.
\end{theo}
The proof of the above result is rather straightforward in view of the previously obtained results.

Open problems and perspectives are addressed in Section \ref{sec:discuss}.
As previously mentioned, one of them consists in showing that solutions to \eqref{BE} approach for large times a self-similar solution $f_H$ to \eqref{BE} of the form \eqref{scalingfpsi}. The first step was the existence of the profile $\psi_H$, which has been obtained in Section \ref{sec_autosim}. Besides, one is also interested in the well-posedness of \eqref{BE} and, following the same arguments as in the proof of Theorem \ref{well_posedness} the existence of a solution to \eqref{BE} may be easily  obtained. More precisely, we have

\begin{theo}\label{exi_BE}
Let $f_0\in L^1_{2+\gamma}(\R^d)$ be a nonnegative distribution function. Then,  there exists a unique nonnegative solution $f\in\C([0,\infty);L_2^1(\R^d)) \cap L^1_{\mathrm{loc}}((0,\infty),L^1_{2+\gamma}(\R^d))$ to \eqref{BE} such that $f(0,\cdot)=f_0$ and
\beq\label{ineg}\int_{ \R^d} f(t,v) \, \d v\leq \int_{ \R^d} f_0(v) \, \d v, \qquad
\int_{ \R^d} f(t,v) \,|v|^2 \, \d v\leq \int_{ \R^d} f_0(v ) \,|v|^2 \, \d v  \qquad \forall t \geq 0.\eeq
\end{theo}

We give the main lines for the proof of this Theorem in Appendix \ref{appendixA}. Finally, the particular case of Maxwellian molecules is discussed in the Appendix \ref{appendixB}.

\section{Cauchy problem}\label{sec:cauchy}

This section is devoted to the proof of Theorem \ref{well_posedness}. To this aim, we first consider a truncated equation.

\subsection{Truncated equation}

In this section, we only assume that $\psi_0 \in W^{1,\infty}(\R^d)\cap
L^1_{2+\delta}(\R^d)$ (for some $\delta > 0$) is a \textit{fixed} nonnegative distribution function that does not necessarily satisfy the above \eqref{massenergie}  and we truncate the
collision kernel $\B$. Thereby, for $n\in\N$, we consider here the
well-posedness of the following equation
\beq\label{annihi}
\partial_t \psi (t,\xi)+ \mathbf{A}^n_\psi(t) \, \psi(t,\xi)
+\mathbf{B}^n_\psi(t)\, \xi\cdot\nabla\psi(t,\xi)
= \mathbb{B}^n(\psi,\psi)(t,\xi),
\eeq
where the collision operator $\mathbb{B}^n(\psi,\psi)$ is given by
\beq\label{Bn}
\mathbb{B}^n(\psi,\psi)=(1-\alpha)\Q_+^n(\psi,\psi)-\Q_-^n(\psi,\psi),
\eeq
for which the collision operator $\Q^n$ is defined as above with a collision kernel $\B_n$ given by
$$\B_n(\xi-\xi_*,\s)=\Phi_n(|\xi-\xi_*|)b_n(\cos \theta)$$
with
$$b_n(x)=\mathbf{1}_{\{|x|\leq 1-1/n\}}b(x)  \qquad \text{ and } \quad \Phi_n(r)=\left(\min\left\{r,n\right\}\right)^\gamma, \quad
 \gamma\in(0,1].$$
Finally,
$$\mathbf{A}^n_\psi(t):=-\frac{\alpha}{2} \int_{\R^d}
\left(\frac{d+2}{\int_{\R^d} \psi(0,\xi_*)\, \d\xi_*}
-\frac{d\,|\xi|^2}{\int_{\R^d} \psi(0,\xi_*)\, |\xi_*|^2\, \d\xi_*}\right)
\Q^n_-(\psi ,\psi )(t,\xi)\d \xi$$
and
$$\mathbf{B}^n_\psi(t) :=-\frac{\alpha}{2}\int_{\R^d}
\left(\frac{1}{\int_{\R^d} \psi(0,\xi_*)\, \d\xi_*}
-\frac{|\xi|^2}{\int_{\R^d} \psi(0,\xi_*)\, |\xi_*|^2\, \d\xi_*}\right)
\Q^n_-(\psi,\psi )(t,\xi)\d \xi.$$
We notice here that the definitions of  $\mathbf{A}^n_\psi(t)$ and
$\mathbf{B}_\psi^n(t)$ match the definitions of $\mathbf{A}_\psi(t)$ and
$\mathbf{B}_\psi(t)$ given in the introduction with $\Q^n_-$ replacing $\Q_-$
when $\psi_0$ is assumed to satisfy \eqref{massenergie}. The main result of this section is the following well-posedness theorem:
 \begin{theo}\label{cauchypb} Let $\delta>0$ and let $\psi_0 \in
   W^{1,\infty}(\R^d)\cap L^1_{2+\delta}(\R^d)$ be a nonnegative distribution
   function. Then, for any $n \geq 1$,  there exists a nonnegative solution $\psi=\psi_n
   \in\C([0,\infty);L^1(\R^d))$ to the  truncated problem  \eqref{annihi} such
   that $\psi_n(0,\cdot)=\psi_0$ and
$$\int_{ \R^d} \psi_n(t,\xi) \, \d\xi= \int_{ \R^d} \psi_0(\xi) \, \d\xi, \qquad
\int_{ \R^d} \psi_n(t,\xi) \,|\xi|^2 \, \d\xi= \int_{ \R^d} \psi_0(\xi)
\,|\xi|^2 \, \d\xi\qquad \forall t \geq 0.$$
\end{theo}

The proof of this  result follows classical paths already
employed for the classical space homogeneous Boltzmann equation but is made
much more technical because of the contribution of some nonlinear drift-term.
Let $T>0$ and
$$h\in \C([0,T]; L^1(\R^d))\cap L^\infty((0,T); L^1(\R^d,|\xi|^{2+\delta}\, \d\xi))$$
be fixed.
We consider the auxiliary equation:
\begin{equation}\label{annihi_lin}
\begin{cases}
\partial_t \psi (t,\xi)&+ \mathbf{A}^n_h(t) \, \psi(t,\xi)
+\mathbf{B}^n_h(t)\, \xi\cdot\nabla_\xi\psi(t,\xi) +\LLn(t,\xi)\, \psi(t,\xi)\\
& \phantom{++++++++++++++}= (1-\alpha)\, \Q^n_+(h,h)(t,\xi),\\
\psi(0,\xi)&=\psi_0(\xi).\end{cases}\end{equation}
Here, $\mathbf{A}_h^n$ and  $\mathbf{B}^n_h$ are defined as
$\mathbf{A}_\psi^n$ and  $\mathbf{B}^n_\psi$ with $\Q^n_-(h,h)$ replacing
$\Q^n_-(\psi,\psi)$ and
$$\LLn(t,\xi):=\int_{\R^d \times \S^{d-1}} \B_n(\xi-\xi_*,\s)\,
h(t,\xi_*)\, \d\xi_*\, \d\sigma=\|b_n\|_{\L} \int_{\R^d} \Phi_n(|\xi-\xi_*|)\, h(t,\xi_*)
\, \d\xi_*.$$
We solve this equation using the characteristic method:
notice that, by assumption on $h$, the mapping $t \mapsto \mathbf{B}_h^n(t)$
is continuous on $[0,T]$ and, for any $\xi\in\R^d$, the characteristic
equation
\begin{equation}\label{cara}\dfrac{\d}{\d t}X(t;s,\xi)=\mathbf{B}^n_h(t)\, X(t;s,\xi),\qquad \qquad X(s;s,\xi)=\xi,\end{equation}
gets a unique global solution given by
$$X_h(t;s,\xi)=\xi \exp\left(\int_s^t \mathbf{B}_h^n(\tau)\, \d\tau\right).$$
Then, the Cauchy problem \eqref{annihi_lin} admits a unique solution given by
\begin{multline}\label{sol}
\psi (t,\xi)=\psi^1(t,\xi)+\psi^2(t,\xi)
= \psi_0\left(X_h(0;t,\xi)\right)
\exp\left(-\int_0^t \left[\mathbf{A}_h^n(\tau)
+\LLn\left(\tau,X_h(\tau;t,\xi)\right)\right]\, \d\tau\right)
\\
 +  (1-\alpha) \int_0^t \exp\left(-\int_s^t \left[\mathbf{A}_h^n(\tau)
+\LLn\left(\tau,X_h(\tau;t,\xi)\right)\right]\, \d\tau\right)
\, \Q_+^n(h,h)\left(s,X_h(s;t,\xi)\right)\, \d s.
\end{multline}

For any $T >0$ and any $M_1,M_2,\ell, C_\delta >0$ (to be fixed later on), we define $\H=\H_{T,M_1,M_2,\ell, C_\delta}$ as the set of all nonnegative
$h\in\C([0,T];L^1(\R^d))$ such that
$$\sup_{t\in[0,T]} \int_{\R^d} h(t,\xi)\, \d\xi\leq M_1, \qquad
\sup_{t\in[0,T]} \int_{\R^d} h(t,\xi)\,|\xi|^2\,  \d\xi \leq M_2,$$
and
$$ \sup_{t\in[0,T]} \int_{\R^d} h(t,\xi)\,|\xi|^{2+\delta}\,  \d\xi \leq C_\delta,
\qquad  \sup_{t\in[0,T]} \| h(t)\|_{W^{1,\infty}} \leq \ell.$$
Define then the mapping $$\T\::\:\H\longrightarrow \C([0,T];L^1(\R^d))$$
 which, to any $h \in \H$, associates the solution $\psi=\T(h)$ to \eqref{annihi_lin} given by \eqref{sol} (notice that, clearly, it would be more correct to write $\T_n$ instead of $\T$ since $n$ has been fixed). We look for parameters $T, M_1, M_2, C_\delta$ and $\ell$ that ensure $\T$ to map $\H$ into itself. To do so, we shall use the following lemma whose proof is omitted and relies only on the very simple estimate:
$$\Q_-^n(h,h)(t,\xi)=h(t,\xi)\LLn(t,\xi) \leq \left(n^\gamma M_1 \|b_n\|_{\L}\right) h(t,\xi) \qquad \forall t \in [0,T]$$
valid for any $h \in \H$.
\begin{lem}\label{lem1} Define, for any $n \in \mathbb{N}$ and any $M_1 >0$,
$$\mu_n=\mu_n(M_1)=\frac{\alpha}{\|\psi_0\|_{L^1}}\; n^\gamma M_1  \|b_n\|_{\L}
\quad \mbox{ and } \quad \nu_n=\nu_n(M_1)=\frac{\alpha\,n^\gamma M_1  \|b_n\|_{\L}}{\int_{ \R^d} \psi_0(\xi)\,|\xi|^2 \, \d\xi}.$$
For any fixed $h \in \H$ and any $(t,\xi)\in[0,T]\times \R^d$ the following hold
\begin{enumerate}[(i)]
\item $0\leq d\mathbf{B}_h^n(t)-\mathbf{A}_h^n(t)=\frac{\alpha}{\|\psi_0\|_{L^1}}
  \displaystyle \int_{\R^d} \Q_-^n(h,h)(t,\xi)\, \d \xi \leq \mu_n M_1.$
\item $-\frac{\mu_n}{2}M_1  \leq  \mathbf{B}_h^n(t)
\leq \frac{\nu_n}{2}M_2.$
\item  $-\frac{\mu_n(d+2)}{2}M_1\leq  \mathbf{A}_h^n(t).$
\item $0 \leq (d+2)\mathbf{B}_h^n(t)-\mathbf{A}_h^n(t)=\dfrac{\alpha}{\int_{ \R^d} \psi_0(\xi)\,|\xi|^2 \, \d\xi} \displaystyle \int_{\R^d} |\xi|^2\Q_-^n(h,h)(t,\xi)\, \d \xi \leq
\nu_n M_2.$
\end{enumerate}

\end{lem}

\paragraph{\textit{\textbf{Control of the density.}}}   By  a simple change of variables, one checks easily that  the solution $\psi(t,\xi)$ given by \eqref{sol} fulfills
\begin{multline*}
\int_{\R^d}\psi(t,\xi)\, \d\xi =  \int_{\R^d}\psi_0(\xi) \,
\exp\left(\int_0^t \left[ d\,\mathbf{B}_h^n(\tau)-\mathbf{A}_h^n(\tau)
-\LLn\left(\tau, X_h(\tau;0,\xi)\right)\right]\, \d\tau\right)\,
\d\xi  \\
 + (1-\alpha) \int_0^t \d s\int_{\R^d} \exp\left(\int_s^t \left[ d\,\mathbf{B}_h^n(\tau)-\mathbf{A}_h^n(\tau)
-\LLn\left(\tau,X_h(\tau,s,\xi)\right)\right]\, \d\tau\right)\,
\Q_+^n(h,h)(s,\xi) \, \d\xi.
\end{multline*}
It comes then from the above Lemma \ref{lem1} 
that
\bean
\int_{\R^d}\psi(t,\xi)\, \d\xi & \leq &
\|\psi_0\|_{L^1} \,\exp\left(t\, \mu_n\,M_1\right)
+ (1-\alpha) \int_0^t \exp\left((t-s)\,\mu_n\,M_1\right)\d s
\int_{\R^d} \Q_+^n(h,h)(s,\xi) \, \d\xi , \\
& \leq &  \|\psi_0\|_{L^1} \, \exp\left(t\mu_n\,M_1\right)
+ \frac{1-\alpha}{\alpha}\mu_n\,M_1\|\psi_0\|_{L^1}
\int_0^t  \exp\left((t-s)\,\mu_n\,M_1\right) \, \d s,\eean
where we also used that $ \int_{\R^{d}}\Q^{n}_{+}(h,h)(s,\xi)\d\xi=\int_{\R^{d}}\Q^{n}_{-}(h,h)(s,\xi)\d\xi$. We deduce from this that
\begin{equation}\label{density}\sup_{t \in [0,T]} \int_{\R^d}\psi(t,\xi)\,\d
  \xi \leq \|\psi_0\|_{L^1}\,\left(  \exp\left(T\,\mu_n\,M_1\right)
+ \frac{1-\alpha}{\alpha}\left(\exp\left(T\,\mu_n\,M_1\right)-1\right)\right) \qquad \forall h \in \H.\end{equation}

\paragraph{\textit{\textbf{Control of the moments.}}} We now focus on the control of moments of order $r$ with $r \geq 2$ to the solution $\psi$ given by \eqref{sol}. Arguing as above,
\begin{multline*}
\int_{\R^d}\psi(t,\xi)\, |\xi|^r\, \d\xi  =  \int_{\R^d}\psi_0(\xi) \,
|\xi|^r\, \exp\left(\int_0^t \left[ (r+d)\,\mathbf{B}_h^n(\tau)-\mathbf{A}_h^n(\tau)
-\LLn\left(\tau,X_h(\tau,0,\xi)\right)\right]\, \d\tau\right)\,
\d\xi  \\
+ (1-\alpha) \int_0^t \, \d s\int_{\R^d} \exp\left(\int_s^t \left[ (r+d)\,\mathbf{B}_h^n(\tau)
    -\mathbf{A}_h^n(\tau)  -\LLn\left(\tau,X_h(\tau,s,\xi)\right)
  \right]\, \d\tau\right)\\\, \Q_+^n(h,h)(s,\xi) \, |\xi|^r\, \d\xi.
\end{multline*}
Using again Lemma \ref{lem1}, we get
\begin{multline*}
\int_{\R^d}\psi(t,\xi)\, |\xi|^r\, \d\xi  \leq  \exp\left(t\,(\mu_n\,
  M_1+\frac{\nu_n \, r}{2} M_2)\right)
\int_{\R^d}\psi_0(\xi) \, |\xi|^r\, \d\xi  \\
+(1-\alpha) \int_0^t  \exp\left((t-s)\,(\mu_n\, M_1+\frac{\nu_n\,r}{2} M_2)\right)
\int_{\R^d} \, \Q_+^n(h,h)(s,\xi) \, |\xi|^r\, \d\xi\, \d s.
\end{multline*}
Now, the change of variables $(\xi,\xi_*) \to (\xi',\xi_*')$ together with the fact that $|\xi'|\leq |\xi|+|\xi_*|$, yields
\bean
\int_{\R^d} \, \Q_+^n(h,h)(s,\xi) \, |\xi|^r\, \d\xi & \leq &
\int_{\R^d\times \R^d} \int_{\S^{d-1}} \mathcal{B}_n(\xi-\xi_*,\s) \,h(s,\xi)\, h(s,\xi_*)\, |\xi'|^r\, \d\s \, \d\xi \, \d\xi_* \\
& \leq & 2^{r-1} \, n^\g \,\|b_n\|_{\L} \int_{\R^d\times \R^d} h(s,\xi)\,
h(s,\xi_*)\, ( |\xi|^r+ |\xi_*|^r)\, \d\xi \, \d\xi_* \\
& \leq & 2^r \, n^\g \,\|b_n\|_{\L} \, M_1 \int_{\R^d} h(s,\xi)\,
 |\xi|^r \, \d\xi.
\eean
Hence,
\begin{multline*}
\int_{\R^d}\psi(t,\xi)\, |\xi|^r\, \d\xi \leq  \exp\left(t\,(\mu_n\,
  M_1+\frac{\nu_n\, r}{2} M_2)\right)
\int_{\R^d}\psi_0(\xi) \, |\xi|^r\, \d\xi  \\
+ (1-\alpha)2^r \, \frac{\mu_n}{\alpha} \;\|\psi_0\|_{L^1} \int_0^t
\exp\left((t-s)(\mu_n\, M_1+\frac{\nu_n\, r}{2} M_2)\right)\, \d s
\int_{\R^d} h(s,\xi)\, |\xi|^r \, \d\xi .
\end{multline*}
In particular, choosing successively $r=2$ and $r=2+\delta$ one gets that
\begin{multline}\label{r=2}
\sup_{t \in [0,T]}\int_{\R^d}\psi(t,\xi)\, |\xi|^2\, \d\xi \leq
\exp\big(T\,(\mu_n\, M_1+ \nu_n\, M_2)\big) \int_{ \R^d} \psi_0(\xi)\,|\xi|^2
\, \d\xi \\
 + 4 \,\|\psi_0\|_{L^1}\;\frac{1-\alpha}{\alpha}\dfrac{\mu_n \,M_2}{\mu_n\,
   M_1+\nu_n\, M_2}\left(\exp\big(T\,(\mu_n\,M_1+\nu_n\, M_2)\big)-1\right)
\end{multline}
and
\begin{multline}\label{r=2+d}
\sup_{t \in [0,T]}\int_{\R^d}\psi(t,\xi)\, |\xi|^{2+\delta}\, \d\xi \leq
\exp\big(T\,(\mu_n\, M_1+\frac{2+\delta}{2}\; \nu_n\,  M_2)\big)
\int_{\R^d}\psi_0(\xi)\, |\xi|^{2+\delta}\, \d\xi\\
 + \|\psi_0\|_{L^1} \;\frac{1-\alpha}{\alpha}\dfrac{C_\delta\,2^{2+\delta}\, \mu_n
 }{\mu_n\, M_1+\frac{2+\delta}{2} \; \nu_n\,M_2}\left(\exp\big(T\,(\mu_n\,
   M_1+\frac{2+\delta}{2}\; \nu_n\,  M_2)\big)-1\right)\end{multline}
for any $h \in \H$.\\

\paragraph{\textit{\textbf{Control of the $W^{1,\infty}$ norm.}}} Our assumption on the collision kernel of the operator $\Q^n$ allows us to apply
\cite[Theorem 2.1]{MouhVill04} with $k=\eta=0$ and $\sin^2(\t_b/2)=1/(2n)$ to get directly
$$\|\Q^n_+(h,h)\|_{L^\infty} \leq 2 \, n^{1+\g} \, \|b_n\|_{\L}\,\|h\|_{L^1}
\, \|h\|_{L^\infty}.$$
Then, the change of variable $\s\to -\s$ yields
$$\nabla \Q^n_+(h,h)=\Q^n_+(\nabla h,h)+\Q^n_+(h,\nabla h)=2\, \Q^n_+(h,\nabla h)$$
and, applying again \cite[Theorem 2.1]{MouhVill04}:
$$\|\nabla \Q^n_+(h,h)\|_{L^\infty} \leq 2 \|\Q^n_+(h,\nabla h)\|_{L^\infty}
\leq 4 \, n^{1+\g} \, \|b_n\|_{\L}\,\|h\|_{L^1} \, \|\nabla h\|_{L^\infty}.$$
Consequently
$$\|\Q^n_+(h,h)\|_{W^{1,\infty}} \leq 4\, n^{1+\g} \, \|b_n\|_{\L}\,
\|h\|_{L^1} \, \|h\|_{W^{1,\infty}}.$$
In the same way, since $\frac{\d}{\d r}\Phi_n(r)  \leq \gamma n^{\gamma-1} \leq 1$, one checks easily that
$$\|\LLn(t,\cdot)\|_{W^{1,\infty}} \leq 2n^\gamma
\|b_n\|_{\L}\|h(t)\|_{L^1}  \leq 2\,\frac{\mu_n}{\alpha}\, \|\psi_0\|_{L^1} \qquad  \forall t \in [0,T], h \in \H.$$
Recall now the expression of the solution $\psi=\psi^1+\psi^2$ given in
\eqref{sol}. It is easy to see that, for any $t \in [0,T]$
\begin{multline*}
 \|\psi^1(t)\|_{W^{1,\infty}} \leq \exp\left(-\int_0^t \mathbf{A}_h^n(\tau)\d\tau\right)\|\psi_0\|_{L^\infty}+  \exp\left(-\int_0^t (\mathbf{A}_h^n(\tau)+\mathbf{B}_h^n(\tau))\d\tau\right)\|\nabla_\xi \psi_0\|_{L^\infty} \\
+\|\psi_0\|_{L^\infty}  \exp\left(-\int_0^t \mathbf{A}_h^n(\tau)\d\tau\right)\int_0^t\exp\left(-\int_\tau^t \mathbf{B}_h^n(s)\d s\right)\|\nabla_\xi \LLn(\tau,\cdot)\|_{L^\infty}\d \tau
\end{multline*}
so that, using again Lemma \ref{lem1}:
\begin{equation*}\begin{split}
\|\psi^1(t)\|_{W^{1,\infty}} &\leq \exp\left(\frac{\mu_n(d+3)}{2}M_1\,t\right)\|\psi_0\|_{W^{1,\infty}}  \\
&\phantom{+++++}+\frac{2}{\alpha}\mu_n  \|\psi_0\|_{L^1}\, \|\psi_0\|_{L^\infty}  \exp\left( \frac{\mu_n(d+2)}{2}M_1 t\right) \int_0^t \exp\left(\frac{\mu_n}{2}M_1(t-\tau)\right)\d \tau \\
\end{split}\end{equation*}
i.e.
$$\|\psi^1(t)\|_{W^{1,\infty}} \leq \max\left(1,\frac{4 \|\psi_0\|_{L^1}}{\alpha M_1}\right)
\exp\left(\frac{\mu_n (d+3)}{2}M_1\,t\right)\|\psi_0\|_{W^{1,\infty}} \qquad \forall t \in [0,T].$$
In the same way,
\begin{equation*}\begin{split}
 \|\psi^2(t)\|_{W^{1,\infty}} &\leq (1-\alpha)\max(1, \tfrac{4 \|\psi_0\|_{L^1}}{\alpha M_1}) \int_0^t \exp\left(\frac{\mu_n(d+3)}{2}M_1(t-s)\right)\|\Q_+^n(h,h)(s)\|_{W^{1,\infty}}\d s\\
&\leq (1-\alpha)\max(1,\tfrac{4 \|\psi_0\|_{L^1}}{\alpha M_1})\frac{8n^{1+\gamma} \|b_n\|_{\L}\,\ell}{\mu_n (d+3)} \left[\exp\left(\frac{\mu_n(d+3)}{2}M_1 t\right)-1\right].
\end{split}
\end{equation*}
Consequently,
\begin{equation}\label{Winfty}\begin{split}
\sup_{t \in [0,T]}\|\psi(t)\|_{W^{1,\infty}} &\leq  \max\left(1,\frac{4 \|\psi_0\|_{L^1}}{\alpha M_1}\right) \exp\left(\frac{\mu_n(d+3)}{2}M_1\,T\right)\|\psi_0\|_{W^{1,\infty}} \\
&+ \max\left(1,\frac{4 \|\psi_0\|_{L^1}}{\alpha M_1}\right)\frac{1-\alpha}{\alpha} \frac{8n\,\ell \|\psi_0\|_{L^1}}{M_1(d+3)} \left[\exp\left(\frac{\mu_n(d+3)}{2}M_1 T\right)-1\right].\end{split}\end{equation}
Now, from \eqref{density}, \eqref{r=2}, \eqref{r=2+d} and \eqref{Winfty}, one sees that, choosing for instance
 $M_1=4 \|\psi_0\|_{L^1}$,
$$M_2=4\int_{ \R^d} \psi_0(\xi)\,|\xi|^2 \, \d\xi, \qquad
C_\delta= 4 \int_{\R^d} \psi_0(\xi) \, |\xi|^{2+\delta}\, \d\xi, \qquad
\ell= \frac{4}{\alpha} \, \|\psi_0\|_{W^{1,\infty}} $$
and
\begin{multline*}
T=\frac{2}{\mu_n \, M_1 }\; \min\Bigg\{\frac{\log 2
}{(4+\delta)},\;
\frac{1}{(4+\delta)}\;\log\left(1+\frac{\alpha\,
    (4+\delta)}{(1-\alpha)\,2^{2+\delta}}\right),\;
\frac{1}{2}\;\log\left(1+\frac{\alpha\, M_1}{2(1-\alpha)}\right),\\
\frac{\log 2 }{d+3}, \frac{|\log(1-\alpha)|}{4}\,,\;\frac{1}{d+3}\;\log\left(1+\frac{\alpha^2(d+3)}{4\,n\, (1-\alpha)}
\right) \Bigg\},
\end{multline*}
we get that $\psi \in \H$, i.e. with the above choice of the parameters
$M_1,M_2,C_\delta,\ell,T$, one has  $\T(\H) \subset \H$ (notice that with this
choice, $\mu_n\, M_1=\nu_n\, M_2$). Moreover, one can prove the following:
\begin{prop} \label{cont} The mapping $\,\T\::\:\H \to \C([0,T],L^1(\R^d))$ is continuous for the topology induced by $\C([0,T],L^1(\R^d))$.   More precisely, for any $R_1>0$ and $R_2>0$,  there exist some constants $K>0$ (independent of $R_1$ and $R_2$), $K'$ (independent of $R_2$) and $C_{R_1,R_2} >0$ such that, for any $h_1,h_2 \in \H$,
\begin{equation}\label{stability}\sup_{t \in [0,T]} \left\|\T(h_1)(t)-\T(h_2)(t)\right\|_{L^1} \leq C_{R_1,R_2} \sup_{t \in [0,T]}\left\|h_1(t)-h_2(t)\right\|_{L^1} + \frac{K}{R_1^2}+\frac{K'}{R_2^\delta}.\end{equation}
Moreover, $\T(\H)$ is a relatively compact subset of  $\C([0,T],L^1_2(\R^d))$.
\end{prop}

In the proof of the above Proposition, we shall use the following result which
is very classical:

\begin{lem}\label{ABC} Let $h_1,h_2\in \C([0,T],L^1_2(\R^d))$. Then,
$$\|\LLN(t,\cdot)-\LLM(t,\cdot)\|_{L^\infty}
\leq \|b_n\|_{\L} \|\Phi_n\|_{L^\infty} \|h_1(t)-h_2(t)\|_{L^1} \qquad \forall t  >0.$$
Consequently, the following hold for any $t >0:$
\begin{multline*}|\mathbf{B}^n_{h_1}(t)-\mathbf{B}^n_{h_2}(t)|  \leq
\frac{\alpha\|b_n\|_{\L}\|\Phi_n\|_{L^\infty}}{2}\left(\|h_1(t)\|_{L^1_2}+
  \|h_2(t)\|_{L^1_2} \right) \\
\times \|h_1(t)-h_2(t)\|_{L^1_2}
\left(\frac{1}{\int_{\R^d} \psi_0(\xi)\,|\xi|^2\, d\xi}+ \frac{1}{\|\psi_0\|_{L^1}}\right),\end{multline*}
and
\begin{multline*}|\mathbf{A}^n_{h_1}(t)-\mathbf{A}^n_{h_2}(t)|
\leq \frac{\alpha\|b_n\|_{\L}\|\Phi_n\|_{L^\infty}}{2}\left(\|h_1(t)\|_{L^1_2}+
  \|h_2(t)\|_{L^1_2} \right)\\
\times\|h_1(t)-h_2(t)\|_{L^1_2}\left(\frac{d}{\int_{\R^d} \psi_0(\xi)\,|\xi|^2\, d\xi}+ \frac{d+2}{\|\psi_0\|_{L^1}}\right).\end{multline*}
Moreover, for $t \leq s$,
$$ \left| X_{h_1}(t;s,\xi)- X_{h_2}(t;s,\xi)\right|\leq |\xi|\,  \exp\left(\frac{\mu_n}{2} \; M_1\, (s-t) \right) \left|\int_s^t (\mathbf{B}^n_{h_1}(\tau)-\mathbf{B}^n_{h_2}(\tau))\, \d\tau\right|.$$
\end{lem}

\begin{proof}[Proof of Proposition \ref{cont}] Given $h_1,h_2\in \H$, we set for simplicity $X_i=X_{h_i}$, $\mathbf{A}_i^n=\mathbf{A}_{h_i}^n$ and $\mathbf{B}_i^n=\mathbf{B}_{h_i}^n$, for $i\in\{1,2\}$. We then deduce from \eqref{sol} that
\beq\label{somme}
\|\T(h_1)(t) -\T(h_2)(t) \|_{L^1} \leq  \mathcal{J}_1+  \mathcal{J}_2 +  \mathcal{J}_3 +  \mathcal{J}_4 + \mathcal{J}_5,
\eeq
where
\bean
 \mathcal{J}_1 &: =& \int_{\R^d} \left|\psi_0\left(X_1(0;t,\xi)\right)-\psi_0\left(X_2(0;t,\xi)\right) \right| \\
& & \hspace{4cm} \exp\left(-\int_0^t \left[\mathbf{A}_1^n(\tau)+\LLN\left(\tau,X_1(\tau;t,\xi)\right)\right] \d\tau\right)  \d\xi \\
\mathcal{J}_2 & :=& \int_{\R^d} \psi_0\left(X_2(0;t,\xi)\right)
\left| \exp\left(-\int_0^t \left[\mathbf{A}_1^n(\tau)+\LLN\left(\tau,X_1(\tau;t,\xi)\right)\right] \d\tau\right) \right. \\
& & \hspace{4cm}\left. -\exp\left(-\int_0^t \left[\mathbf{A}_2^n(\tau)+\LLM\left(\tau,X_2(\tau;t,\xi)\right)\right] \d\tau\right)\right|   \d\xi \\
\mathcal{J}_3 & :=&  \int_0^t \int_{\R^d} \left| \Q_+^n(h_1,h_1)\left(s,X_1(s;t,\xi)\right) -  \Q_+^n(h_2,h_2)\left(s,X_1(s;t,\xi)\right) \right| \\
& & \hspace{4cm}   \exp\left(-\int_s^t \left[\mathbf{A}_1^n(\tau)
+\LLN\left(\tau,X_1(\tau;t,\xi)\right)\right] \d\tau\right) \d\xi \, \d s\\
 \mathcal{J}_4 & :=&  \int_0^t \int_{\R^d} \left| \Q_+^n(h_2,h_2)\left(s,X_1(s;t,\xi)\right) -  \Q_+^n(h_2,h_2)\left(s,X_2(s;t,\xi)\right) \right| \\
& & \hspace{4cm} \exp\left(-\int_s^t \left[\mathbf{A}_1^n(\tau)
+\LLN\left(\tau,X_1(\tau;t,\xi)\right)\right]\d\tau\right) \d\xi \, \d s\\
 \mathcal{J}_5 & :=&  \int_0^t \int_{\R^d} \Q_+^n(h_2,h_2)\left(s,X_2(s;t,\xi)\right) \left|\exp\left(-\int_s^t \left[\mathbf{A}_1^n(\tau)
+\LLN\left(\tau,X_1(\tau;t,\xi)\right)\right] \d\tau\right)\right. \\
& & \hspace{4cm}\left.
-\exp\left(-\int_s^t \left[\mathbf{A}_2^n(\tau) +\LLM\left(\tau,X_2(\tau;t,\xi)\right)\right] \d\tau\right)\right| \d\xi \, \d s.
\eean
Let us estimate these five terms separately. Let $R_1>0$. First, since $\psi_0\in W^{1,\infty}(\R^d)$ and $\LLN$ is nonnegative, it follows from  Lemma \ref{lem1} that
\bean
\mathcal{J}_1 & \leq & \| \psi_0\|_{W^{1,\infty}}
\exp\left(\frac{T\,\mu_n\,(d+2)\,M_1}{2}\right)
\int_{|\xi|\leq R_1} \left|X_1(0;t,\xi)-X_2(0;t,\xi) \right|  \d\xi \\
& +  & \frac{1}{R_1^2} \int_{\R^d} \left(\psi_0\left(X_1(0;t,\xi)\right)+\psi_0\left(X_2(0;t,\xi)\right) \right) \exp\left(-\int_0^t \mathbf{A}_1^n(\tau)\, \d\tau\right) |\xi|^2\, \d\xi.
\eean
Now, by a simple change of variable, the use of Lemma \ref{ABC} and Lemma \ref{lem1} leads to\bean
\mathcal{J}_1 & \leq & \frac{R_1^{d+1}}{{d+1}}\,|\mathbb{S}^{d-1}|\, \| \psi_0\|_{W^{1,\infty}}\exp\left(\frac{T\,\mu_n\,(d+3)\,M_1}{2}\right)
\left|\int_0^t (\mathbf{B}^n_1(\tau)-\mathbf{B}^n_2(\tau))\, \d\tau\right| \\
& +  & \frac{2}{R_1^2}\;   \| \psi_0\|_{L^1_2}
 \exp\left(\frac{\mu_n\, (d+2)}{2} M_1 T+\frac{\nu_n\, (d+2)}{2} M_2 T \right).
\eean
We then deduce from Lemma \ref{ABC} the existence of some constants $C_{1,R_1}>0$
and $K_1>0$ (independent of $R_1$) such that
\beq\label{J1}
\mathcal{J}_1 \leq  C_{1,R_1} \sup_{t\in[0,T]} \|h_1(t)-h_2(t)\|_{L^1_2}
+\frac{K_1}{R_1^2}.
\eeq
Let us turn our attention to $\mathcal{J}_2$. One deduces from the mean value theorem and Lemma \ref{lem1} that
\begin{multline}\label{decomp_J2}
\mathcal{J}_2 \leq \exp\left(\frac{\mu_n\,(d+2)}{2}\; M_1 T\right) \int_{\R^d} \psi_0\left(X_2(0;t,\xi)\right)  \left( \int_0^t \left|\mathbf{A}_1^n(\tau)-\mathbf{A}_2^n(\tau)\right| \d\tau \right.\\
+ \left. \int_0^t \left|\LLN\left(\tau,X_1(\tau;t,\xi)\right)- \LLM\left(\tau,X_1(\tau;t,\xi)\right)\right| \d\tau  \right.\\
+ \left. \int_0^t \left|\LLM\left(\tau,X_1(\tau;t,\xi)\right)- \LLM\left(\tau,X_2(\tau;t,\xi)\right)\right| \d\tau \right) \d\xi.
\end{multline}
But, for $j\in\{1,2\}$, a change of variables leads to
\begin{multline*}
\LLM\left(\tau,X_j(\tau;t,\xi)\right)= \|b_n\|_{\L} \exp\left(-d \int_\tau^t \mathbf{B}_j^n(s) \d s\right) \\
\int_{\R^d} \Phi_n\left(|X_j(\tau;t,\xi-\xi_*)|\right)\,
 h_2(\tau,X_j(\tau;t,\xi_*)) \d\xi_*.
\end{multline*}
Thus, since $h_2\in W^{1,\infty}(\R^d)$ and since, for any $\lambda_1,\lambda_2,r\geq 0$,
$$ \left|\Phi_n(\lambda_1\, r)-\Phi_n(\lambda_2\, r)\right| \leq |\lambda_1^\gamma-\lambda_2^\gamma| \, r^\gamma,$$
we obtain, in virtue of  Lemma \ref{ABC},
\bean
& & \hspace{-5mm}
\left|\LLM\left(\tau,X_1(\tau;t,\xi)\right)- \LLM\left(\tau,X_2(\tau;t,\xi)\right)\right| \\
& & \hspace{5mm} \leq d\, \|b_n\|_{\L}\, \|\Phi_n\|_{L^\infty} \, \| h_2\|_{L^1}\,  e^{\frac{d\,\mu_n}{2} \; M_1\, T} e^{\frac{d\,\nu_n}{2} \; M_2\, T} \left|\int_\tau^t (\mathbf{B}^n_1(s)-\mathbf{B}^n_{h_2}(s)) \d s\right|  \\
& & \hspace{5mm} + \|b_n\|_{\L}\, \|\Phi_n\|_{L^\infty} \, \| h_2\|_{W^{1,\infty}}\,  e^{\frac{(d+1)\, \mu_n}{2} \; M_1\, T} \left|\int_\tau^t (\mathbf{B}^n_1(s)-\mathbf{B}^n_{h_2}(s)) \d s\right| \int_{|\xi_*|\leq R_1} |\xi_*| \, \d \xi_* \\
& & \hspace{5mm} + \frac{1}{R_1^2} \; \|b_n\|_{\L}\,\|\Phi_n\|_{L^\infty} \,e^{\frac{d\,\mu_n}{2} \; M_1\, T} \int_{\R^d} \left(h_2(\tau,X_1(\tau;t,\xi_*))+  h_2(\tau,X_2(\tau;t,\xi_*)) \right)\, |\xi_*|^2\, \d\xi_* \\
& & \hspace{5mm} + \gamma \, \|b_n\|_{\L}  e^{\frac{(d+\gamma) \,\mu_n}{2} \: M_1\, T} \left|\int_\tau^t (\mathbf{B}^n_1(s)-\mathbf{B}^n_2(s)) \d s\right|
 \int_{\R^d} |\xi-\xi_*|^\gamma \, h_2(\tau,X_2(\tau;t,\xi_*)) \d\xi_*.
\eean
A change of variables and Lemma \ref{ABC} then lead to the existence of some constants  $C_{\ell,R_1}>0$ and $K_\ell>0$ (independent of $R_1$) such that
\beq \label{L}
\left|\LLM\left(\tau,X_1(\tau;t,\xi)\right)- \LLM\left(\tau,X_2(\tau;t,\xi)\right)\right|\leq   \langle\xi\rangle^\gamma C_{\ell,R_1} \sup_{t\in[0,T]} \|h_1(t)-h_2(t)\|_{L^1_2}
+\frac{K_\ell}{R_1^{2}}.
\eeq
Gathering \eqref{decomp_J2}, \eqref{L} and Lemma \ref{ABC}, we deduce that there
exist some constants $C_{2,R_1}>0$ and $K_2>0$ (independent of $R_1$) such that
\beq\label{J2}
\mathcal{J}_2 \leq  C_{2,R_1} \sup_{t\in[0,T]} \|h_1(t)-h_2(t)\|_{L^1_2}
+\frac{K_2}{R_1^{2}}.
\eeq
Performing the same manipulations for $\mathcal{J}_5$, one may show that there exist some constants $C_{5,R_1}>0$ and $K_5>0$ (independent of $R_1$) such that
\beq\label{J5}
\mathcal{J}_5 \leq  C_{5,R_1} \sup_{t\in[0,T]} \|h_1(t)-h_2(t)\|_{L^1_2}
+\frac{K_5}{R_1^{2}}.
\eeq
Then,
$$ \mathcal{J}_4 \leq  e^{\frac{\mu_n\,(d+2)}{2}\: M_1\, T} \int_0^t \int_{\R^d} \left| \Q_+^n(h_2,h_2)\left(s,X_1(s;t,\xi)\right) -  \Q_+^n(h_2,h_2)\left(s,X_2(s;t,\xi)\right) \right| \d\xi \, \d s,  $$
and, changing variables, we get, for $j\in\{1,2\}$,
\begin{multline*}
 \Q_+^n(h_2,h_2)\left(s,X_j(s;t,\xi)\right) =
\exp\left(-d \int_s^t \mathbf{B}_j^n(s) \d s\right)
\int_{\R^d}\int_{\mathbb{S}^{d-1}} b_n(\cos\theta)\,  \Phi_n(|X_j(s;t,\xi-\xi_*)|) \\
h_2(s,X_j(s;t,\xi'))\,  h_2(s,X_j(s;t,\xi'_*))\d\sigma\, d\xi_*.
\end{multline*}
 Thus, proceeding as for $\LLM$, one may prove that there exist some constants
$C_{4,R_1}>0$ and $K_4>0$ (independent of $R_1$) such that
\beq\label{J4}
\mathcal{J}_4 \leq  C_{4,R_1} \sup_{t\in[0,T]} \|h_1(t)-h_2(t)\|_{L^1_2}
+\frac{K_4}{R_1^{2}}\,.
\eeq
For the last integral, we have
\bea
\mathcal{J}_3 & \leq & e^{\mu_n\, M_1\, T} \int_0^t \int_{\R^d} \left| \Q_+^n(h_1-h_2,h_1)(s,\xi)\right| +  \left| \Q_+^n(h_2,h_1-h_2)(s,\xi) \right| \d\xi \, \d s \nonumber\\
& \leq & C_3  \sup_{t\in[0,T]} \|h_1(t)-h_2(t)\|_{L^1}  \label{J3}
\eea
for some constant $C_3>0$. Finally, gathering \eqref{somme}, \eqref{J1}, \eqref{J2}, \eqref{J5}, \eqref{J4}, \eqref{J3} and, noticing that, for $R_2>0$,
$$ \|h_1(t)-h_2(t)\|_{L^1_2} \leq (1+R_2^2)\, \|h_1(t)-h_2(t)\|_{L^1} + \frac{1}{R_2^\delta} \left(\|h_1(t)\|_{L^1_{2+\delta}}+ \|h_2(t)\|_{L^1_{2+\delta}}\right)$$
this completes the proof of \eqref{stability}.  Let us now prove the compactness of $\T(\H)$. Recall that, according to Riesz-Fréchet-Kolmogorov Theorem, the embedding $$L^1_{2+\delta}(\R^d)\cap W^{1,\infty}(\R^d) \subset L^1_2(\R^d)$$ is compact. Moreover,  $L^1_2(\R^d)$ is continuously embedded into $\left(H^m(\R^d)\right)'$ for $m>d/2$. On the other hand,
$$\T(\H) \mbox{ is a bounded subset of  }
L^\infty\left((0,T); L^1_{2+\delta}(\R^ d)\cap W^{1,\infty}(\R^d)\right)$$
and, setting $\partial_t \T(\H)=\{\partial_t \psi\,;\,\psi=\T(h),\,h \in \H\}$, one has
$$\partial_t\T(\H) \mbox{ is a bounded subset of  } L^r((0,T);(H^m(\R^d))'),$$
with $r>1$. As a consequence, one can apply \cite[Corollary 4]{Sim} to conclude that $\T(\H)$
is a relatively compact subset of $\C([0,T]; L^1_2(\R^d))$.\end{proof}

We are in position to conclude the proof of Theorem \ref{cauchypb}.

\begin{proof}[Proof of Theorem \ref{cauchypb}]  The proof is split into two parts: the first one consists in proving the well-posedness of the Cauchy problem \eqref{annihi} on the time interval $[0,T]$ (where $T >0$ has been defined hereabove) through \textit{Schauder fixed point theorem}. The second part consists in extending this solution to a global solution.\\

\noindent{\it Local existence:} Since $\H$ is a closed bounded (nonempty) subset of
$\C([0,T]; L^1_2(\R^d))$ and since $\T$ is a continuous and compact application from $\H$ to $\H$, Schauder fixed point theorem ensures the existence of some fixed point $\psi^1$ of $\T$, i.e. there exists $\psi^1  \in \C([0,T]; L^1_2(\R^d))\cap L^\infty((0,T); L^1_{2+\delta}(\R^d)\cap
W^{1,\infty}(\R^d))$ solution to \eqref{annihi}.\\

\noindent{\it Global existence:} Integrating the equation \eqref{annihi} over $\R^d$, we get
$$\frac{\d}{\d t} \int_{ \R^d}  \psi^1(t,\xi)\, \d\xi
= \frac{\alpha}{\|\psi_0\|_{L^1}}  \left(\int_{ \R^d}\Q_-^n(\psi^1,\psi^1)(t,\xi)\, \d\xi \right)
\left( \int_{ \R^d} \psi^1(t,\xi) \, \d\xi- \|\psi_0\|_{L^1}\right).$$
Since $\displaystyle \int_{ \R^d} \psi^1(0,\xi) \, \d\xi=\|\psi_0\|_{L^1}$, we see that the density of $\psi^1$ is conserved:
$$\int_{ \R^d} \psi^1(t,\xi) \, \d\xi= \int_{ \R^d} \psi_0(\xi) \, \d\xi \qquad \forall t \in [0,T].$$

In the same way, multiplying \eqref{annihi} by  $|\xi|^2$ and integrating over $\R^d$ yields
$$\frac{\d}{\d t}  \int_{ \R^d}  \psi^1(t,\xi)\, |\xi|^2\,  \d\xi
= \alpha \left(\int_{ \R^d} |\xi|^2 \, \Q_-^n(\psi^1,\psi^1)(t,\xi)\, \d\xi
\right) \left( \frac{\int_{ \R^d} \psi^1(t,\xi) \,|\xi|^2\,  \d\xi}{\int_{
      \R^d} \psi_0(\xi)\,|\xi|^2 \, \d\xi}-1 \right).$$
Since $\displaystyle \int_{ \R^d} \psi^1(0,\xi)\,|\xi|^2 \, \d\xi=\int_{ \R^d} \psi_0(\xi)\,|\xi|^2 \, \d\xi$, the energy of $\psi^1(t,\xi)$ is conserved:
$$\int_{ \R^d} \psi^1(t,\xi) \,|\xi|^2 \, \d\xi=\int_{ \R^d} \psi_0(\xi)\,|\xi|^2 \, \d\xi  \qquad \forall t \in [0,T].$$
Thus, $\psi^1(T,.)$ has the same mass and energy as $\psi_0$. Since the time
$T$ only depends on these values, by a standard continuation argument, we
construct a global solution $\psi$ to \eqref{annihi}. \end{proof}

\subsection{Uniform estimates}

In order to prove Theorem \ref{well_posedness}, we now need to get rid
of the bound in $W^{1,\infty}(\R^d)$ for the initial condition and to pass to
the limit as $n\to +\infty$.

Let  $p>1$. Let $\psi_0\in L^1_{2+\gamma}(\R^d)\cap L^p(\R^d)$ be
a nonnegative distribution function satisfying \eqref{massenergie}. There
exists a sequence of nonnegative functions $(\psi_0^n)_{n\in\N}$ in
$W^{1,\infty}(\R^d)\cap L^1_{2+\gamma}(\R^d)$  that converges to $\psi_0$ in $L^1_2(\R^d)$ and that satisfies, for any $n\in\N$,
$$ \|\psi_0^n\|_{L^1}\leq \|\psi_0\|_{L^1}\quad \mbox{ and } \quad\|\psi_0^n\|_{L^p}\leq \|\psi_0\|_{L^p}.$$
Moreover, if $\psi_0\in L^1_s(\R^d)$ with $s>2$ then one may also assume that
\beq\label{majo}
\int_{\R^d} \psi_0^n(\xi)\,|\xi|^s\, \d\xi \leq 2^{s-1} \|\psi_0\|_{L^1}
+2^{s-1} \int_{\R^d}\psi_0(\xi)\, |\xi|^s\,\d\xi.
\eeq
We infer from the above properties of $(\psi_0^n)_{n\in\N}$ and from \eqref{massenergie} that there exists some $N_0\in\N$ such that for $n\geq N_0$,
\beq\label{mino}
\frac{1}{2} \leq \int_{\R^d}\psi_0^n(\xi)\, \d\xi \leq 1 \qquad \mbox{ and }
\qquad \frac{d}{4} \leq \int_{\R^d}\psi_0^n(\xi)\,|\xi|^2\, \d\xi \leq d.
\eeq
For each $n\in\N$, we denote by $\psi_n$ a solution to \eqref{annihi} with
initial condition $\psi_0^n$. Notice that, for any given $T > 0$ and any $n \in \mathbb{N}$, the solution $\psi_n$ constructed as a "mild solution" is also a weak solution, i.e., the following holds for any $\varrho \in \mathcal{C}^1_c(\R^d)$ and any $t \geq 0$:
\begin{multline}\label{weakformun}
\int_{\R^d} \psi_n(t,\xi)\varrho(\xi)\d\xi + \int_0^t \d s
\big[\mathbf{A}_{\psi_n}^n(s)-d \mathbf{B}_{\psi_n}^n(s)\big]\int_{\R^d}  \varrho(\xi)\,
\psi_n(s,\xi)\,\d\xi \\
= \int_0^t \d s \mathbf{B}_{\psi_n}^n(s) \int_{\R^d}
\psi_n(s,\xi)\, \xi \cdot \nabla_\xi \varrho(\xi) \d\xi
+ \int_{\R^d} \varrho(\xi)\psi_0^n(\xi) \d \xi
+ \int_0^t \d s\int_{\R^d} \mathbb{B}_n(\psi_n,\psi_n)(s,\xi)\varrho(\xi)\d\xi.
\end{multline}

Our purpose is to show that $(\psi_n)_{n\in\N}$ is converging in $\C(|0,T], w-L^1(\R^d))$  for any $T>0$. However, this
requires uniform estimates on $\psi_n$. So, we now tackle this question and
show uniform bounds for moments of $\psi_n$. The underlying difficulty comes
from the two terms $\mathbf{A}_{\psi_n}^n$ and $\mathbf{B}_{\psi_n}^n$ which
already involve moments of order $2+\g$ and thereby prevent us from performing
direct estimates. In all the sequel, we shall simply set
$$\mathbf{A}_n(t)=\mathbf{A}_{\psi_n}^n(t), \qquad \mathbf{B}_n(t)=\mathbf{B}_{\psi_n}^n(t), \qquad n \in \mathbb{N},\qquad t \geq 0.$$
We begin with proving that both $\mathbf{A}_n$ and $\mathbf{B}_n$ are bounded
in $L^1_{\mathrm{loc}}(0,\infty)$. Here again we first need to show uniform
$L^p$-estimates, which is the aim of the following lemma.

\begin{lem}\label{lem:lp}
There exist some integer $N_1\geq N_0$ and some constant $C>0$ depending only
on $\alpha$, $p$, $d$ and $\gamma$ such that, for all $n\geq N_1$,
\beq \label{estim_lp}
\|\psi_n(t)\|_{L^p}\leq e^{Ct} \; \|\psi_0\|_{L^p},\qquad t\geq 0.
\eeq
\end{lem}

\begin{proof}  For $n\in\N_*$, we multiply (\ref{annihi}) by
$p\, \psi_n(t,\xi)^{p-1}$ and integrate over $\R^d$. An integration by parts
then leads to
\bea
\frac{\d}{\d t} \|\psi_n(t)\|^p_{L^p}
& = & (d \mathbf{B}_n(t)-p \mathbf{A}_n(t))\,
\|\psi_n(t)\|_{L^p}^p  \nonumber \\
& + & (1-\a)\, p \int_{\R^d}\Q_+^n(\psi_n,\psi_n)(t,\xi)\,
\psi_n(t,\xi)^{p-1}\, \d\xi \nonumber\\
& - & \a \, p \int_{\R^d}\Q_-^n(\psi_n,\psi_n)(t,\xi) \, \psi_n(t,\xi)^{p-1}\,
\d\xi. \label{LP}
\eea
First, since $p>1$, we have, for $n\geq N_0$,
\bea
d \mathbf{B}_n(t)-p \mathbf{A}_n(t)
& = & \frac{\a}{2} \int_{\R^d} \left(\frac{d(p-1)+2p}{\|\psi_0^n\|_{L^1}}
-\frac{d(p-1)|\xi|^2}{\int_{\R^d}\psi_0^n(\xi_*)\,|\xi_*|^2\, \d\xi_*}\right) \,
\Q_-^n(\psi_n,\psi_n)(t,\xi) \, \d\xi  \nonumber\\
& \leq  & \a(d(p-1)+2p) \int_{\R^d}
 \Q_-^n(\psi_n,\psi_n)(t,\xi) \, \d\xi. \label{estim_1}
\eea
But, since $\g\in(0,1]$,
\beq\label{maj_Phi_n}
\Phi_n(|\xi-\xi_*|)\leq |\xi-\xi_*|^\g \leq |\xi|^\g+|\xi_\ast|^\g.
\eeq
Consequently,
\beq \label{estim_2}
\int_{\R^d} \Q_-^n(\psi_n,\psi_n)(t,\xi) \, \d\xi
\leq 2\,  \|b_n\|_{\L} \int_{\R^d} |\xi|^\g \psi_n(t,\xi)\, \d\xi
\leq 2\, \|b\|_{\L}\, (1+d).
\eeq
Thereby, we obtain a bound for the first term in the right-hand side of
(\ref{LP}). We now need to estimate the two remaining integrals. We first
notice that, due to the symmetry, we can reduce the domain of integration with
respect to $\s$ to those $\s$ that satisfy
$\langle \xi-\xi_* , \s \rangle \geq 0$, which corresponds to
$\theta\in[0,\pi/2]$. This
amounts to taking $ b_n (x) =\, \mathbf{1}_{\{ 0\leq x\leq 1-1/n\}}\overline{b}(x)$ in the
collision operator $\Q$ where
$$\overline{b}(x)=b(x)+b(-x).$$
Then, for some fixed $\theta_0\in[\arccos(1-1/n),\pi/2] $, we
split $b_n$ as  $ b_n= b_{n,c} + b_{n,r}$ where
$$b_{n,c}(x)=\mathbf{1}_{\{ 0\leq x\leq \cos \t_0\}}\overline{b}(x) \quad \text{ and } \quad b_{n,r}(x)= \mathbf{1}_{\{ \cos \t_0\leq x\leq 1-1/n\}}\overline{b}(x).$$
It is important to point out that
 $b_{n,c}$ and consequently the norm $\|b_{n,c}\|_{\L}$ do not depend on $n$ but only on $\theta_0$.
This splitting leads to the corresponding decomposition of the collision
operators:
\beq \label{estim_3}
 \Q_+^n=\Q_+^{n,c}+\Q_+^{n,r}\qquad \mbox{ and } \qquad
\Q_-^n=\Q_-^{n,c}+\Q_-^{n,r}.
\eeq
We first consider $\Q_+^{n,r}$ and $\Q_-^{n,r}$. We have
\beq \label{estim_8}
\int_{\R^d} \Q_-^{n,r}(\psi_n,\psi_n)(t,\xi)\, \psi_n(t,\xi)^{p-1}\, \d\xi
\geq 0.
\eeq
Then, for the integral involving $\Q_+^{n,r}$, the change of variables
$(\xi,\xi_*)\to (\xi',\xi'_*)$ yields
\bean
& & \hspace{-8mm} \int_{\R^d} \Q_+^{n,r}(\psi_n,\psi_n)(t,\xi) \,
\psi_n(t,\xi)^{p-1}\, \d\xi \\
& & = \int_{\R^d}\int_{\R^d}\int_{\S^{d-1}}  \psi_n(t,\xi)\,
\psi_n(t,\xi_*) \, \psi_n(t,\xi')^{p-1}b_{n,r}(\cos\theta)
\Phi_n(|\xi-\xi_*|) \, \d\s \, \d\xi\, \d\xi_*
\eean
Now, we have
$$ \psi_n(t,\xi)\, \psi_n(t,\xi')^{p-1}  \leq  \frac{1}{p} \: \psi_n(t,\xi)^p\,
+ \frac{p-1}{p} \:\psi_n(t,\xi')^p,  $$
and (see \cite[Section 3, Proof of Lemma 1]{ADVW00} or \cite[Eq. (2.7)]{DeMou05})
\bean
& &  \hspace{-1cm}\int_{\R^d}\int_{\S^{d-1}} \psi_n(t,\xi')^p \,
\mathbf{1}_{\left\{\cos\t_0 \leq \cos \t \leq 1-1/n \right\}} \,\overline{b}(\cos\theta)
\Phi_n(|\xi-\xi_*|) \, \d\s \, \d\xi \\
&  &  \hspace{1cm} = |\S^{d-2}|  \int_{\R^d}\int_{\arccos(1-1/n)}^{\t_0}
\psi_n(t,\xi)^p \,   \Phi_n\left(\frac{|\xi-\xi_*|}{\cos(\t/2)}\right)
\frac{\sin^{d-2}(\t)}{\cos^d(\t/2)} \;\,\overline{b}(\cos\theta) \d\t\, \d\xi.
\eean
Then, thanks  to the inequalities
\beq\label{maj_wn}
\Phi_n(|\xi-\xi_*|) \leq  \Phi_n(|\xi|) + |\xi_*|^\g  \quad \mbox{ and }\quad
\Phi_n\left(\frac{|\xi-\xi_*|}{\l}\right) \leq
\l^{-\g} \, \Phi_n(|\xi-\xi_*|), \qquad \forall 0 < \l < 1,
\eeq
we get
\bea
& & \hspace{-8mm} \int_{\R^d} \Q_+^{n,r}(\psi_n,\psi_n)(t,\xi) \,
\psi_n(t,\xi)^{p-1}\, \d\xi \nonumber \\
& & \leq  |\S^{d-2}| \,
\int_{\arccos(1-1/n)}^{\t_0} \,\overline{b}(\cos\theta) (1+ (\cos(\t/2))^{-d-\g}) \,\sin^{d-2}(\t)\, \d\t
\nonumber   \\
& & \hspace{3cm}  \times
\left(\int_{\R^d} \psi_n(t,\xi)^p \, \Phi_n(|\xi|)\, \d\xi
+ \left( 1+ d \right) \|\psi_n(t)\|^p_{L^p}\right). \label{estim_9}
\eea
Let us now consider $\Q_+^{n,c}$ and $\Q_-^{n,c}$. We proceed as in the proof
of \cite[Proposition 2.4]{DeMou05}. Since
\beq\label{min_Phi_n}
\Phi_n(|\xi-\xi_*|)\geq \Phi_n(|\xi|) -|\xi_*|^\g,
\eeq
 we deduce that
\begin{multline}\label{estim_4}
\int_{\R^d} \Q_-^{n,c}(\psi_n,\psi_n)(t,\xi)\, \psi_n(t,\xi)^{p-1}\, \d\xi
\geq  \frac{1}{2} \|b_{n,c}\|_{\L}
\int_{\R^d} \psi_n(t,\xi)^p\, \Phi_n(|\xi|)\, \d\xi \\
- \|b_{n,c}\|_{\L} \, (1+d) \;  \|\psi_n(t)\|^p_{L^p}\,.\end{multline}
 On the other hand,
\beq \label{estim_5}
\int_{\R^d} \Q_+^{n,c}(\psi_n,\psi_n)(t,\xi) \, \psi_n(t,\xi)^{p-1}\, \d\xi
= J_1+J_2,
\eeq
where
\begin{equation*}\begin{split}
J_1 &=  \int_{\R^{2d}} \int_{\S^{d-1}} \psi_n(t,\xi')\, \psi_n(t,\xi'_*) \,
\mathbf{1}_{\{|\xi'|\leq r\}} \, \psi_n(t,\xi)^{p-1}b_{n,c}(\cos\theta)
\Phi_n(|\xi-\xi_*|) \, \d\s \, \d\xi\, \d\xi_*,\\
J_2 &=     \int_{\R^{2d}} \int_{\S^{d-1}} \psi_n(t,\xi')\, \psi_n(t,\xi'_*) \,
\mathbf{1}_{\{|\xi'|\geq r\}} \, \psi_n(t,\xi)^{p-1}b_{n,c}(\cos\theta)
\Phi_n(|\xi-\xi_*|) \, \d\s \, \d\xi\, \d\xi_*,\end{split}\end{equation*}
with $r>0$. Performing the same calculations as in the proof of
\cite[Proposition 2.4]{DeMou05} and using the same notations, we prove easily (using again \eqref{maj_wn}) that the following hold for any $\mu_1 >0$ and any $\mu_2 >0$:
\begin{multline}\label{estim_6}
J_1 \leq (\cos(\pi/4))^{-d-\g} \,
\left(1-\frac{1}{p}\right) \mu_1^{-1} \|b_{n,c}\|_{\L} \left(\int_{\R^d} \psi_n(t,\xi)^p \, \Phi_n(|\xi|)\,
  \d\xi + \left( 1+d \right) \|\psi_n(t)\|^p_{L^p}\right)\\
 +  \frac{1}{p} \, \mu_1^{p-1}\, \|b_{n,c}\|_{\L} \, \left(1+r^\g+d\right)
\,  \|\psi_n(t)\|^p_{L^p}
\end{multline}
and
\begin{multline}\label{estim_7} J_2  \leq   (\sin(\t_0/2))^{-d-\g} \,
\left(1-\frac{1}{p}\right)
\mu_2^{-1} \|b_{n,c}\|_{\L}\left(\frac{d}{r^2} \int_{\R^d}
  \psi_n(t,\xi)^p \,\Phi_n(|\xi|)\, \d\xi + \frac{d}{r^{2-\g}}\,
  \|\psi_n(t)\|^p_{L^p}\right)\\
+ \frac{\mu_2^{p-1}}{p} \,\|b_{n,c}\|_{\L}\left( \int_{\R^d} \psi_n(t,\xi)^p \, \Phi_n(|\xi|)\,
  \d\xi + (1+d)\,  \|\psi_n(t)\|^p_{L^p} \right).\end{multline}
It remains now to choose the parameters $\theta_0$, $\mu_1$, $\mu_2$ and $r$ so that all the terms involving $\int_{\R^d} \psi_n(t,\xi)^p \, \Phi_n(|\xi|)\,
  \d\xi$ that appear in the gain term can be absorbed by the one appearing in the estimate of the loss term. Precisely, we first choose $\theta_0$ small enough such that
$$|\S^{d-2}|\int_0^{\t_0} \,\overline{b}(\cos\theta) (1+ (\cos(\t/2))^{-d-\g}) \,\sin^{d-2}(\t)\, \d\t\leq
a \|b_{n,c}\|_{\L}$$
for some $a >0$ to be determined later (recall that $\|b_{n,c}\|_{\L}$ only depends on $\theta_0$). Then, we choose $\mu_1$ big enough and $\mu_2$ small enough such that
$$(p-1) (\cos(\pi/4))^{-d-\g}\mu_1^{-1}\leq ap\qquad
\mbox{ and } \qquad \mu_2^{p-1} \leq ap. $$
Finally, we choose $r$  big enough such that
$$(p-1) (\sin(\t_0/2))^{-d-\g}\mu_2^{-1} \frac{d}{r^2}
\leq ap. $$
Let $N_1\in\N_*$ be such that $N_1\geq \max
\left\{\frac{1}{1-\cos\t_0},N_0\right\}$. Gathering
(\ref{estim_1}), (\ref{estim_2}), (\ref{estim_3}), (\ref{estim_8}),
(\ref{estim_9}), (\ref{estim_4}), (\ref{estim_5}), (\ref{estim_6}) and
(\ref{estim_7}) we conclude that, for $n\geq N_1$,
$$\frac{\d}{ \d t} \|\psi_n(t)\|^p_{L^p}  \leq \frac{8(1-\a)ap-p}{2} \; \|b_{n,c}\|_{\L}\int_{\R^d} \psi_n(t,\xi)^p \, \Phi_n(|\xi|)\,
  \d\xi  + C\|\psi_n(t)\|_{L^p}^p$$
for some positive constant $C$ that only depends on $\a$, $b(\cdot)$, $p$, $d$, $\mu_1$, $r$ and $\g$. Taking then $a=\frac{1}{16(1-\a)}$  we get
$$\frac{\d}{ \d t} \|\psi_n(t)\|^p_{L^p} + \frac{ p}{4}\|b_{n,c}\|_{\L}
 \int_{\R^d} \psi_n(t,\xi)^p \, \Phi_n(|\xi|)\, \d\xi
\leq C  \|\psi_n(t)\|^p_{L^p}.$$
Recalling again that $\|b_{n,c}\|_{\L}$ does not depend on $n$, the Gronwall Lemma and the inequality $\|\psi_0^n\|_{L^p}\leq
\|\psi_0\|_{L^p}$ then imply that (\ref{estim_lp}) holds.
\end{proof}

We now deduce from these $L^p$-estimates the following lemma, which implies
that  $\mathbf{A}_{n}$ and $\mathbf{B}_n$ are uniformly
bounded in $L^1_{\mathrm{loc}}(0,\infty)$.

\begin{lem} \label{lem_2+gamma}
Let $T>0$. There exists some constant $C$ depending only on $\a$, $d$, $\g$,
$p$, $T$ and $\|\psi_0\|_{L^p}$ such that, for $n\geq N_1$,
\beq\label{mom_2+gamma}
\int_0^T \int_{\R^d} \psi_n(t,\xi)\, |\xi|^{2} \, \Phi_n(|\xi|) \, \d\xi \,
\d t \leq C.
\eeq
\end{lem}

\begin{proof}
Let $n\geq N_1$. For $s\in(0,2)$, we multiply \eqref{annihi} by $|\xi|^s$
and integrate over $\R^d$. Integrations by parts then lead to
\bea
\frac{\d Y^n_s}{\d t}(t)
& = & \frac{\a}{2} \; Y^n_s(t) \int_{\R^d} \left(\frac{2-s}{\|\psi_0^n\|_{L^1}}
+ \frac{s \,|\xi|^2}{\int_{\R^d}\psi_0^n(\xi_*)\, |\xi_*|^2\, \d\xi_*} \right)
\Q^n_-(\psi_n,\psi_n)(t,\xi)\, \d\xi  \nonumber\\
& + & \frac{1-\a}{2} \int_{\R^d}\int_{\R^d} \psi_n(t,\xi)\, \psi_n(t,\xi_*)
\, \Phi_n(|\xi-\xi_*|) \, K^n_s(\xi,\xi_*)\, \d\xi\, \d\xi_* \nonumber \\
& - & \a \int_{\R^d} \Q^n_-(\psi_n,\psi_n)(t,\xi)\, |\xi|^s\, \d\xi,
\label{eq_mom}
\eea
where we set $ \displaystyle Y^n_s(t) = \int_{\R^d} \psi_n(t,\xi)\,
|\xi|^s \, \d\xi$ and
$$ K^n_s(\xi,\xi_*) = \int_{\S^{d-1}}\mathbf{1}_{\{|\cos \t |\leq 1-1/n \}}b(\cos\t)
\left(|\xi'|^s+|\xi'_*|^s-|\xi|^s -|\xi_*|^s \right)\, \d\s. $$
By \cite[Lemma 2.2 (ii)]{MiWe99}, one can write
$K^n_s(\xi,\xi_*)=G^n_s(\xi,\xi_*)-H^n_s(\xi,\xi_*)$ with
$$ H^n_s(\xi,\xi_*) \leq 0 \qquad \mbox{ and } \qquad
|G^n_s(\xi,\xi_*)| \leq c_1 \, |\xi|^{s/2} \, |\xi_*|^{s/2},$$
for some constant $c_1$ depending only on $b(\cdot)$, $s$ and $d$. Integrating the
previous inequality between $0$ and $T$, we get
\bean
& & \hspace{-5mm} Y^n_s(0) +\frac{\a\, s}{2 \int_{\R^d}\psi_0^n(\xi_*)\,
  |\xi_*|^2\, \d\xi_*} \int_0^T
\left(\int_{\R^d} |\xi|^2\, \Q^n_-(\psi_n,\psi_n)(\tau,\xi)\,
\d\xi \right) Y^n_s(\tau)\, \d\tau  \\
& & \hspace{5mm}\leq Y^n_s(T)+ \|b_n\|_{\L} \int_0^T \int_{\R^d}
\int_{\R^d} \Phi_n(|\xi-\xi_*|) \, |\xi|^s \, \psi_n(\tau,\xi)\,
\psi_n(\tau,\xi_*) \, \d\xi\, \d\xi_*\, \d\tau  \\
& &\hspace{5mm} + \frac{c_1}{2} \int_0^T \int_{\R^d} \int_{\R^d}
\Phi_n(|\xi-\xi_*|) \, |\xi|^{s/2}\, |\xi_*|^{s/2}   \,
\psi_n(\tau,\xi)\, \psi_n(\tau,\xi_*) \, \d\xi\, \d\xi_*\, \d\tau,
\eean
since $s<2$ and $0<\a<1$. We then deduce from \eqref{mino}, (\ref{maj_Phi_n})
and (\ref{min_Phi_n}) that
\bean
& & \hspace{-7mm}  \frac{\a s }{2 d} \int_0^T \left(\int_{\R^d} \Phi_n(|\xi|)
\,|\xi|^2 \psi_n(\tau,\xi)\, \d\xi \right) Y^n_s(\tau)\, \d\tau
\leq  \frac{s}{2} \int_0^T Y^n_\g(\tau)\, Y^n_s(\tau)\, \d\tau  \\
& & \hspace{1cm} + Y^n_s(T)+ \|b_n\|_{\L} \int_0^T \left(Y^n_{s+\g}(\tau)
+  Y^n_s(\tau) \,  Y^n_\g(\tau) \right) \, \d\tau \\
& & \hspace{4cm} + c_1 \int_0^T  Y^n_{s/2+\g}(\tau) \,Y^n_{s/2}(\tau) \, \d\tau\,.
\eean
Taking $s=2-\g$ and using that for any $\nu\in(0,2)$,
$Y^n_\nu(\tau) \leq Y^n_0(\tau) +Y^n_2(\tau) \leq 1+d$
we get
$$\int_0^T \left(\int_{\R^d} \Phi_n(|\xi|) \, |\xi|^2 \psi_n(\tau,\xi)\,
\d\xi \right) Y^n_{2-\g}(\tau)\, \d\tau \leq C,$$
for some constant $C$ depending only on $b(\cdot)$, $\a$, $d$, $\g$ and $T$. Now, for
$R>0$ and $p>1$,
$$  Y^n_{2-\g}(\tau) \geq R^{2-\g}\left(\frac{1}{2}-\int_{|\xi|\leq R}
  \psi_n(\tau,\xi)\, \d\xi \right),$$
and, by the H\"older inequality,
$$ \int_{|\xi|\leq R} \psi_n(\tau,\xi)\, \d\xi \leq
\left(\frac{|\S^{d-1}|\, R^d}{d}\right)^{p/(p-1)} \, \|\psi_n(\tau)\|_{L^p} \leq
\left(\frac{|\S^{d-1}|\, R^d}{d}\right)^{p/(p-1)} \, e^{CT}\,\|\psi_0\|_{L^p}.$$
Thus, (\ref{mom_2+gamma}) follows for $R$ small enough.
\end{proof}

We are now in a position to prove that moments of $\psi_n$ remain bounded
uniformly in $n\geq N_1$.

\begin{lem}\label{lemCT}
Let $T>0$ and $s>2$. Assume that $\|\psi_0\|_{L^1_s}<\infty$.
Then, there exists some constant $C$ depending only on $b(\cdot)$, $\a$, $d$, $\g$, $p$,
$s$, $T$, $\|\psi_0\|_{L^p}$ and  $\|\psi_0\|_{L^1_s}$ such that, for $n\geq N_1$,
\beq\label{mom_s}
\sup_{t\in[0,T]} \int_{\R^d} \psi_n(t,\xi)\, |\xi|^{s} \, \d\xi \leq C \quad
\mbox{ and } \quad\int_0^T \int_{\R^d} \psi_n(t,\xi)\,  \Phi_n(|\xi|) \,
|\xi|^s\, \d\xi \, \d t \leq C.
\eeq
\end{lem}

\begin{proof}
Let $s>2$ and $n\geq N_1$. Our proof follows the same lines as the proof
of \cite[Lemma 4.2]{MiWe99}. We use here the same notations as in the proof of
Lemma \ref{lem_2+gamma}. As previously, (\ref{eq_mom}) holds. Now, by
\cite[Lemma 11]{Lu00},  we have
$$K^n_s(\xi,\xi_*) \leq c_1 \, (|\xi|^{s-\g}\, |\xi_*| + |\xi|\, |\xi_*|^{s-\g})
-c_2(n) |\xi|^s,$$
for some constant $c_1$ depending only on $s$ and $d$ and
 $$c_2(n)= 2^{-s}\; \frac{s-2}{2}\; |\S^{d-2}| \int_0^\pi
 \mathbf{1}_{\{|\cos\theta|\leq 1-1/n\}}
\left(\min\{\cos\t ,1-\cos\t\}\right)^s \,b(\cos\t) \d\t. $$
Thus, by  (\ref{mino}), (\ref{maj_Phi_n}), (\ref{maj_wn}), (\ref{min_Phi_n})
and the above estimate, (\ref{eq_mom}) yields
\bean
\frac{\d}{\d t} Y^n_s(t) & \leq  & \frac{2 s}{d} \;\|b_n\|_{\L}\,Y^n_s(t)
\left(\int_{\R^d} |\xi|^2 \, \Phi_n(|\xi|) \, \psi_n(t,\xi)\, \d\xi \right)
+ \frac{s}{2} \;\|b_n\|_{\L}\,  Y^n_s(t)\, Y^n_\g(t) \\
& + & c_1 \int_{\R^d}\int_{\R^d} \psi_n(t,\xi)\, \psi_n(t,\xi_*)
\, ( |\xi|^\g+ |\xi_*|^\g) \, |\xi|^{s-\g}\, |\xi_*|  \, \d\xi\, \d\xi_* \\
& - & \frac{(1-\a)\, c_2(n)}{2} \int_{\R^d}\int_{\R^d} \psi_n(t,\xi)\,
\psi_n(t,\xi_*)  \, ( \Phi_n(|\xi|)- |\xi_*|^\g) \, |\xi|^s\, \d\xi\, \d\xi_*\,.
\eean
Consequently,
\bean
& & \hspace{-1cm} \frac{\d}{\d t} Y^n_s(t) + \frac{(1-\a)\, c_2(n)}{2}
\int_{\R^d} \psi_n(t,\xi)\,  \Phi_n(|\xi|) \, |\xi|^s\, \d\xi \\
&  & \leq   \frac{2 s}{d} \;\|b_n\|_{\L}\, Y^n_s(t)  \left(\int_{\R^d} |\xi|^2 \,
\Phi_n(|\xi|) \, \psi_n(t,\xi)\, \d\xi \right)
+  \frac{s\,\|b_n\|_{\L}+ c_2(n)}{2} \; Y^n_s(t)\,   Y^n_\g(t) \\
& &  \hspace{ 1cm} + c_1 \left( Y^n_s(t)\, Y^n_1(t) + Y^n_{s-\g}(t)\, Y^n_{1+\g}(t) \right),
\eean
but, for each $n\geq 2$,
$$0 < c_2(2)\leq c_2(n) \leq c_2^\infty:= 2^{-s}\; \frac{s-2}{2}\; |\S^{d-2}|
\int_0^\pi \left(\min\{\cos\t ,1-\cos\t\}\right)^s \,b(\cos\t) \d\t. $$
Hence, since $Y_{s-\gamma}^n(t) \leq Y_s^n(t)+1$, setting
$$h_n(t)= \frac{2 s}{d} \;\|b_n\|_{\L} \int_{\R^d} |\xi|^2 \, \Phi_n(|\xi|) \,
\psi_n(t,\xi)\, \d\xi + \frac{(s\, \|b_n\|_{\L}+c^\infty_2+4c_1)(1+d)}{2}$$
we obtain
$$ \frac{\d}{\d t} Y^n_s(t) + \frac{(1-\a)\, c_2(2)}{2} \int_{\R^d}
\psi_n(t,\xi)\,  \Phi_n(|\xi|) \, |\xi|^s\, \d\xi
\leq  h_n(t) \,  Y^n_s(t) +c_1\, (d+1).$$
Then, (\ref{mom_s}) follows easily from the Gronwall Lemma, \eqref{majo} and
Lemma \ref{lem_2+gamma}.
\end{proof}

 \begin{rmq} Applying the above to $s=2+\gamma$ and using \eqref{maj_Phi_n} one gets that
\begin{equation}\label{estQ-n}
\sup_{t \in [0,T]}\int_{\R^{d}}\Q_{-}^{n}\left(\psi_{n},\psi_{n}\right)(t,\xi)\,|\xi|^{2}\d \xi \leq C\end{equation}
for any $n \geq N_{1}$ and any $T > 0$ where  $C > 0$ is as in Lemma \ref{lemCT}. 
\end{rmq}

From the above Lemmas \ref{lemCT} and \ref{lem:lp} and the Dunford-Pettis Theorem, the set $\left\{\psi_n(t)\,,\,n \geq N_1\,\right\}$ is weakly relatively compact in $L^1(\R^d)$ for any $t \in [0,T]$. One can be more precise:

\begin{prop}\label{conv4} For any $T>0$, the sequence $(\psi_n)_{n\geq N_1}$ is relatively  compact in $\C([0,T];w-L^1(\R^d))$.
\end{prop}

\begin{proof} We follow here closely an approach already used in \cite{LLW03,Ba05}.
Let $T>0$. Due to \cite[Theorem 1.3.2]{vra}, since we already noticed that $\left\{\psi_n(t)\,,\,n \geq N_1\,\right\}$ is weakly relatively compact in $L^1(\R^d)$ for any $t \in [0,T]$, it suffices to check that
\begin{equation}\label{point1}
\mbox{ the family } (\psi_n)_{n \geq N_{1}}: [0,T] \tend L^1(\R^d)
\mbox{ is weakly equicontinuous}.\end{equation}
Let $\lambda \in L^\infty(\R^d)$.
There exists a sequence of functions $(\lambda_k)$ in $\C^1_c(\R^d)$ such that
\begin{equation}\label{pp}
\lambda_k(\cdot) \underset{k \to \infty}{\longrightarrow } \lambda(\cdot)  \quad  \mbox{ a.e. in } \R^d \quad \text{ and } \quad 
\sup_{k \geq 1}\| \lambda_k \|_{L^\infty}  \leq   \| \lambda \|_{L^\infty}
\end{equation}
We fix $\eta\in(0,1)$. From \eqref{estim_lp}, we deduce the existence of some
real $\omega(\eta)>0$ such that, for any measurable subset $E$ of $\R^d$,
\beq
\textrm{meas}(E) \leq \omega(\eta) \implies \sup_{n\geq N_1} \sup_{t\in [0,T ]} \int_E \psi_n(t,\xi) \, \d\xi \leq \eta.
\eeq
Moreover, Egorov theorem and (\ref{pp}) imply the
existence of a measurable subset $E_\eta$ of $B(0,1/\eta)$ such that
$$\mbox{meas }(E_\eta) \leq \omega(\eta) \qquad \mbox{ and } \qquad
\lim_{k\tend +\infty} \sup_{\xi \in B(0,1/\eta)\backslash E_\eta} |\lambda_k(\xi)-\lambda(\xi)| =0.$$
Consequently, for all $t\in (0,T)$, $h\in (-t,T-t)$ and $R\in(0,1/\eta]$,
we have
\bean
\left| \int_{\R^d} [\psi_n(t+h,\xi)-\psi_n(t,\xi)] \, \lambda(\xi) \, \d\xi \right|
& \leq & \left| \int_{\R^d} [\psi_n(t+h,\xi)-\psi_n(t,\xi)] \, \lambda_k(\xi) \, \d\xi\right| \\
& + & \left| \int_{|\xi|\leq R} [\psi_n(t+h,\xi)-\psi_n(t,\xi)] \, [\lambda(\xi)-\lambda_k(\xi)]\, \d\xi \right|\\
& + & \int_{|\xi|> R} [\psi_n(t+h,\xi)+\psi_n(t,\xi)] \, [|\lambda(\xi)|+|\lambda_k(\xi)|] \, \d\xi .
\eean
Thus, by the definition of $\omega(\eta)$, $E_\eta$ and $\lambda_k$, we deduce from \eqref{mom_s} that
\begin{multline}\label{equicont}
\left| \int_{\R^d} [\psi_n(t+h,\xi)-\psi_n(t,\xi)] \, \lambda(\xi) \, \d\xi \right|
\leq  \left| \int_{\R^d} [\psi_n(t+h,\xi)-\psi_n(t,\xi)] \, \lambda_k(\xi) \, \d\xi\right| \\
 + 2\, C\, \sup_{\xi \in B(0,R)\backslash E_\eta} |\lambda_k(\xi)-\lambda(\xi)| + 4 \,\|\lambda\|_{L^\infty}\, \eta + \frac{ 4 \, \|\lambda\|_{L^\infty}\,C }{R^2}.
\end{multline}
Let us now consider the first integral in the right-hand side of \eqref{equicont}. We infer from \eqref{annihi} that
\begin{multline*}
\frac{\d}{\d t} \int_{\R^d} \psi_n(t,\xi)\, \lambda_k(\xi)\, \d\xi
= (d \mathbf{B}_n(t)-\mathbf{A}_n(t) ) \, \int_{\R^d} \psi_n(t,\xi)\, \lambda_k(\xi)\, \d\xi \\
 + \mathbf{B}_n(t) \int_{\R^d} \psi_n(t,\xi)\,  \xi \cdot \nabla \lambda_k(\xi)\, \d\xi
+ \int_{\R^d} \mathbb{B}^n(\psi_n,\psi_n)(t,\xi)\,\lambda_k(\xi)\, \d\xi\,.
\end{multline*}
Now, by \eqref{mino}, \eqref{mom_s} and  inequalities \eqref{estim_2}, \eqref{maj_wn} and \eqref{estQ-n}, we have
$$0\leq d \mathbf{B}_n(t)-\mathbf{A}_n(t)
=\frac{\alpha}{ \|\psi_0^n\|_{L^1}} \int_{\R^d}\Q^n_-(\psi_n,\psi_n)(t,\xi)\,\d\xi
\leq 4\, \|b\|_{\L}\, (1+d), $$
\bean
|\mathbf{B}_n(t)|  & \leq&  \frac{\alpha}{2} \int_{\R^d} \left(\frac{1}{ \|\psi_0^n\|_{L^1}}+\frac{|\xi|^2}{\int_{\R^d}\psi_0^n(\xi_*)\,|\xi_*|^2\, \d\xi_* }\right) \, \Q^n_-(\psi_n,\psi_n)(t,\xi)\, \d\xi \\
& \leq & \frac{5}{2} \, \|b\|_{\L}\, (1+d)+ \frac{2 \,C}{d}\; \|b\|_{\L}
\eean
and
\bean
\left| \int_{\R^d} \mathbb{B}^n(\psi_n,\psi_n)(t,\xi)\, \lambda_k(\xi)\, \d\xi\right| & \leq & \|\lambda_k\|_{L^\infty}  \int_{\R^d}  \left(\Q^n_+(\psi_n,\psi_n)(t,\xi) +  \Q^n_-(\psi_n,\psi_n)(t,\xi) \right)  \d\xi \\
& \leq &  4\,\|\lambda_k\|_{L^\infty} \, \|b\|_{\L}\, (1+d).
\eean
Consequently,
\begin{multline*}
 \left| \int_{\R^d} [\psi_n(t+h,\xi)-\psi_n(t,\xi)] \, \lambda_k(\xi) \, d\xi\right| \\
\leq    |h|\, \|\lambda_k\|_{W^{1,\infty}} \,  \|b\|_{\L} \, (1+d)
\left(8 + \frac{5}{2} \; (1+d) +  \frac{2\, C}{d}  \right).
\end{multline*}
With the above estimate, we let $h\tend 0$ in \eqref{equicont} and obtain that
\begin{multline*}
\limsup_{h \tend 0} \sup_{n\geq N_1} \sup_{t \in(0,T)} \left| \int_{\R^d}
[\psi_n(t+h,\xi)-\psi_n(t,\xi)] \, \lambda(\xi) \, d\xi \right| \\
\leq 2 \, C \, \sup_{\xi \in B(0,R)\backslash E_\eta} |\lambda_k(\xi)-\lambda(\xi)| + 4 \, \|\lambda\|_{L^\infty} \, \eta
+ \frac{ 4 \, \|\lambda\|_{L^\infty}\,C }{R^2}.
\end{multline*}
We now pass to the successive limits $k \tend +\infty$, $\eta \tend 0$ and
$R\tend +\infty$ and deduce that (\ref{point1}) holds. Therefore, the
proof of Proposition \ref{conv4} is complete.
\end{proof}

\subsection{Well-posedness for the rescaled equation}

We are now in position to prove that the rescaled equation \eqref{BEscaled} is well-posed. Indeed, according to Proposition \ref{conv4}, up to a subsequence, the sequence  $(\psi_n)_{n\in \N}$ converges in $\mathcal{C}([0,T];w-L^1(\R^d))$ towards some limit $\psi=\psi(t,\xi) \in \mathcal{C}([0,T];w-L^1(\R^d))$.
One notices that, according to Lemma \ref{lemCT} and Fatou's Lemma,
$$\sup_{t\in[0,T]} \int_{\R^d} \psi(t,\xi)\, |\xi|^{2+\gamma} \, \d\xi \leq C,\quad \mbox{ and } \quad \int_0^T \d t \int_{\R^d} \psi(t,\xi)\,|\xi|^{2+2\gamma}\d\xi \leq C,$$
i.e.
$$\psi \in L^\infty(0,T;L^1_{2+\gamma}(\R^d)) \cap L^1(0,T;L^1_{2+2\gamma}(\R^d)).$$

The above estimates, together with Lemma \ref{lemCT}, the convergences of $(\psi_0^n)_{n\in \N}$ and $(\psi_n)_{n\in \N}$ enable us to pass to the limit in \eqref{weakformun} as in \cite[p. 860-861]{veronique}. We finally get that $\psi$ is indeed  a solution to the annihilation equation \eqref{BEscaled} in the sense of Definition \ref{defi:sol}. Notice moreover that, for any $T > 0$, the following holds
\bean
& & \lim_{n \to \infty} \int_0^T \left|\mathbf{B}_n(t)-\mathbf{B}_\psi(t)\right|\d t=0\,,\\
& &\lim_{n \to \infty} \sup_{t \in [0,T]}\left\|\Q_{\pm}^n(\psi_n,\psi_n)(t)-\Q_{\pm}(\psi,\psi)(t)\right\|_{L^1_s} =0 \qquad \forall 0 \leq s \leq 2.
\eean

Let us now tackle the problem of uniqueness. We take $p=2$ and assume that \eqref{hypini} holds. Then, weighted Sobolev norms propagate on \emph{finite time intervals}. More precisely, one has 
\begin{lem}\label{lem:sob}Let $\kappa>0$.
Let $\psi_0\in L^1_{2+\gamma}(\R^d)$ be a nonnegative function satisfying 
\eqref{hypini}.
If  $\psi \in L^\infty(0,T;L^1_{2+\gamma}(\R^d))\cap L^1(0,T;L^1_{2+2\gamma}(\R^d))$ denotes a solution to \eqref{BEscaled} with initial condition $\psi_0$ then
\beq\label{L2norm}
\sup_{t \in [0,T]} \|\psi(t)\|_{L^2_{\frac{9+d}{2}+\kappa}}< \infty
\eeq
while
\beq\label{H1norm}
\sup_{t\in [0,T]} \|\psi(t)\|_{\mathbb{H}^1_{3+\frac{d+\gamma+\kappa}{2}}} < \infty
\qquad \mbox{ and } \qquad \int_0^T  \|\psi(t)\|_{\mathbb{H}^1_{3+\gamma+\frac{d+\kappa}{2}}}\, \d t < \infty 
\eeq
\end{lem}

\begin{proof}

 For given $k > 0$, we multiply \eqref{BEscaled} by $2 \psi(t,\xi)\langle \xi \rangle^{2k}$ and integrate over $\R^d$. Then, one obtains, after an integration by parts, 
\begin{multline}\label{outil4}
 \frac{\d}{\d t} \|\psi(t)\|^2_{L^2_k}+ \left(2\mathbf{A}_{\psi}(t) -(d+2k)\mathbf{B}_{\psi}(t)\right)\|\psi(t)\|^2_{L^2_k} + 2k\mathbf{B}_{\psi}(t) \|\psi(t)\|_{L^2_{k-1}}^2\\
= 2(1-\alpha)\int_{\R^d}\Q_+(\psi,\psi)(t,\xi)\psi(t,\xi)\langle \xi\rangle^{2k}\d\xi -2\int_{\R^d} \Q_-(\psi,\psi)(t,\xi)\psi(t,\xi)\langle \xi\rangle^{2k}\d \xi.
\end{multline}
 First, since 
\beq\label{aide}
 |\xi-\xi_*|^\gamma\geq \langle\xi\rangle^\gamma-2\langle\xi_*\rangle^\gamma,
\eeq
we deduce that 
$$\int_{\R^d} \Q_-(\psi,\psi)(t,\xi)\psi(t,\xi)\langle \xi\rangle^{2k}\d \xi\geq \|\psi(t)\|^2_{L^2_{k+\gamma/2}} - 2\|\psi(t)\|_{L^1_\gamma}\|\psi(t)\|^2_{L^2_k}.$$
On the other hand, we have 
\beq\label{tictac}
\sup_{t\in[0,T]} \left|\mathbf{A}_{\psi}(t)\right|\leq C_T  \qquad \mbox{ and }\qquad  \sup_{t\in[0,T]} \left|\mathbf{B}_{\psi}(t) \right|   \leq C_T
\eeq
for some constant $C_T>0$.
Finally, proceeding as in the proof of Lemma \ref{lem:lp} (see also \cite{DeMou05}), we deduce that for any $\varepsilon>0$, there exists $C_\varepsilon >0$ that depends on $\sup_{t\in[0,T]}\|\psi(t)\|_{L^1_{2k+\gamma}}$ such that 
$$\int_{\R^d}\Q_+(\psi,\psi)(t,\xi)\psi(t,\xi)\langle \xi\rangle^{2k}\d\xi 
\leq \varepsilon \|\psi(t)\|^2_{L^2_{k+\gamma/2}}+C_\varepsilon \|\psi(t)\|^2_{L^2_k}.$$
Gathering the above estimates with $\varepsilon=1$, we get that there exists some constant $C>0$ depending on $\|\psi_0\|_{L^1_{2k+\gamma}}$ such that  
$$ \frac{\d}{\d t} \|\psi(t)\|^2_{L^2_k} + 2\alpha \|\psi(t)\|^2_{L^2_{k+\gamma/2}} \leq C \|\psi(t)\|^2_{L^2_k},$$
whence \eqref{L2norm} for $k=\frac{9+d}{2}+\kappa$.\\

Let us now consider the $\mathbb{H}^1_q$-norm. Let $j\in\{1,\ldots,d\}$. We set 
$G_j(t,\xi)=\partial_j\psi(t,\xi)$. Then, $G_j$ satisfies 
\beq\label{eqGj_1}
\partial_t G_j(t,\xi) + \left(\mathbf{A}_\psi(t) + \mathbf{B}_\psi(t)\right) G_j(t,\xi) + \mathbf{B}_\psi(t) \,\xi \cdot \nabla_\xi G_j(t,\xi)=\partial_j \mathbb{B}(\psi,\psi)(t,\xi).
\eeq
For given $q \geq 0$, we multiply this equation by $2\,G_j(t,\xi)\langle \xi \rangle^{2q}$ and integrate over $\R^d$. Then, one obtains, after an integration by parts and using \eqref{aide},
 \begin{multline}\label{eq_GJ}
\frac{\d}{\d t} \|G_j(t)\|^2_{L^2_q}+ \left(2 \mathbf{A}_{\psi}(t)+ (2-d-2q) \mathbf{B}_{\psi}(t)\right) \|G_j(t)\|^2_{L^2_q} + 2q \mathbf{B}_{\psi}(t)  \|G_j(t)\|^2_{L^2_{q-1}}  \\
\leq 2 (1-\alpha)\int_{\R^d}\partial_j\Q_+(\psi,\psi)(t,\xi)G_j(t,\xi)\langle \xi\rangle^{2q}\d\xi - 2\|G_j(t)\|^2_{L^2_{q+\gamma/2}}  \\ + 4 \|\psi(t)\|_{L^1_{\gamma}}\, \|G_j(t)\|^2_{L^2_q} 
-2\int_{\R^d} \Q_-(\psi,G_j)(t,\xi)G_j(t,\xi)\langle \xi\rangle^{2q}\d \xi.
 \end{multline}
Let us assume now that $q \geq \gamma/2$. One has
\begin{equation*}
\int_{\R^d} |\partial_j\Q_+(\psi,\psi)(t,\xi)|\,|G_j(t,\xi)|\, \langle \xi\rangle^{2q}\d\xi \leq \|\partial_j \Q_+(\psi,\psi)(t)\|_{L^2_{q-\gamma/2}}\|G_j(t)\|_{L^2_{q+\gamma/2}}.
\end{equation*}
One can extend  \cite[Theorem 2.7]{BaLo} (see also \cite[Theorem 2.5]{AloLo3}) to any $\gamma \in (0,1]$  and use the  $L^1_{\max\{2q+3/2-\gamma+\kappa,q+\gamma/2\}}$ and $L^2_{q+3/2-\gamma/2+\kappa}$ bounds to get, for any $\varepsilon  > 0$, the existence of some positive constants $C_1(\varepsilon,q) > 0$ and $C_2(q) > 0$ such that 
$$\|\partial_j \Q_+(\psi,\psi)(t)\|_{L^2_{q-\gamma/2}} \leq \|\Q_+(\psi,\psi)(t)\|_{\mathbb{H}^1_{q-\frac{\gamma}{2}}} \leq C_1(\varepsilon,q) +  \varepsilon\,C_2(q)\, \sum_{i=1}^d \|G_i\|_{L^{2}_{q+\frac{\gamma}{2}}}.$$
Moreover, using Cauchy-Schwarz inequality we obtain, 
\bean
\left|\int_{\R^d}  \Q_{-}(\psi,G_j)(t,\xi)\,G_j(t,\xi)\langle \xi\rangle ^{2q}\d \xi\right|
& \leq & \|\psi(t)\|_{L^2_{q+\gamma}} \,\|G_j(t)\|_{L^2_q}\,\|G_j(t)\|_{L^1_\gamma} \nonumber\\
& \leq & C_{q,\kappa}\|G_j(t)\|_{L^2_q} \|G_j(t)\|_{L^2_{\gamma+\frac{d+\kappa}{2}}}
\eean
for some positive constant $C_{q,\kappa}>0$. We have 
$\gamma+\frac{d+\kappa}{2} \leq 3 + \gamma + \frac{d+\kappa}{2}$.
 Thus, for $q=3+\frac{d+\gamma+\kappa}{2}$,
summing \eqref{eq_GJ} over all $j\in\{1,\ldots, d\}$, we get, thanks to \eqref{tictac} and the above estimates,  
\begin{multline*}
\frac{\d}{\d t} \sum_{j=1}^d\|G_j(t)\|^2_{L^2_{3+\frac{d+\gamma+\kappa}{2}}}
+  2  \sum_{j=1}^d\|G_j(t)\|^2_{L^2_{3+\gamma+\frac{d+\kappa}{2}}} \\
\leq \tilde{C}_1(q) \sum_{j=1}^d\|G_j(t)\|^2_{L^2_{3+\frac{d+\gamma+\kappa}{2}}}
+  \tilde{C}_2(\varepsilon,q) + \varepsilon \tilde{C}_3 \sum_{j=1}^d\|G_j(t)\|^2_{L^2_{3+\gamma+\frac{d+\kappa}{2}}}
  \end{multline*}
whence \eqref{H1norm}. 
\end{proof}

One then has the following stability result:
\begin{prop}\label{propostab} Let $T >0$ and let $\psi, \varphi \in \mathcal{C}([0,T];w-L^1) \cap L^\infty(0,T;L^1_{2+\gamma}) \cap L^1(0,T;L^1_{2+2\gamma})$ be two solutions to \eqref{BEscaled} with initial data $\psi_0,\varphi_0$ satisfying \eqref{massenergie} and \eqref{hypini}. Then, there exists $C_T >0$ such that
$$\|\psi(t)-\varphi(t)\|_{L^1_{2+\gamma}} \leq \|\psi_0-\varphi_0  \|_{L^1_{2+\gamma}}\exp(C_T) \qquad \forall t \in [0,T].$$
\end{prop}
\begin{proof} Since $\varphi,\psi \in   L^\infty(0,T;L^1_{2+\gamma})$, one has
$$\max_{t\in[0,T]}(|\mathbf{A}_\psi(t)|,|\mathbf{A}_\varphi(t)|,|\mathbf{B}_\psi(t)|,|\mathbf{B}_\varphi(t)|)\, \leq C_T < \infty.$$
Then, setting $F(t,\xi)=\psi(t,\xi)-\varphi(t,\xi)$,  multiplying by  $H(t,\xi)=\mathrm{sign}(F(t,\xi))(1+|\xi|^{2+\gamma})$ the equation satisfied by $F$ and integrating over $\R^d$, we get, for $t\in[0,T]$
$$\dfrac{\d}{\d t}\int_{\R^d}|F(t,\xi)| (1+|\xi|^{2+\gamma})\, \d\xi 
\leq K_T \int_{\R^d}|F(t,\xi)| (1+|\xi|^{2+\gamma})\, \d\xi + \mathcal{I}_{\psi,\varphi}^1(t)+\mathcal{I}_{\psi,\varphi}^2(t)
+\mathcal{I}_{\psi,\varphi}^3(t)$$
where $K_T >0$,
\begin{multline*}
\mathcal{I}_{\psi,\varphi}^1(t) = \left(\mathbf{A}_\varphi(t)-\mathbf{A}_\psi(t)\right)\int_{\R^d}\varphi(t,\xi)H(t,\xi)\d\xi\\
+  \left(\mathbf{B}_\varphi(t)-\mathbf{B}_\psi(t)\right)\int_{\R^d}\left(\xi \cdot \nabla \varphi(t,\xi)\right)H(t,\xi)\d\xi,
\end{multline*}
while
$$\mathcal{I}_{\psi,\varphi}^2(t)=- \alpha\int_{\R^d}\left(\Q_-(\psi,\psi)-\Q_-(\varphi,\varphi)\right)H(t,\xi)\d\xi$$
   and
   $$\mathcal{I}_{\psi,\varphi}^3(t)=(1-\alpha)\int_{\R^d} \left(\Q(\psi,\psi)-\Q(\varphi,\varphi)\right)H(t,\xi)\d\xi.$$
Thanks to Cauchy-Schwarz inequality, we have, for $\kappa>0$,  
\begin{multline*}
\mathcal{I}_{\psi,\varphi}^1(t) \leq \left|\mathbf{A}_\varphi(t)-\mathbf{A}_\psi(t)\right| \|\varphi(t)\|_{L^1_{2+\gamma}} 
+\left|\mathbf{B}_\varphi(t)-\mathbf{B}_\psi(t)\right|
 \|\nabla \varphi(t)\|_{L^{1}_{3+\gamma}}\\
 \leq \left|\mathbf{A}_\varphi(t)-\mathbf{A}_\psi(t)\right| \|\varphi(t)\|_{L^1_{2+\gamma}} 
+C_{\kappa}\left|\mathbf{B}_\varphi(t)-\mathbf{B}_\psi(t)\right|
 \|\varphi(t)\|_{\mathbb{H}^1_{3+\gamma+\frac{d+\kappa}{2}}}
 \end{multline*}
 where $C_{\kappa}:=\int_{\R^{d}}\langle \xi\rangle^{-d-\kappa}\d\xi < \infty.$ Thus, 
$$\mathcal{I}_{\psi,\varphi}^1(t) \leq \Lambda(t)\int_{\R^d}|F(t,\xi)| (1+|\xi|^{2+\gamma})\, \d\xi,$$
where $\Lambda\in L^1(0,T)$ by \eqref{H1norm}. Now, 
$$ 
(\Q_-(\psi,\psi)-\Q_-(\varphi,\varphi))(t,\xi)=F(t,\xi)L(\psi)(t,\xi)+\varphi(t,\xi)L(F)(t,\xi) 
$$
from which we deduce that  
$$\mathcal{I}_{\psi,\varphi}^2(t) \leq  - \alpha\int_{\R^d}  |F(t,\xi)|L(\psi)(t,\xi) (1+|\xi|^{2+\gamma})\,\d\xi +  c_\gamma \|\varphi(t)\|_{L^1_{2+\gamma}} \|F(t)\|_{L^1}.$$
Finally, proceeding as in \cite[Theorem 4.1]{MiWe99}, we get 
\begin{multline*} 
\mathcal{I}_{\psi,\varphi}^3(t) \leq (\|\psi(t)\|_{L^1_{\gamma}}+\|\varphi(t)\|_{L^1_{\gamma}}) \|F(t)\|_{L^1_\gamma}\\
+ \frac{1}{2}\int_{\R^d}\int_{\R^d}|\xi-\xi_*|^\gamma |F(t,\xi)|
(\varphi+\psi)(t,\xi_*) {\mathcal K}(\xi,\xi_*) \d\xi\d\xi_*, 
\end{multline*}
where 
$${\mathcal K}(\xi,\xi_*)   = \int_{\S^{d-1}}(|\xi'|^{2+\gamma}+|\xi'_*|^{2+\gamma}+|\xi_*|^{2+\gamma}-|\xi|^{2+\gamma})b(\cos \theta)\d\sigma.$$
We then deduce from  Lemma \ref{lemGPHI} derived in the next Section and \cite[Lemma 2]{BoGaPa} that
\bean 
{\mathcal K}(\xi,\xi_*) & \leq & \varrho_{1+\frac{\gamma}{2}} 
\left( (|\xi|^2+|\xi_*|^2)^{1+\gamma/2} - |\xi|^{2+\gamma} -|\xi_*|^{2+\gamma}\right)
-(1-\varrho_{1+\frac{\gamma}{2}}) |\xi|^{2+\gamma} + (1+\varrho_{1+\frac{\gamma}{2}})|\xi_*|^{2+\gamma} \\
& \leq & \varrho_{1+\frac{\gamma}{2}} \left(1+\frac{\gamma}{2}\right))
\left( |\xi|^2|\xi_*|^\gamma + |\xi|^\gamma |\xi_*|^2 \right)+ (1+\varrho_{1+\frac{\gamma}{2}})|\xi_*|^{2+\gamma}, 
\eean
where $\varrho_{1+\gamma/2}$ is defined subsequently by \eqref{varrhoK}.
Thus, using the estimate $|\xi-\xi_*|^\gamma\leq |\xi|^\gamma+|\xi_*|^\gamma$, we obtain 
$$\mathcal{I}_{\psi,\varphi}^3(t) \leq C \int_{\R^d}|F(t,\xi)| (1+|\xi|^{2+\gamma})\, \d\xi,  $$
for some constant $C>0$. We finally deduce from the above estimates that there
exists some function $\overline{\Lambda} \in L^1(0,T)$ such that
$$\dfrac{\d}{\d t}\int_{\R^d}|F(t,\xi)| (1+|\xi|^{2+\gamma})\, \d\xi 
\leq \overline{\Lambda}(t) \, \int_{\R^d}|F(t,\xi)| (1+|\xi|^{2+\gamma})\, \d\xi
 \qquad \forall t \in [0,T]$$
which gives the result.\end{proof}

\section{Moment estimates}\label{sec_mom}

We now prove uniform in time estimates of higher-order moments of the solution to \eqref{BEscaled} yielding to a proof of Theorem \ref{theomom}. We fix a nonnegative initial distribution $\psi_0$ satisfying \eqref{massenergie} and such that
$$\psi_0\in L^1_{2+\gamma}(\R^d)\cap L^p(\R^d)$$
for some $p >1$. Let then $\psi
\in\C([0,\infty);w-L^1(\R^d)) \cap L^\infty_{\mathrm{loc}}((0,\infty),L^1_{2+\gamma}(\R^d)) $
be a nonnegative solution to \eqref{BEscaled}-\eqref{CI}. We define, for any $k \geq 0$, the following moment of order $2k$:
$$M_k(t)=\int_{\R^d}\psi(t,\xi)\,|\xi|^{2k}\d\xi \qquad \qquad k \geq 0.$$
Using \eqref{BEscaled}, one easily gets that $M_k(t)$ satisfies the following identity
\begin{equation*}
\dfrac{\d}{\d t}M_k(t)=-\left(\mathbf{A}_\psi(t)-(d+2k)\mathbf{B}_\psi(t)\right)M_k(t) + \int_{\R^d}\mathbb{B}(\psi,\psi)(t,\xi)\,|\xi|^{2k}\d\xi, \qquad t >0.
\end{equation*}
Let us define
$$\mathbf{a}_\psi(t)=\int_{\R^d}\Q_-(\psi,\psi)(t,\xi)\d\xi \qquad \text{ and } \qquad \mathbf{b}_\psi(t)=\int_{\R^d}\Q_-(\psi,\psi)(t,\xi)\,|\xi|^2\,\d\xi$$
so that
$$\mathbf{A}_\psi(t)=-\dfrac{\alpha}{2}(d+2)\mathbf{a}_\psi(t)+\alpha\mathbf{b}_\psi(t) \qquad \text{ and } \qquad \mathbf{B}_\psi(t)=-\dfrac{\alpha}{2} \mathbf{a}_\psi(t)+\frac{\alpha}{d}\mathbf{b}_\psi(t).$$
Then, $M_k(t)$ satisfies
\begin{equation}\label{dmk}
\dfrac{\d}{\d t}M_k(t)+\a(k-1)\mathbf{a}_\psi(t)M_k(t)=\tfrac{2\a\,k}{d}\mathbf{b}_\psi(t)M_k(t) + \int_{\R^d}\mathbb{B}(\psi,\psi)(t,\xi)\,|\xi|^{2k}\d\xi.\end{equation}
In order to estimate in a precise way the last integral involving $\mathbb{B}(\psi,\psi)$, we shall resort to Povzner's estimates as derived in \cite{BoGaPa}.

\subsection{Povzner-type inequalities}
For any convex function $\Phi\::\:\R \to \R$, one has
\begin{equation}\label{mathB}
\int_{\R^d}\mathbb{B}(\psi,\psi)(t,\xi)\Phi(|\xi|^2)\d\xi=\int_{\R^{2d}}\psi(t,\xi)\psi(t,\xi_*)|\xi-\xi_*|^\gamma\mathcal{W}_\Phi(\xi,\xi_*)\d\xi\d\xi_*\end{equation}
where
\begin{equation}\label{K}\mathcal{W}_\Phi(\xi,\xi_*)=\dfrac{1}{2}
\int_{\S^{d-1}}\bigg[(1-\alpha)\Phi(|\xi'|^2)+(1-\alpha)\Phi(|\xi'_*|^2)-\Phi(|\xi|^2)-\Phi(|\xi_*|^2)\bigg]b(\cos \theta)\d\sigma.\end{equation}
Clearly
$$\mathcal{W}_\Phi(\xi,\xi_*)=(1-\alpha)\mathcal{G}_\Phi(\xi,\xi_*) -\frac{1}{2}\left(\Phi(|\xi|^2)+\Phi(|\xi_*|^2)\right)$$
with
$$\mathcal{G}_\Phi(\xi,\xi_*)=\dfrac{1}{2}\int_{\S^{d-1}}\bigg[ \Phi(|\xi'|^2)+ \Phi(|\xi'_*|^2)\bigg]b(\cos\theta)\d\sigma$$
where we recall that we assumed $\|b\|_{\L}=1.$ The following lemma allows to estimate $\mathcal{G}_\Phi(\xi,\xi_*)$ for any convex function $\Phi$.
\begin{lem}\label{lemGPHI} Let $\Phi\::\:\R \to \R$ be convex. Then,\begin{equation}\label{GPHI}
\mathcal{G}_\Phi(\xi,\xi_*)\leq \frac{1}{2} \int_{\S^{d-1}} \bigg[ \Phi\left(E\dfrac{1+ \hat{U} \cdot \sigma}{2}\right)+\Phi\left(E\dfrac{1- \hat{U} \cdot \sigma}{2}\right)\bigg]b(\hat{u} \cdot \sigma)\d\sigma\end{equation}
where, for any fixed $\xi,\xi_*$, we set
 $$U=\dfrac{\xi+\xi_*}{2},\qquad u=\xi-\xi_*, \qquad E=|\xi|^2+|\xi_*|^2, \qquad \hat{U}=U/|U|,\qquad \hat{u}=u/|u|.$$ 
\end{lem}
\begin{proof} We give a very short proof of the lemma, referring to \cite{BoGaPa} for the general strategy. For any fixed $\xi,\xi_*$, with the above notations one has $\cos \theta=\hat{u}\cdot \sigma$ and
 $$|\xi'|^2=E\dfrac{1+\lambda \hat{U} \cdot \sigma}{2} \qquad \text{ while } \qquad |\xi'_*|^2=E\dfrac{1-\lambda \hat{U} \cdot \sigma}{2}$$
where $\lambda=2\dfrac{|u|\,|U|}{E} \leq 1.$
Since $\Phi$ is convex, one can prove as in \cite{BoGaPa} that, for any fixed $x,y >0$, the mapping $t \mapsto \Phi(x+ty)+\Phi(x-ty)$ is nondecreasing and, because $\lambda \leq 1$, we have
\begin{equation*}\begin{split}
\Phi(|\xi'|^2)+ \Phi(|\xi'_*|^2)&=\Phi\left(E\dfrac{1+\lambda \hat{U} \cdot \sigma}{2}\right)+\Phi\left(E\dfrac{1-\lambda \hat{U} \cdot \sigma}{2}\right)\\
 &\leq  \Phi\left(E\dfrac{1+ \hat{U} \cdot \sigma}{2}\right)+\Phi\left(E\dfrac{1- \hat{U} \cdot \sigma}{2}\right).\end{split}\end{equation*}
Since $b(\cdot)$ is nonnegative, this gives \eqref{GPHI} after integration.\end{proof}
With the special choice $\Phi(x)=x^k$, $k \geq 1$, one has the following estimate
\begin{lem}\label{lemMk} For any $k \geq 1$, one has
$$\int_{\R^d}\mathbb{B}(\psi,\psi)(t,\xi)\,|\xi|^{2k}\d\xi \leq -(1-\beta_k(\a)) \, M_{k+\frac{\gamma}{2}}(t)+S_k(t)$$
with
\begin{multline*}
S_k(t)= \beta_k(\alpha)\sum^{[\frac{k+1}{2}]}_{j=1}\left(
\begin{array}{c}
k\\j
\end{array}
\right)\left(M_{j+\frac{\gamma}{2}}(t)\;M_{k-j}(t)+M_{j}(t)\;M_{k-j+\frac{\gamma}{2}}(t)\right)\\
+ (1-\beta_k(\a))M_k(t)\, M_{\frac{\gamma}{2}}(t)
\end{multline*}
where $[\frac{k+1}{2}]$ denote the integer part of $\frac{k+1}{2}$, $\beta_k(\a)=(1-\a)\varrho_k$ and
\begin{equation}\label{varrhoK}
\varrho_k = \sup_{\hat{U}, \hat{u} \in \S^{d-1}}  \int_{\S^{d-1}}\left[ \left(\dfrac{1+ \hat{U} \cdot \sigma}{2}\right)^k+\left(\dfrac{1- \hat{U} \cdot \sigma}{2}\right)^k\right]b(\hat{u}\cdot \sigma)\d\sigma.
\end{equation}
\end{lem}
\begin{proof} One applies the above estimate \eqref{GPHI} with the convex function $\Phi(x)=x^k$ to get
$$\mathcal{G}_\Phi(\xi,\xi_*)\leq\frac{1}{2}  \varrho_k\,E^k$$
where $E=|\xi|^2+|\xi_*|^2$.   One gets therefore
$$\mathcal{W}_\Phi (\xi,\xi_*) \leq -\frac{1}{2}\left(1-\beta_k(\a)\right)\left(|\xi|^{2k}+|\xi_*|^{2k}\right)+\frac{1}{2}\beta_k(\a)\big[\left(|\xi|^2+|\xi_\ast|^2\right)^k-|\xi|^{2k}-|\xi_*|^{2k}\big]$$
where $(1-\beta_k(\a)) >0$. Consequently,
\begin{multline}\label{Qp}
\int_{\R^d}\mathbb{B}(\psi,\psi)(t,\xi)\,|\xi|^{2k}\d\xi\leq  -(1-\beta_k(\a))\int_{\R^d}\psi(t,\xi)|\xi|^{2k}\d\xi\int_{\R^d}\psi(t,\xi_*)|\xi-\xi_*|^\gamma\d\xi_*\\
+\dfrac{\beta_k(\a)}{2}\int_{\R^{2d}}\psi(t,\xi)\psi(t,\xi_*)|\xi-\xi_*|^\gamma \big[\left(|\xi|^2+|\xi_*|^2\right)^k-|\xi|^{2k}-|\xi_*|^{2k}\big]\d\xi\d\xi_*.
\end{multline}
One then applies \cite[Lemma 2]{BoGaPa} with  $x=|\xi|^2$ and $y=|\xi_*|^2$ and uses the estimate
$$|\xi-\xi_*|^\gamma \leq |\xi|^\gamma+|\xi_*|^\gamma $$
to get
\begin{multline*}
\int_{\R^d}\mathbb{B}(\psi,\psi)(t,\xi)\,|\xi|^{2k}\d\xi \leq  -(1-\beta_k(\a))\int_{\R^d}\psi(t,\xi)|\xi|^{2k}\d\xi\int_{\R^d}\psi(t,\xi_*)|\xi-\xi_*|^\gamma\d\xi_*\\
+ \beta_k(\alpha)\sum^{[\frac{k+1}{2}]}_{j=1}\left(
\begin{array}{c}
k\\j
\end{array}
\right)\left(M_{j+\gamma/2}(t)\;M_{k-j}(t)+M_{j}(t)\;M_{k-j+\gamma/2}(t)\right).
\end{multline*}
To estimate the nonpositive term, one notices that
$$|\xi-\xi_*|^\gamma\geq  \,|\xi|^\gamma-|\xi_*|^\gamma$$
and gets
 $$\int_{\R^d}\psi(t,\xi)|\xi|^{2k}\d\xi\int_{\R^d}\psi(t,\xi_*)|\xi-\xi_*|^\gamma\d\xi_* \geq  M_{k+\frac{\gamma}{2}}(t)-M_k(t)\, M_{\frac{\gamma}{2}}(t).$$
This clearly yields the conclusion.
\end{proof}
\begin{rmq}It is easy to check that $\varrho_1 = \|b\|_{\L}= 1$ and that the mapping $k > 1 \mapsto \varrho_k \geq 0$ is strictly decreasing.
\end{rmq}
\subsection{Uniform estimates} Thanks to the above lemma, we can derive uniform in time estimates of $M_k(t)$ for $k=1+\frac{\g}{2}$. Precisely, one has the following:
\begin{prop}\label{theoMom} Let $$\alpha_0=\frac{1-\varrho_{1+\frac{\g}{2}}}{1+\frac{\g}{2}-\varrho_{1+\frac{\g}{2}}} \in (0,1]$$ where $\varrho_k$ is defined by \eqref{varrhoK} for any $k \geq 1.$ Then, if $0 <\a < \alpha_0$, there exists a constant $\overline{M}$ depending only on $\alpha$, $\gamma$, $b(\cdot)$ and $d$ such that any solution $\psi(t)$ to \eqref{BEscaled} satisfies
$$\sup_{t \geq 0}M_{1+\frac{\g}{2}}(t) \leq \max\left\{M_{1+\frac{\g}{2}}(0),\overline{M}\right\}.$$
\end{prop}
\begin{proof} Let us fix $k >1$.  Since $\mathbf{a}_\psi(t) \geq 0$, one gets from \eqref{dmk}:
$$\dfrac{\d}{\d t}M_k(t) \leq \tfrac{2\a\,k}{d}\mathbf{b}_\psi(t)M_k(t) + \int_{\R^d}\mathbb{B}(\psi,\psi)(t,\xi)\,|\xi|^{2k}\d\xi.$$
Now, we recall that
$$\mathbf{b}_\psi(t)= \int_{\R^d\times\R^d}|\xi-\xi_*|^\gamma \psi(t,\xi)\psi(t,\xi_*)|\xi|^2\d\xi\d\xi_*$$
so that, since $|\xi-\xi_*|^\gamma \leq  |\xi|^\gamma+|\xi_*|^\gamma $, one has
$$\mathbf{b}_\psi(t) \leq  M_{1+\frac{\gamma}{2}}(t)+M_{\frac{\gamma}{2}}(t)M_1(t) \leq
 M_{1+\frac{\gamma}{2}}(t)+\frac{d}{2}(1+\frac{d}{2})$$
 where we recall that $M_1(t)=M_1(0)=\frac{d}{2}$ for any $t \geq 0.$ We get therefore
\begin{equation*}\dfrac{\d}{\d t}M_k(t) \leq  \frac{2\a\,k}{d} M_{1+\frac{\gamma}{2}}(t)M_k(t)+  \a k(1+\frac{d}{2}) M_k(t)+\int_{\R^d}\mathbb{B}(\psi,\psi)(t,\xi)\,|\xi|^{2k}\d\xi.\end{equation*}
Now, one estimates the last integral thanks to Lemma \ref{lemMk} and get
\begin{multline}\label{dMp1}
\dfrac{\d}{\d t}M_k(t)+ (1-\beta_k(\a)) \, M_{k+\frac{\gamma}{2}}(t) \leq S_k(t) + \frac{2\a\,k}{d} M_{1+\frac{\gamma}{2}}(t)M_k(t)
+ \a k(1+\frac{d}{2}) M_k(t).
\end{multline}
Using now H\"{o}lder's inequality, one has, for $k\geq 1+\frac{\gamma}{2}$,
$$M_{k+\frac{\gamma}{2}}(t)\geq \left(\frac{2}{d}\right)^{\frac{\gamma}{2k-2}} \,
  \left(M_k(t)\right)^{\frac{2k+\gamma-2}{2k-2}}  \qquad  \mbox{ and } \qquad
M_{1+\frac{\gamma}{2}}(t)\leq \left(\frac{2}{d}\right)^{{-1+\frac{\gamma}{2k-2}}} \,
\left(M_k(t)\right)^{\frac{\gamma}{2k-2}}$$
where we used again that $M_1(t)=\frac{d}{2}$ for any $t\geq 0.$ With these estimates, \eqref{dMp1} becomes
\begin{equation}\label{dMp2}
\dfrac{\d}{\d t}M_k(t)+  c_{\a,k,d}
\left(\frac{2}{d}\right)^{\frac{\gamma}{2k-2}}  \left(M_{k}(t)\right)^{1+\frac{\gamma}{2k-2}}
 \leq S_k(t) +  \a k(1+\frac{d}{2}) M_k(t),\end{equation}
 with
$$ c_{\a,k,d}=1-\beta_k(\a)-\a k=1-\varrho_k+\a(\varrho_k-k).$$
Notice that
\begin{equation}\label{c>0} c_{\a,k,d} >0 \qquad \Longleftrightarrow \qquad 0 < \a < \frac{1-\varrho_k}{k-\varrho_k}.\end{equation}
Taking now $k=1+\frac{\g}{2}$ in the above inequality  \eqref{dMp2} and using the explicit expression of $S_{1+\frac{\g}{2}}(t)$   we find
\begin{equation*}\begin{split}
\dfrac{\d}{\d t}M_{1+\frac{\g}{2}}(t)+ c_{\a,1+\frac{\g}{2},d}\left(\frac{2}{d}\right) M_{1+\frac{\g}{2}}(t)^{2}
\leq \beta_{1+\frac{\g}{2}}(\a)&\left(\begin{array}{c}
1+\frac{\g}{2}\\1
\end{array}
\right)\left(M_{1+\frac{\g}{2}}(t)M_{\frac{\g}{2}}(t)+M_1(t)M_\g(t)\right) \\
&+(1-\beta_{1+\frac{\g}{2}}(\a))M_{1+\frac{\g}{2}}(t)M_{\frac{\g}{2}}(t)\\
&+ \a (1+\frac{\g}{2})(1+\frac{d}{2}) M_{1+\frac{\g}{2}}(t).
\end{split}\end{equation*}
Since $\gamma \leq 1$ and $M_1(t)=\frac{d}{2}$ for any $t \geq 0$, it is clear that $M_{\frac{\g}{2}}(t)$ and $M_\g(t)$ are uniformly bounded by $1+\frac{d}{2}$ so that there are two positive constants $C_0,C_1 >0$  depending only on $\alpha$, $\gamma$, $b(\cdot)$ and $d$ such that
$$\dfrac{\d}{\d t}M_{1+\frac{\g}{2}}(t)+ c_{\a,1+\frac{\g}{2},d}\left(\frac{2}{d}\right) M_{1+\frac{\g}{2}}(t)^{2} \leq C_0 M_{1+\frac{\g}{2}}(t)+C_1 \qquad \forall t \geq 0.$$
 Therefore, using \eqref{c>0} and some comparison principle, we get the conclusion.
\end{proof}
\begin{rmq}\label{rmqalpha0} The parameter $\alpha_0$ depends only on $\g,d$ and the collision kernel $b(\cdot)$. In particular, in dimension $d=3$, for constant collision kernel $b(\cdot)=\dfrac{1}{4\pi}$ (recall that $\|b\|_{\L}=1$) and with $\gamma=1$, one has $\varrho_{\frac{3}{2}}=\dfrac{4}{5}$ and $\alpha_0=\dfrac{2}{7}.$
\end{rmq}

Notice that the above result allows actually to deal with higher-order moments:
\begin{cor} With the notations of the above proposition, if $0 <\a < \alpha_0$ then any solution $\psi(t)$ to \eqref{BEscaled} satisfies for any $k \geq 1+\frac{\g}{2}$
\beq\label{h_o_mom}M_k(0) < \infty \Longrightarrow \sup_{t \geq 0}M_{k}(t) < \infty.\eeq
\end{cor}
\begin{proof} The strategy follows classical arguments already used in \cite{BoGaPa}, the crucial point being that, for $k \geq 1+\frac{\g}{2}$, the first term in the expression of $S_k(t):$
\begin{multline*}
S_k(t)= \beta_k(\alpha)\sum^{[\frac{k+1}{2}]}_{j=1}\left(
\begin{array}{c}
k\\j
\end{array}
\right)\left(M_{j+\frac{\gamma}{2}}(t)\;M_{k-j}(t)+M_{j}(t)\;M_{k-j+\frac{\gamma}{2}}(t)\right)\\
+ (1-\beta_k(\a))M_k(t)\, M_{\frac{\gamma}{2}}(t)
\end{multline*}
involves only moments of order less than $\max\{k-1+\frac{\g}{2}, [\frac{k+1}{2}]+\frac{\gamma}{2}\}\leq \max\{k-\frac{1}{2}, [\frac{k+1}{2}]+\frac{\gamma}{2}\}$ since $\g \leq 1.$

First observe that mass is conserved and thus, using classical interpolation, it suffices to prove the result for any $k \geq 1+\frac{\g}{2}$ such that $2k\in \mathbb{N}$. We proceed by induction. Since $\gamma\in(0,1]$, the first step consists in checking that the result holds for $k=3/2$. We shall come back to this point later on. Let $k>3/2$ such that $2k\in\N$. Let us assume that for any $j$ satisfying $2j\in\N$ and $1\leq j\leq k-1/2$, there exists $K_j>0$ such that $M_j(t)\leq K_j$ for any $t\geq 0$.
Note that for such a $k$, then $\max\{k-\frac{1}{2}, [\frac{k+1}{2}]+\frac{\gamma}{2}\}=k-\frac{1}{2}$. Consequently, the induction hypothesis together with the fact that $M_{\frac{\g}{2}}(t)$ is uniformly bounded imply that
$$S_k(t) \leq C_k + A_k M_k(t)$$
with $A_k=(1-\beta_k(\a))(1+\frac{d}{2})$ and
$$C_k=\beta_k(\alpha)\sum^{[\frac{k+1}{2}]}_{j=1}\left(
\begin{array}{c}
k\\j
\end{array}
\right)\left(K_{j+\frac{\gamma}{2}} \;K_{k-j} +K_{j} \;K_{k-j+\frac{\gamma}{2}}\right).$$
Then, from \eqref{dMp1}:
$$\dfrac{\d}{\d t}M_k(t)+ (1-\beta_k(\a)) \, M_{k+\frac{\gamma}{2}}(t) \leq C_k + \left(A_k+\a k(1+\frac{d}{2})\right) M_k(t) + \frac{2\a\,k}{d} M_{1+\frac{\gamma}{2}}(t)M_k(t).$$
Now, from Theorem \ref{theoMom}, as soon as $\alpha \in (0,\alpha_0)$, $\sup_{t \geq 0}M_{1+\frac{\g}{2}}(t) < \infty$ and the above identity becomes
$$\dfrac{\d}{\d t}M_k(t)+ (1-\beta_k(\a)) \, M_{k+\frac{\gamma}{2}}(t) \leq C_k + B_k\,M_k(t)$$
for some explicit constant $B_k >0$. From Jensen's inequality, one has
$$M_{k+\frac{\gamma}{2}}(t) \geq \left(M_k(t)\right)^{1+\frac{\g}{2k}}$$
from which the above differential inequality yields the conclusion.

It only remains to check that \eqref{h_o_mom} holds for $k=3/2$. If $\gamma=1$, it directly follows from Theorem \ref{theoMom}. Otherwise, we have $\max\{k-\frac{1}{2}, [\frac{k+1}{2}]+\frac{\gamma}{2}\}=\max\{1, [\frac{5}{4}]+\frac{\gamma}{2}\}=1+\frac{\gamma}{2}$ and we deduce from Theorem \ref{theoMom} and usual interpolations that $$S_{3/2}(t) \leq C_{3/2} + A_{3/2} M_{3/2}(t),$$
for some constants $C_{3/2}>0$ and $A_{3/2}>0$, which leads, following the same lines as above, to the desired result.
\end{proof}

\subsection{Lower bounds}

We shall now use Lemma \ref{lemGPHI} to derive suitable  lower bounds for the moments of $\psi(t,\xi)$:
\begin{lem}\label{leminf} For any $\gamma \in (0,1]$, there exists $\alpha_\star \in (0,1)$ such that, for any $\alpha \in (0,\alpha_\star)$  any solution $\psi(t,\xi)$ to \eqref{BEscaled} satisfies
\begin{equation}\label{CAl}
\int_{\R^d}  \psi(t,\xi_*) |\xi_*|^\gamma \d\xi_* \geq \mathbf{C}_\alpha \int_{\R^d}  \psi_0(\xi_*) |\xi_*|^\gamma \d\xi_*.
\end{equation}
for some explicit constant  $\mathbf{C}_\alpha >0$ depending only on $\a,\g,d$ and $b(\cdot)$. Moreover, one has the following propagation of lower bounds
\begin{enumerate}[i)]
\item Assume that $\gamma=1$ and, given $0< \a <\a_\star$, let $0<\kappa(\alpha)\leq \sqrt{\left(\dfrac{\beta_{\frac{1}{2}}(\a)-1}{\beta_{\frac{1}{2}}(\a)+1}\right)\dfrac{d}{2}}$. If $M_{\frac{1}{2}}(0) \geq \kappa(\a)$ then $M_{\frac{1}{2}}(t) \geq \kappa(\a)$ for any $t \geq 0.$
\item Assume that $\gamma \in (0,1)$ and let $j_0 \in \mathbb{N}$ be such that $k_0=\frac{j_0\gamma}{2} < 1$ and $k_0+\frac{\gamma}{2} \geq 1.$ Given $0< \a <\a_\star$ let $(\kappa_j(\a))_{j=1,\ldots,j_0}$ be some positive constants such that
$$\kappa_{j_0}(\a) \leq \left( \frac{\beta_{\frac{j_0\gamma}{2}}(\alpha)-1}{\beta_{\frac{j_0\gamma}{2}}(\alpha)+1} \right)^{\frac{j_0}{1+j_0}}\left(\frac{d}{2}\right)^{\frac{j_0\gamma}{2}} \qquad \mbox{ and } \qquad
\kappa_j(\a)\leq \left( \frac{\beta_{\frac{j\gamma}{2}}(\alpha)-1}{\beta_{\frac{j\gamma}{2}}(\alpha)+1}\; \kappa_{j+1}(\alpha)\right)^\frac{j}{1+j},$$
for $j=1,\ldots,j_0-1.$
If the initial datum $\psi_0$ is such that $M_{\frac{j\gamma}{2}}(0) \geq \kappa_j(\a)$ for any $j=1,\ldots,j_0$ then $\inf_{t \geq 0} M_{\frac{j\gamma}{2}}(t) \geq \kappa_j(\a)$ for any $j=1,\ldots,j_0$.
\end{enumerate}
\end{lem}
\begin{proof} We first prove \eqref{CAl}. We estimate the  moment  $M_k(t)$ for $k < 1$ applying the above Lemma \ref{lemGPHI}  to the convex function $\Phi(x)=-x^k$. We obtain easily that
\begin{multline*}
-\int_{\R^d}\mathbb{B}(\psi,\psi)(t,\xi)\,|\xi|^{2k}\d\xi \leq -\dfrac{\beta_k(\a)}{2}\int_{\R^{2d}}\psi(t,\xi)\psi(t,\xi_*)|\xi-\xi_*|^\gamma\,\left(|\xi|^2+|\xi_*|^2\right)^k\d\xi\d\xi_*\\
+\dfrac{1}{2}\int_{\R^{2d}}\psi(t,\xi)\psi(t,\xi_*)|\xi-\xi_*|^\gamma\,\left(|\xi|^{2k}+|\xi_*|^{2k}\right)\d\xi\d\xi_*\end{multline*}
where, as in Lemma \ref{lemMk}, $\beta_k(\a)=(1-\a)\varrho_k$ with $\varrho_k$ given by
 \begin{equation*}
\varrho_k=\sup_{\hat{U}, \hat{u} \in \S^{d-1}}\int_{\S^{d-1}}\left[ \left(\dfrac{1+ \hat{U} \cdot \sigma}{2}\right)^k+\left(\dfrac{1- \hat{U} \cdot \sigma}{2}\right)^k\right]b(\hat{u} \cdot \sigma)\d\sigma \qquad \forall 0 < k < 1.
\end{equation*} Using the fact that  $k-1 < 0$, $\mathbf{a}_\psi(t) \geq 0$ and $\mathbf{b}_\psi(t) \geq 0$, we deduce from \eqref{dmk} that
\begin{equation*}\dfrac{\d}{\d t}M_k(t)\geq \frac{1}{2}\int_{\R^{2d}}\psi(t,\xi)\psi(t,\xi_*) \mathcal{J}_k(\xi,\xi_*)\d\xi\d\xi_*
\end{equation*}
where
$$\mathcal{J}_k(\xi,\xi_*)=\beta_k(\a)|\xi-\xi_*|^\gamma\,\left(|\xi|^2+|\xi_*|^2\right)^k -|\xi-\xi_*|^\gamma\,\left(|\xi|^{2k}+|\xi_*|^{2k}\right).$$
Since $\gamma \in (0,1]$, one has $|\,|\xi|^\gamma-|\xi_*|^\gamma\,| \leq |\xi-\xi_*|^\gamma\leq |\xi|^\gamma+|\xi_*|^\gamma$ while $$\left(|\xi|^2+|\xi_*|^2\right)^k \geq \left|\,|\xi|^{2k}-|\xi_*|^{2k}\,\right| \qquad \forall k \in (0,1).$$
As a consequence,
\begin{multline*}
\mathcal{J}_k(\xi,\xi_*) \geq \beta_k(\a)\left(\,|\xi|^\gamma-|\xi_*|^\gamma\,\right)\left(\,|\xi|^{2k}-|\xi_*|^{2k}\,\right) -\left(|\xi|^\gamma+|\xi_*|^\gamma\right)\,\left(|\xi|^{2k}+|\xi_*|^{2k}\right)\\
=\left(\beta_k(\a)-1\right)\left(|\xi|^{\gamma+2k}+|\xi_*|^{\gamma+2k}\right)-\left(\beta_k(\a)+1\right)\left(|\xi|^\gamma\,|\xi_*|^{2k}+|\xi_*|^\gamma\,|\xi|^{2k}\right).
\end{multline*}
yielding the following inequality, for any $0 < k < 1$:
\begin{equation}\label{dMkt}\dfrac{\d}{\d t}M_k(t) \geq \left(\beta_k(\a)-1\right)M_{k+\frac{\gamma}{2}}(t)-\left(\beta_k(\a)+1\right)M_{\frac{\gamma}{2}}(t)M_{k}(t).\end{equation}
We are now in position to resume the argument of \cite[Lemma 2]{GaPaVi2} to get \eqref{CAl}.  We recall here the main steps in order to explicit the parameter $\alpha_\star$ (and, for $\g=1$, the constant $\mathbf{C}_\alpha$). Assume first that $\gamma=1$, using then \eqref{dMkt} with $k=\frac{1}{2}$, we get
$$\dfrac{\d}{\d t}M_{\frac{1}{2}}(t) \geq \left(\beta_{\frac{1}{2}}(\a)-1\right)M_{1}(t)-\left(\beta_{\frac{1}{2}}(\a)+1\right)M_{\frac{1}{2}}(t)^2.$$
Since $M_1(t)=M_1(0)=d/2$ for any $t \geq 0$, we see that, if $\beta_{\frac{1}{2}}(\a) -1 >0 $ then
 \begin{equation}\label{lowerM12}
 M_{\frac{1}{2}}(t) \geq \min\left(M_{\frac{1}{2}}(0), \sqrt{\dfrac{\beta_{\frac{1}{2}}(\a)-1}{\beta_{\frac{1}{2}}(\a)+1}M_1(0)}\right) \qquad \forall t \geq 0.\end{equation}
Since moreover $M_1(0) \geq M_{\frac{1}{2}}(0)^2$ we obtain
$$M_{\frac{1}{2}}(t) \geq \mathbf{C}_\alpha M_{\frac{1}{2}}(0) \qquad \forall 0 < \alpha < \alpha_\star:=\dfrac{\varrho_{\frac{1}{2}}-1}{\varrho_{\frac{1}{2}}}$$
where $\mathbf{C}_\alpha=\sqrt{\dfrac{\beta_{\frac{1}{2}}(\a)-1}{\beta_{\frac{1}{2}}(\a)+1}}$ (notice that $0 < \alpha < \alpha_\star \Longleftrightarrow \beta_{\frac{1}{2}}(\a) >1$). In other words, for any $0 < \alpha < \alpha_\star$,
$$\int_{\R^d}|\xi|\psi(t,\xi)\d\xi \geq \mathbf{C}_\alpha\int_{\R^d}|\xi|\psi_0(\xi)\d\xi \qquad \forall t \geq 0.$$
For $\gamma < 1$, one argues by induction as in \cite[Lemma 2]{GaPaVi2} iterating the above argument with $k=\frac{j\gamma}{2}$ for $j=1,\ldots,j_0$ where $j_0 \in \mathbb{N}$ is such that $k_0=\frac{j_0\gamma}{2} < 1$ and $k_0+\frac{\gamma}{2} \geq 1.$ Then, from \eqref{dMkt} with $k=k_0$, we get
\begin{equation*}
\dfrac{\d}{\d t}M_{k_0}(t) \geq \left(\beta_{k_0}(\a)-1\right)M_{k_0+\frac{\gamma}{2}}(t)-\left(\beta_{k_0}(\a)+1\right)M_{\frac{\gamma}{2}}(t)M_{k_0}(t).\end{equation*}
A simple use of Jensen's inequality shows that
$$\dfrac{\d}{\d t}M_{k_0}(t) \geq \left(\beta_{k_0}(\a)-1\right)\left(\frac{d}{2}\right)^{k_0+\frac{\gamma}{2}}-\left(\beta_{k_0}(\a)+1\right)M_{k_0}(t)^{1+\frac{\gamma}{2k_0}}$$
from which we deduce, as above, that
$$M_{k_0}(t) \geq \left(\frac{\beta_{k_0}(\alpha)-1}{\beta_{k_0}(\a)+1}\right)^{\frac{1}{1+\frac{\gamma}{2k_0}}}M_{k_0}(0) \qquad \forall t \geq 0$$
if $\beta_{k_0}(\alpha)>1.$ Now, one can repeat the argument exactly  with $k_1=k_0-\frac{\gamma}{2}$, $k_2=k_1-\frac{\gamma}{2}$ and so on. Notice that, if $\beta_{k_0}(\alpha) >1$, then $\beta_{k}(\a) > 1$ for any $k \leq k_0$. In particular, we get \eqref{CAl} for any $0 < \alpha < \frac{\varrho_{k_0}-1}{\varrho_{k_0}}=:\alpha_\star.$

Let us now prove the second part of the lemma, regarding the propagation of lower bounds. The proof in the case $\gamma=1$ is a direct consequence of \eqref{lowerM12}. For $0 < \gamma < 1$, the proof uses arguments similar to those used in the proof of \eqref{CAl}. Precisely, since  $M_{\frac{\gamma}{2}}(t) \leq M_{\frac{j\gamma}{2}}(t)^{\frac{1}{j}}$ according to  Jensen's inequality, one deduces from Eq. \eqref{dMkt} that
\begin{equation*}
\dfrac{\d}{\d t}M_{\frac{j\gamma}{2}}(t) \geq \left(\beta_{\frac{j\gamma}{2}}(\a)-1\right)M_{\frac{(j+1)\gamma}{2}}(t)-\left(\beta_{\frac{j\gamma}{2}}(\a)+1\right)M_{\frac{j\gamma}{2}}(t)^\frac{1+j}{j} , \qquad \text{ for any  } j=1,\ldots,j_0.\end{equation*}
According to Jensen's inequality one also has $$M_{\frac{(j_0+1)\gamma}{2}}(t)\geq M_1(t)^{\frac{(j_0+1)\gamma}{2}}=\left(\frac{d}{2}\right)^\frac{(j_0+1)\gamma}{2} \qquad \forall t \geq 0$$
 and, by a simple decreasing induction argument, one checks that if $M_{\frac{j\gamma}{2}}(0) \geq \kappa_j(\a)$ holds for any $j=1,\ldots,j_0$, then $\inf_{t \geq 0} M_{\frac{j\gamma}{2}}(t) \geq \kappa_j(\a)$ will hold  for any $j=1,\ldots,j_0.$
\end{proof}
\begin{rmq}\label{Cgamma} With the notations of Lemma \ref{leminf}, we define the set $\mathcal{C}_\gamma(\alpha)$ $(0 < \a < \a_\star)$ as follows:
\begin{enumerate}[(i)\:]
\item If $\gamma=1$ then $\mathcal{C}_1(\alpha)$ is the set of nonnegative $\psi(\xi)$ such that $\int_{\R^d} \psi(\xi)|\xi|\d\xi  \geq \kappa(\alpha)$.
\item If $\gamma \in (0,1)$ let $j_0 \in \mathbb{N}$ be such that $k_0=\frac{j_0\gamma}{2} < 1$ and $k_0+\frac{\gamma}{2} \geq 1.$ Then,
$\mathcal{C}_\gamma(\alpha)$ is defined as the set of nonnegative $\psi(\xi)$ such that $\int_{\R^d}\psi(\xi)|\xi|^{j\gamma} \d\xi \geq \kappa_j(\alpha)$ for any $j=1,\ldots,j_0.$\end{enumerate}
The second part of Lemma \ref{leminf} can be reformulated as follows: given $\gamma \in (0,1]$ and $0 < \alpha < \alpha_\star$, if the initial datum $\psi_0 \in \mathcal{C}_\gamma(\alpha)$ then the associated solution $\psi(t)$ to \eqref{BEscaled} is such that $\psi(t) \in \mathcal{C}_\gamma(\alpha)$ for any $t \geq 0.$%
\end{rmq}

The above lower bounds have several important consequences when dealing with isotropic functions. Precisely, one has the following result, already stated in \cite[Lemma 10]{Lu00} in dimension $d=3$:
\begin{lem}\label{lem:LU} Assume that $f(\xi)=\overline{f}(|\xi|) \geq 0$ is an isotropic integrable function and let $k(r) \geq 0$ be a nondecreasing mapping on $[0,\infty).$ Then, for any $\xi \in \R^d$,
$$\int_{\R^d} f(\xi_*)k\left(|\xi-\xi_*|\right)\d\xi_* \geq \frac{1}{2}\int_{\R^d} f(\xi_*)k\left(\sqrt{|\xi|^2+|\xi_*|^2}\right)\d\xi_*.$$
\end{lem}
\begin{proof}  We give an elementary proof of this result. Using spherical coordinates, with $\xi_*=\varrho\omega$ and $\xi=r\sigma$, $r,\varrho >0$, $\omega,\sigma \in \mathbb{S}^{d-1}$, one has
\begin{equation*}\begin{split}
\int_{\R^d} f(\xi_*)k\left(|\xi-\xi_*|\right)\d\xi_*&=\int_0^\infty \overline{f}(\varrho)\varrho^{d-1}\d\varrho\int_{\mathbb{S}^{d-1}} k\left(\sqrt{\varrho^2+r^2-2r\,\varrho\,  \sigma \cdot \omega}\right)\d\omega\\
&\geq \int_0^\infty \overline{f}(\varrho)\varrho^{d-1}\d\varrho\int_{\mathbb{S}^{d-1}_-} k\left(\sqrt{\varrho^2+r^2-2r\,\varrho\,  \sigma \cdot \omega}\right)\d\omega
\end{split}\end{equation*}
where $\mathbb{S}^{d-1}_-=\left\{\omega \in \mathbb{S}^{d-1}\,;\,\sigma \cdot \omega < 0\right\}.$ Then, for any $\omega \in \mathbb{S}^{d-1}_-$, since $k(\cdot)$ is nondecreasing,
$$k\left(\sqrt{\varrho^2+r^2-2r\,\varrho\,  \sigma \cdot \omega}\right)\geq k\left(\sqrt{\varrho^2+r^2}\right)$$
and
$$\int_{\R^d} f(\xi_*)k\left(|\xi-\xi_*|\right)\d\xi_*\geq \int_0^\infty \overline{f}(\varrho)\varrho^{d-1}k\left(\sqrt{\varrho^2+r^2}\right)\d\varrho\int_{\mathbb{S}^{d-1}_-}\d\omega$$
which, turning back to the original variables yields the conclusion, the factor $\frac{1}{2}$ coming from the integration over the half-sphere $\mathbb{S}^{d-1}_-.$
\end{proof}

Thanks to the above lemma, one can complement Lemma \ref{leminf} for isotropic solutions. We first recall that, if $\psi_0(\xi)=\overline{\psi_0}(|\xi|)$ is an isotropic function, then the solution $\psi$ to \eqref{BEscaled} with initial condition $\psi_0$ is isotropic for any $t \geq 0.$ Indeed, for any rotation matrix $R \in SO(d)$, defining $\tilde{\psi}$ by $\tilde{\psi}(t,\xi)=\psi(t,R\cdot \xi)$ for any $(t,\xi)\in (0,\infty)\times \R^d$, we have
$$\Q_-( \tilde{\psi},\tilde{\psi})(t,\xi)=\Q_-(\psi,\psi)(t,R\cdot\xi), \qquad
\Q_+( \tilde{\psi},\tilde{\psi})(t,\xi)=\Q_+(\psi,\psi)(t,R\cdot\xi),$$
for any $(t,\xi)\in (0,\infty)\times \R^d$.
Consequently, one checks easily that $\tilde{\psi}$ is a solution to \eqref{BEscaled} with initial condition $\psi_0$. By uniqueness, we deduce that $\tilde{\psi}=\psi$. Thus, $\psi$ is an isotropic function. This leads to

\begin{lem}\label{lem:ka} Assume that $\psi_0(\xi)=\overline{\psi_0}(|\xi|)$ is a nonnegative isotropic initial datum satisfying \eqref{massenergie} and \eqref{hypini}. For any $\gamma \in (0,1]$, there exists $\alpha_\star \in (0,1)$ such that, for any $\alpha \in (0,\alpha_\star)$ the solution $\psi(t,\xi)$ to \eqref{BEscaled} satisfies
$$\int_{\R^d} \psi(t,\xi_*)|\xi-\xi_*|^\gamma \d\xi_*\geq \mu_\alpha \langle \xi \rangle ^\gamma, \qquad \forall \xi \in \R^d,\:\:t \geq 0$$
for some positive constant $\mu_\alpha >0$ depending on $b(\cdot)$, $\gamma,d,\alpha$ and on the initial datum $\psi_0$.
\end{lem}
\begin{proof} Applying Lemma \ref{lem:LU} with the function $k(x)=x^\gamma$ we get that
$$\int_{\R^d} \psi(t,\xi_*) |\xi-\xi_*|^\g \d\xi_* \geq \frac{1}{2}\int_{\R^d} \psi(t,\xi_*)\left(|\xi|^2+|\xi_*|^2\right)^{\frac{\g}{2}}\d\xi_*.$$
Moreover, for any $\gamma \in (0,1]$, there exists $c_\gamma >0$ such that $\left(|\xi|^2+|\xi_*|^2\right)^{\frac{\gamma}{2}}\geq c_\gamma (|\xi|^\gamma+|\xi_*|^\gamma)$ for any $\xi,\xi_* \in \R^d$. Then,
$$\int_{\R^d} \psi(t,\xi_*) |\xi-\xi_*|^\gamma \d\xi_* \geq  \dfrac{c_\gamma}{2} \left(|\xi|^\gamma + \int_{\R^d}  \psi(t,\xi_*) |\xi_*|^\gamma \d\xi_*\right).$$
Now, according to Lemma \ref{leminf}, whenever $\alpha \in (0,\alpha_\star)$ there exists $\mathbf{C}_\alpha$ such that
$$\int_{\R^d}  \psi(t,\xi_*) |\xi_*|^\gamma \d\xi_* \geq \mathbf{C}_\alpha \int_{\R^d}  \psi_0(\xi_*) |\xi_*|^\gamma \d\xi_*, \qquad t\geq 0.$$
Consequently,
$$\int_{\R^d} \psi(t,\xi_*) |\xi-\xi_*|^\gamma \d\xi_*   \geq   \frac{c_\gamma}{2} \min\left\{1, \mathbf{C}_\alpha \int_{\R^d}  \psi_0(\xi_*) |\xi_*|^\gamma \d\xi_*\right\}(1+|\xi|^\gamma) \qquad \forall \xi \in \R^d, \quad t \geq 0.$$
Now, since there exists  $\kappa_\gamma >0$ such that $(1+|\xi|^\gamma)\geq \kappa_\gamma \left(1+|\xi|^2\right)^{\frac{\gamma}{2}}$ for any $\xi \in \R^d$, we finally obtain the conclusion with $\mu_\alpha=\frac{c_\gamma\,\kappa_\gamma}{2} \min\left(1, \mathbf{C}_\alpha \int_{\R^d}  \psi_0(\xi_*) |\xi_*|^\gamma \d\xi_*\right).$
\end{proof}
\begin{rmq}\label{rmqastar} The parameter $\alpha_\star$ is exactly the one of Lemma \ref{leminf}. Precisely, $$\alpha_\star=\frac{\varrho_{k_0}-1}{\varrho_{k_0}}$$ where $k_0=\frac{j_0\gamma}{2} < 1$ with $j_0 \in \mathbb{N}$ such that $k_0 < 1$ and $k_0+\frac{\gamma}{2} \geq 1.$ In particular, for $\gamma=1$, $k_0=\frac{1}{2}$ and, in dimension $d=3$ and hard spheres interactions $b(\cdot)=\frac{1}{4\pi}$, one sees that $\alpha_\star=\frac{1}{4}.$
\end{rmq}

\section{$L^p$-estimates}\label{sec_LP}

We are now interested in uniform in time propagation of $L^p$-norms for the solution to \eqref{BEscaled} and we prove Theorem \ref{theoLp}.
As in the previous section,  we fix a nonnegative initial distribution $\psi_0$ satisfying \eqref{massenergie} and \eqref{hypini} and such that
$$\psi_0\in L^1_{2+\gamma}(\R^d)\cap L^p(\R^d)$$
for some \textit{fixed} $p >1$ and we let then $\psi
\in\C([0,\infty);w-L^1(\R^d)) \cap L^\infty_{\mathrm{loc}}((0,\infty),L^1_{2+\gamma}(\R^d))$
be a nonnegative solution to \eqref{BEscaled} with $\psi(0,\cdot)=\psi_0$.
\textit{\textbf{We assume in this section that $\psi_0$ is an isotropic function,}} that is \eqref{HYP} holds. For a given $p >1$, multiplying \eqref{BEscaled} by $p\psi(t,\xi)^{p-1}$ and integrating over $\R^d$, we get
\begin{equation}\label{dLp}\begin{split}
\dfrac{\d}{\d t}\|\psi(t)\|_{L^p}^p &+ \left(p\mathbf{A}_\psi(t)-d\mathbf{B}_\psi(t)\right)\|\psi(t)\|_{L^p}^p\\
&=p(1-\alpha)\int_{\R^d}\Q_+(\psi,\psi)(t,\xi)\psi(t,\xi)^{p-1}\d \xi -p\int_{\R^d}\Q_-(\psi,\psi)(t,\xi)\psi(t,\xi)^{p-1}\d \xi\\
&=:(1-\alpha)p\,\mathbb{G}_p(\psi(t)) -p\mathbb{L}_p(\psi(t)) \end{split}\end{equation}
where we set
$$\mathbb{G}_p(\psi(t))=\int_{\R^d}\Q_+(\psi,\psi)(t,\xi)\psi(t,\xi)^{p-1}\d \xi,$$
and
$$\mathbb{L}_p(\psi(t))=\int_{\R^d}\Q_-(\psi,\psi)(t,\xi)\psi(t,\xi)^{p-1}\d \xi.$$
The estimates for $\mathbb{G}_p(\psi(t))$ are well-known \cite{MouhVill04,AloGa} and, for $\varepsilon >0$, there exists some (explicit)   $\theta \in (0,1)$ and  $C_\varepsilon >0$ such that
$$\mathbb{G}_p(\psi(t)) \leq C_\varepsilon \|\psi(t)\|^{1+p\theta}_{L^1}\,\|\psi(t)\|_{L^p}^{p-p\theta} + \varepsilon \|\psi(t)\|_{L^1_2}\,\|\psi(t)\|_{L^p_{\frac{\gamma}{p}}}^p,$$
i.e.
\begin{equation}\label{GpPsi}\mathbb{G}_p(\psi(t)) \leq C_\varepsilon  \|\psi(t)\|_{L^p}^{p-p\theta} + \varepsilon \left(1+\frac{d}{2}\right)\|\psi(t)\|_{L^p_{\frac{\gamma}{p}}}^p
.\end{equation}
Now, all the strategy consists in finding conditions on $\alpha$ and $p >1$ ensuring that
$$-\left(p\mathbf{A}_\psi(t)-d\mathbf{B}_\psi(t)\right)\|\psi(t)\|_{L^p}^p -p\mathbb{L}_p(\psi(t))$$
can absorb the leading order term $\varepsilon (1-\alpha)p\left(1+\frac{d}{2}\right)\|\psi(t)\|_{L^p_{\frac{\gamma}{p}}}^p$. One has
$$\left(p\mathbf{A}_\psi(t)-d\mathbf{B}_\psi(t)\right)=-\dfrac{\alpha}{2}\left(d(p-1)+2p\right)\mathbf{a}_\psi(t)
+ \alpha(p-1)\mathbf{b}_\psi(t)$$
and, since $\mathbf{b}_\psi(t) \geq 0$, it is enough to estimate
$$\mathbb{K}_p:=\dfrac{\alpha}{2}\left(d(p-1)+2p\right)\mathbf{a}_\psi(t)\,\|\psi(t)\|_{L^p}^p -p\mathbb{L}_p(\psi(t)).$$
Compounding $\|\psi(t)\|_{L^p}^p$ and $\mathbf{a}_\psi(t)$ into a unique integral, we get
$$ \mathbf{a}_\psi(t)\,\|\psi(t)\|_{L^p}^p= \int_{\R^{3d}}|\xi-\xi_*|^\g \psi(t,\xi)\psi(t,\xi_*)\psi(t,z)^p\d z \d \xi\d\xi_*.$$
One has $|\xi-\xi_*|^\gamma \leq   |z-\xi|^\gamma + |z-\xi_*|^\gamma $ so that
\begin{multline*}
 \mathbf{a}_\psi(t)\,\|\psi(t)\|_{L^p}^p \leq  \int_{\R^{3d}}|z-\xi|^\gamma \psi(t,\xi)\psi(t,\xi_*)\psi(t,z)^p\d z \d \xi\d\xi_*\\
 +   \int_{\R^{3d}}|z-\xi_*|^\gamma \psi(t,\xi)\psi(t,\xi_*)\psi(t,z)^p\d z \d \xi\d\xi_*\end{multline*}
i.e.
$$ \mathbf{a}_\psi(t)\,\|\psi(t)\|_{L^p}^p \leq
2\int_{\R^{2d}}|z-\xi|^\gamma \psi(t,\xi)\psi(t,z)^p\d z \d \xi =2\mathbb{L}_p(\psi(t)).$$
One sees then that $\mathbb{K}_p \leq -\eta_p \mathbb{L}_p(\psi(t))$ with $\eta_p=p-2\a\,p-\a\,d(p-1)$ and
$$\eta_p >0 \Longleftrightarrow  p(\alpha d +2\alpha -1) < \alpha d.$$
One can distinguish between two cases:
\begin{enumerate}[(i)]
\item if $\alpha \leq \frac{1}{d+2}$ then one has $\eta_p \geq \alpha d >0$ for any $p >1;$
\item if  $\alpha > \frac{1}{d+2}$ then $\eta_p > 0$ if and only if $p < p_\alpha^\star$ where $p_\alpha^\star=\frac{\alpha d}{\alpha d + 2\alpha-1}$. Notice that $p_\alpha^\star >1$ if and only if $0  <\alpha < \frac{1}{2}.$
\end{enumerate}
In other words, for any $\alpha < \frac{1}{2},$ there exists $p_\alpha^\star >1$ such that
\begin{equation}\label{Kpa}
\mathbb{K}_p \leq -\eta_p\mathbb{L}_p(\psi(t)) \qquad \text{ with } \eta_p >0 \qquad \forall p \in (1,p_\alpha^\star).\end{equation}
Putting together \eqref{dLp}, \eqref{GpPsi} and \eqref{Kpa} we get, for $\alpha < \frac{1}{2}$ and $p \in (1,p_\alpha^\star)$:
\begin{equation*}
\dfrac{\d}{\d t}\|\psi(t)\|_{L^p}^p \leq C_\varepsilon (1-\a)p \|\psi(t)\|_{L^p}^{p-p\theta}
+ \varepsilon (1-\a)p \left(1+\frac{d}{2}\right)\|\psi(t)\|_{L^p_{\frac{\gamma}{p}}}^p-\eta_p \mathbb{L}_p(\psi(t)).
\end{equation*}
It remains now to compare $\mathbb{L}_p(\psi(t))$ to $\|\psi(t)\|_{L^p_{\frac{\gamma}{p}}}^p$. This is the only point where we shall invoke our assumption \eqref{HYP}.  Precisely, from \eqref{HYP} and Lemma \ref{lem:ka}, if $\alpha \in (0,\alpha_\star)$ there exists $\mu_\alpha >0$ depending on $\psi_0$ such that $$\int_{\R^d}|\xi-\xi_*|^\gamma \psi(t,\xi_*)\d\xi_* \geq \mu_\alpha \langle \xi \rangle ^\gamma \qquad \forall t \geq 0, \qquad \forall \xi \in \R^d.$$
Therefore,
\begin{equation}\label{infLp}\mathbb{L}_p(\psi(t)) \geq \mu_\alpha \int_{\R^d} \psi(t,\xi)^p \langle \xi \rangle^\g \d\xi =\mu_\alpha \|\psi(t)\|_{L^p_{\frac{\gamma}{p}}}^p.\end{equation}
Then, for any fixed $0 < \alpha < \min(\frac{1}{2},\alpha_\star)$ and fixed $p \in (1,p_\a^\star)$, one can choose  $\varepsilon >0$ such that $\varepsilon (1-\a)p \left(1+\frac{d}{2}\right) =\frac{\eta_p\,\mu_\alpha}{2}$ to get the following
\begin{equation*}
\dfrac{\d}{\d t}\|\psi(t)\|_{L^p}^p \leq K  \|\psi(t)\|_{L^p}^{p-p\theta}
-\frac{\eta_p\,\mu_\alpha}{2}\|\psi(t)\|_{L^p_{\frac{\gamma}{p}}}^p,
\end{equation*}
for some positive constant $K >0$. This implies clearly that
$$\sup_{t \geq 0}\|\psi(t)\|_{L^p} \leq \max\left\{\|\psi_0\|_{L^p},\left(\frac{2K}{\eta_p\mu_\alpha}\right)^{\frac{1}{p\theta}}\right\}.$$
This proves Theorem \ref{theoLp} with $C_p(\psi_0)=\left(\frac{2K}{\eta_p\mu_\alpha}\right)^{\frac{1}{p\theta}}$. Notice that, as announced, $C_p(\psi_0)$ depends on the initial datum $\psi_0$ only through $\mu_\alpha$ and so only through the moment $M_{\frac{\gamma}{2}}(0)$.
\begin{rmq}\label{rmqabar} One sees from the above proof that $\overline{\alpha}=\min(\frac{1}{2},\alpha_\star)$ where $\alpha_\star$ is the parameter of Lemma \ref{lem:ka} (see also Remark \ref{rmqastar}).
\end{rmq}

\begin{rmq}\label{rmq:Cpunif} The constant $C_p(\psi_0)$ depends on the initial datum $\psi_0$ only through the inverse of the moment $M_{\frac{\gamma}{2}}(0)=\int_{\R^d} \psi_0(\xi)|\xi|^\gamma \d\xi$. In particular, with the notations of Lemma \ref{leminf} and Remark \ref{Cgamma}, one sees that, given $\gamma \in (0,1]$ and $0<\alpha<\overline{\alpha}$ then for any  $p  \in (1,p_\a^\star)$,
$$\sup_{t \geq 0}\|\psi(t)\|_{L^p} <  \max\left\{\|\psi_0\|_{L^p},C_p\right\}$$
for some constant $C_p >0$ depending only on $\alpha,$ $\gamma$, $b(\cdot)$ and the dimension $d$ provided $\psi_0 \in \mathcal{C}_\gamma(\alpha)$ satisfies the assumption of Theorem \ref{theoLp}.
\end{rmq}

\section{Weighted Sobolev estimates}\label{sec_sob}

We now set $\gamma=1$ and prove Theorem \ref{theo:sob}. The proof is very similar to that of \eqref{L2norm} and \eqref{H1norm} except that we need here to prove uniform in time bounds. The  restriction $\gamma=1$ comes from the fact that the best control of the loss term $\Q_{-}$ is available only for $\gamma=1$, see \eqref{eq:Q-}.

Multiplying \eqref{BEscaled} by $2 \psi(t,\xi)\langle \xi \rangle^{2k}$ and integrating over $\R^d$, we get \eqref{outil4}.
Now, according to \cite[Corollary 2.2]{AloGa}, for any $\varepsilon \in (0,1)$, there exists $C_\varepsilon > 0$ such that  
$$\int_{\R^d}\Q_+(\psi,\psi)(t,\xi)\psi(t,\xi)\langle \xi\rangle^{2k}\d\xi \leq C_\varepsilon \|\psi(t)\|_{L^1_{\frac{d(d-3)}{d-1}+k}}^{2-1/d}\,\|\psi(t)\|_{L^2_k}^{1+1/d}  + \varepsilon \|\psi(t)\|_{L^1_{k}}\,\|\psi(t)\|_{L^2_{k}}^2.$$
According to \eqref{h_o_mom}, since $\psi_0 \in  L^1_{\frac{d(d-3)}{d-1}+k}$ one has
$$\sup_{t \geq 0} \|\psi(t)\|_{ L^1_{\frac{d(d-3)}{d-1}+k}} < \infty$$ and, in turns, $\sup_{t \geq 0}\|\psi(t)\|_{L^1_k} < \infty.$ On the other hand, we have
$$ \sup_{t\geq 0} \left|\mathbf{A}_{\psi}(t)\right|\leq C \quad \mbox{ and }\quad  \sup_{t\geq 0} \left|\mathbf{B}_{\psi}(t) \right|   \leq C,$$
for some constant $C>0$.
Thus, bounding the $L^2_{k-1}$ norm by the $L^2_k$ one, \eqref{outil4} together with Lemma \ref{lem:ka} lead to
 $$\frac{\d}{\d t} \|\psi(t)\|^2_{L^2_{k}}+ 2\mu_\alpha \|\psi(t)\|_{L^2_{k+\frac{1}{2}}}\\ \leq C \|\psi(t)\|_{L^2_{k}}^2 + 2\, C_\varepsilon \, \|\psi(t)\|_{L^2_{k}}^{1+1/d}  + 2 \,\varepsilon\,M\,\|\psi(t)\|_{L^2_{k}}^2,$$
for some constants $C >0$ and $M > 0$ (depending on $k$). Now, choosing  $\varepsilon$ such that $ 2 \varepsilon M \leq \mu_\alpha$ we get the existence of some positive constants $C_{1}  > 0$ and $C_{2} > 0$ (still depending on $k$) such that
 $$\frac{\d}{\d t} \|\psi(t)\|^2_{L^2_{k}}+\mu_\alpha \|\psi(t)\|_{L^2_{k+\frac{1}{2}}}^2\leq C_{1}\|\psi(t)\|_{L^2_{k}}^2 + C_{2} \|\psi(t)\|_{L^2_{k}}^{1+1/d}.$$
 Now, one uses the fact that, for any $R > 0$, $$\|\psi(t)\|_{L^2_{k}}^2 \leq (1+R^2)^{k}\|\psi(t)\|_{L^2}^2+R^{-1}\|\psi(t)\|_{L^2_{k+1/2}}^2$$
and, since $\sup_{t\geq 0}\|\psi(t)\|_{L^2} < \infty$ by  Theorem \ref{theoLp}, one can choose $R > 0$ large enough so that $C_{1} R^{-1}=\mu_\alpha/2$ to obtain
$$\frac{\d}{\d t} \|\psi(t)\|^2_{L^2_{k}}+ \frac{\mu_\alpha}{2}\|\psi(t)\|_{L^2_{k+\frac{1}{2}}}^2 \leq C_{3} + C_{2} \|\psi(t)\|_{L^2_{k}}^{1+1/d}.$$
Taking $k=\frac{9+d}{2}+\kappa$, one obtains \eqref{L2normunif}  since $1+1/d < 2$.

\medskip
Let us  now prove \eqref{H1normunif}.
For the solution $\psi(t,\xi)$ to \eqref{BEscaled}, we set
$G_j(t,\xi)=\partial_j\psi(t,\xi)$ for $j\in\{1,\ldots,d\}$. Then, $G_j$ satisfies \eqref{eqGj_1}. 
For given $q \geq 1/2$, we multiply this equation by $2\,G_j(t,\xi)\langle \xi \rangle^{2q}$ and integrate over $\R^d$. Then,  after an integration by parts and using Lemma \ref{lem:ka},  one obtains
 \begin{multline}\label{eq_GJ_3}
\frac{\d}{\d t} \|G_j(t)\|^2_{L^2_q}+ \left(2 \mathbf{A}_{\psi}(t)+ (2-d-2q) \mathbf{B}_{\psi}(t)\right) \|G_j(t)\|^2_{L^2_q} + 2q \mathbf{B}_{\psi}(t)  \|G_j(t)\|^2_{L^2_{q-1}}  \\
\leq 2 (1-\alpha)\int_{\R^d}\partial_j\Q_+(\psi,\psi)(t,\xi)G_j(t,\xi)\langle \xi\rangle^{2q}\d\xi\\ - 2\mu_\alpha\|G_j(t)\|^2_{L^2_{q+\frac{1}{2}}}
-2\int_{\R^d} \Q_-(\psi,G_j)(t,\xi)G_j(t,\xi)\langle \xi\rangle^{2q}\d \xi.
 \end{multline}
Clearly, one has
\begin{equation*}\begin{split}
\int_{\R^d}\left|\partial_j\Q_+(\psi,\psi)(t,\xi)\right|\,|G_j(t,\xi)|\, \langle \xi\rangle^{2q}\d\xi &\leq \|\partial_j \Q_+(\psi,\psi)(t)\|_{L^2_{q-\frac{1}{2}}}\|G_j(t)\|_{L^2_{q+\frac{1}{2}}}\\
&\leq \|\Q_+(\psi,\psi)(t)\|_{\mathbb{H}^1_{q-\frac{1}{2}}}\|G_j(t)\|_{L^2_{q+\frac{1}{2}}}.\end{split}\end{equation*}
Now, using \cite[Theorem 2.7]{BaLo}, for any $\varepsilon > 0$, there exists $C_\varepsilon > 0$ such that
\begin{multline*}
\|\Q_+(\psi,\psi)(t)\|_{\mathbb{H}^1_{q-\frac{1}{2}}} \leq C_\varepsilon \|\psi(t)\|_{\mathbb{H}^{\frac{3-d}{2}}_{q+1+\kappa}}\|\psi(t)\|_{L^1_{2q+\frac{1}{2}+\kappa}}+\\
\varepsilon\,\|\psi(t)\|_{L^1_{q+\frac{1}{2}}}\,\|\psi(t)\|_{L^2_{q+\frac{1}{2}}}
+2\varepsilon\,\|\psi(t)\|_{L^1_{q+\frac{1}{2}}}\, \sum_{i=1}^d \|G_i(t)\|_{L^2_{q+\frac{1}{2}}}.
\end{multline*}
Since $d \geq 3$, one estimates the $\mathbb{H}^{\frac{3-d}{2}}_{q+1+\kappa}$ norm by the $L^2_{q+1+\kappa}$ norm and, using \eqref{L2normunif} together with \eqref{h_o_mom}, our assumptions on the initial datum implies that
$$\sup_{t \geq 0} \|\psi(t)\|_{L^2_{q+1+\kappa}} < \infty \quad \text{ and } \quad \sup_{t \geq 0}\,\|\psi(t)\|_{L^1_{2q+\frac{1}{2}+\kappa}} < \infty.$$
Therefore, for any $\varepsilon > 0$, there exists $C_1(\varepsilon,q) > 0$ and $C_2(q) > 0$ such that
$$\|\Q_+(\psi,\psi)(t)\|_{\mathbb{H}^1_{q-\frac{1}{2}}} \leq C_1(\varepsilon,q) +  \varepsilon\,C_2(q)\, \sum_{i=1}^d \|G_i(t)\|_{L^{2}_{q+\frac{1}{2}}}.$$
One estimates the last integral in \eqref{eq_GJ_3} as in the proof of \cite[Theorem 2.8]{BaLo}; namely, an integration by parts yields
\begin{equation}\label{eq:Q-}| \Q_{-}(\psi,G_j)(t,\xi)| = \psi(t,\xi) \left|\int_{\R^d} \partial_j \psi(t,\xi_*)|\xi-\xi_*| \, \d\xi_* \right| \leq  \psi(t,\xi) \|\psi(t)\|_{L^1}=\psi(t,\xi).\end{equation}
Then, Cauchy-Schwarz inequality yields
$$
\left|\int_{\R^d}  \Q_{-}(\psi,G_j)(t,\xi)\,G_j(t,\xi)\langle \xi\rangle ^{2q}\d \xi\right| \leq \|\psi(t)\|_{L^2_{q}}\,\|G_j(t)\|_{L^2_q} \leq C_q \,\|G_j(t)\|_{L^2_q}$$
for some positive $C_q >0$ where we used the uniform bounds on the $L^2_q$-norm of $\psi(t)$ provided by  \eqref{L2normunif}.
Recall that
\begin{equation*}
2\mathbf{A}_{\psi}(t)+\left(2-d-2q \right)\mathbf{B}_{\psi}(t)=-\frac{\alpha}{2}\left(d-2q +6\right)\mathbf{a}_{\psi}(t) + \frac{\alpha}{d}\left(d+2-2q\right)\mathbf{b}_{\psi}(t)\end{equation*}
while $2q\mathbf{B}_{\psi}(t)=-\alpha\,q\,\mathbf{a}_{\psi}(t) + \frac{\alpha\,2q}{d}\,\mathbf{b}_{\psi}(t).$
Since $q\leq 1 +\frac{d}{2}$, one may neglect all the terms involving $\mathbf{b}_{\psi}(t)$ to obtain the bound from below:
\bean
& &\left(2\mathbf{A}_{\psi}(t)+\left(2-d-2q\right)\mathbf{B}_{\psi}(t)\right)\|G_j(t)\|_{L^2_q}^2  + 2q\mathbf{B}_{\psi}(t)\,\|G_j(t)\|_{L^2_{q-1}}^2\\
& & \hspace{4cm}\geq -\frac{\alpha}{2}\left(d + 6\right)\mathbf{a}_{\psi}(t)\,\|G_j(t)\|_{L^2_q}^2 + \alpha\,q \,\mathbf{a}_{\psi}(t) \left(\|G_j(t)\|_{L^2_q}^2-\|G_j(t)\|_{L^2_{q-1}}^2\right)\\
& & \hspace{4cm}\geq -\frac{\alpha}{2}\sqrt{d}\left(d+6\right)\|G_j(t)\|_{L^2_q}^2
\eean
using the fact that $\mathbf{a}_\psi(t) \leq \sqrt{d}$ for any $t \geq 0$ (following the arguments of \cite[Lemma 2.1]{BaLo}). Thus, \eqref{eq_GJ_3} reads
\begin{multline*}
\frac{\d}{\d t} \|G_j(t)\|^2_{L^2_q}  -\frac{\alpha}{2}\sqrt{d}\left(d+6\right)\|G_j(t)\|_{L^2_q}^2 +2\mu_\alpha\|G_j(t)\|_{L^2_{q+\frac{1}{2}}}^2\\
\leq 2(1-\alpha)C_1(\varepsilon,q)\|G_j(t)\|_{L^2_{q+\frac{1}{2}}} +  \varepsilon\,C_2(q)\,\|G_j(t)\|_{L^2_{q+\frac{1}{2}}}\, \sum_{i=1}^d \|G_i(t)\|_{L^{2}_{q+\frac{1}{2}}} + 2C_q \|G_j(t)\|_{L^2_{q}}
\end{multline*}
where $C_q, C_1(\varepsilon,q)$ and $C_2(q)$ are positive constants independent of $\alpha$ and $t$. Define, for any $k \geq 0$, the semi-norm
$$\|\psi(t)\|_{\overset{\circ}{\mathbb{H}}^1_k}=\left(\sum_{j=1}^d \|\partial_j \psi(t)\|_{L^2_k}^2\right)^{1/2}.$$
Setting $\alpha_1:=\min\left\{\overline{\alpha},\alpha_0,\frac{4\mu_\alpha}{\sqrt{d}(d+6)}\right\}$ and summing  over all $j\in\{1,\ldots, d\}$, we get
\begin{multline*}
\frac{\d}{\d t} \|\psi(t)\|_{{\overset{\circ}{\mathbb{H}}^1_q}}^2+  \frac{\sqrt{d}}{2} (d+6)(\alpha_1-\alpha) \|\psi(t)\|_{\overset{\circ}{\mathbb{H}}^1_{q+\frac{1}{2}}}^2\\
\leq 2C_{1}(\varepsilon,q) \sum_{j=1}^d\|G_j(t)\|_{L^2_{q+\frac{1}{2}}}+    \varepsilon C_2(q) \left(\sum_{j=1}^d\|G_j(t)\|_{L^{2}_{q+\frac{1}{2}}}\right)^2  +2C_q\sum_{j=1}^d \|G_j(t)\|_{L^2_q} \\
\leq 2C_{1}(\varepsilon,q) \sum_{j=1}^d\|G_j(t)\|_{L^2_{q+\frac{1}{2}}}+    d\varepsilon C_2(q) \|\psi(t)\|_{\overset{\circ}{\mathbb{H}}^1_{q+\frac{1}{2}}}^2+2\sqrt{d}C_q\|\psi(t)\|_{\overset{\circ}{\mathbb{H}}^1_q}.
 \end{multline*}
 Using Young's inequality, for any ${\delta}^{\star} > 0$ one gets
\begin{multline*}
\frac{\d}{\d t} \|\psi(t)\|_{\overset{\circ}{\mathbb{H}}^1_q}^2+  \frac{\sqrt{d}}{2} (d+6)(\alpha_1-\alpha) \|\psi(t)\|_{\overset{\circ}{\mathbb{H}}^1_{q+\frac{1}{2}}}^2\\
\leq \left( 2{\delta}^{\star}\,C_1(\varepsilon,q)  + d\varepsilon\,C_2(q)\right)\|\psi(t)\|_{\overset{\circ}{\mathbb{H}}^1_{q+\frac{1}{2}}}^2 + \frac{2\,d\,C_1(\varepsilon,q)}{{\delta}^{\star}} + 2\sqrt{d}C_q\|\psi(t)\|_{\overset{\circ}{\mathbb{H}}^1_q}.\end{multline*}
For any fixed $\alpha < \alpha_1$, one can choose first $\varepsilon > 0$ small enough and then ${\delta}^{\star} > 0$ small enough so that
 $\left( 2{\delta}^{\star}\,C_1(\varepsilon,q)  + d\varepsilon\,C_2(q)\right)= \frac{\sqrt{d}}{4} (d+6)(\alpha_1-\alpha) $ to get
$$\frac{\d}{\d t} \|\psi(t)\|_{\overset{\circ}{\mathbb{H}}^1_q}^2+ \frac{\sqrt{d}}{4} (d+6)(\alpha_1-\alpha)\|\psi(t)\|_{\overset{\circ}{\mathbb{H}}^1_{q+\frac{1}{2}}}^2 \leq
2\sqrt{d}\,C_q \|\psi(t)\|_{\overset{\circ}{\mathbb{H}}^1_q} + C$$
which yields easily the conclusion taking $q=\frac{d+7+\kappa}{2}$.

\section{Existence of self-similar profile}\label{sec_autosim}

We now proceed to the proof of the main result of this paper, that is the proof of  Theorem \ref{existence}.
 As already announced, the existence of a stationary solution
to \eqref{BEscaled} relies on the application of Theorem \ref{GPV}
to the evolution semi-group $(\mathcal{S}_t)_{t \geq 0}$ governing
\eqref{BEscaled}.  Let us now fix $\alpha < \alpha_1$. For any nonnegative $\psi_0 \in L^1_{3}(\R^d) \cap L^2(\R^d)$ satisfying \eqref{hypini},  let $\psi(t)=\mathcal{S}_t \psi_0$ denote the unique solution to
\eqref{BEscaled} with initial state $\psi(0)=\psi_0$ constructed by Theorem \ref{well_posedness}. 
Thanks to the uniform bounds on the $L^1_{3}(\R^d)$ and $L^{2}(\R^d)$ norms provided by Proposition \ref{theoMom} and Theorem \ref{theoLp} respectively combined with the propagation of lower bounds for $M_{\frac{1}{2}}(t)$ (see Lemma \ref{leminf}, Remarks \ref{Cgamma} \& \ref{rmq:Cpunif}) and the weighted Sobolev estimates of Theorem \ref{theo:sob}, the nonempty convex subset
\begin{multline*}\mathcal{Z}= \left\{ 0 \leq \psi \in L^1(\R^d), \hspace{0.3cm} \psi(\xi)=\overline{\psi}(|\xi|) \quad \forall \xi \in \R^d,
         \int_{\R^d} \psi(\xi)\d \xi= 1,\quad \int_{\R^d}\psi(\xi)|\xi|^2\d\xi=\frac{d}{2}\right.\\
           \left.  \int_{\R^d}\psi(\xi)|\xi|^3\d\xi\leq M_1, \quad  \|\psi\|_{L^{2}} \leq M_2,  
\quad \int_{\R^d}\psi(\xi)|\xi|^{\mathfrak{q}(\kappa)}\d\xi \leq M_3, \right.\\
           \left. \|\psi\|_{L^{2}_{\frac{9+d}{2}+\kappa}} \leq M_4,  
 \quad \|\nabla\psi\|_{L^2_{\frac{7+d+\kappa}{2}}} \leq M_5 \quad  \mbox{ and }\quad \int_{\R^d}\psi(\xi)|\xi|\d\xi\geq K 
  \right\} 
\end{multline*} 
with $\mathfrak{q}(\kappa)=\max\left\{\frac{9+d(d-2)}{2}+\kappa, 10+d+2\kappa\right\}$, is stable by the semi-group provided $M_1$, $M_2$, $M_3$, $M_4$, 
$M_5$ are big enough and $K$ is small enough. This set is compact in 
$\mathcal{Y}= L^1(\R^d)$ endowed with the weak topology by Dunford-Pettis 
Theorem. Let us now justify that  for all $t \ge 0$, $\mathcal{S}_t$  is
 continuous on $\mathcal{Z}$. By \cite[Corollary 1.2.2]{vra}, it is sufficient to check that for all $t \ge 0$, $\mathcal{S}_t$  is sequentially
 continuous on $\mathcal{Z}$. Fix $\psi_0 \in \mathcal{Z}$. Let 
$(\psi_0^n)_{n\in\N}$ be a sequence from $\mathcal{Z}$ that converges to 
$\psi_0$ in $\mathcal{Y}$. For any $n\in\N$, we then denote by $\psi^n$ the 
solution to \eqref{BEscaled} with initial condition $\psi_0^n$. Let $T>0$. 
Proceeding as in the proof of Proposition \ref{conv4}, it is 
clear that the sequence  $(\psi^n)_{n\in\N}$ is relatively compact in 
$\C([0,T],w-L^1(\R^d))$. Thus, there exists a subsequence $(\psi_{n_k})_k$ 
which converges to some $\psi\in \C([0,T],w-L^1(\R^d))$. Passing to the limit in \eqref{weakformu}, we deduce that $\psi$ is the solution to \eqref{BEscaled} with initial condition $\psi_0$. Since $(\psi^n)_{n\in\N}$ admits a unique limit point, this sequence is convergent, which proves the sequential continuity of $\mathcal{S}_t$  at $\psi_0$ for any $t\in[0,T]$.
Then, Theorem \ref{GPV} shows that, for any $\alpha < \alpha_1$, there exists a nonnegative stationary solution to \eqref{BEscaled} in $L^1_{3}(\R^d) \cap L^2(\R^d)$ with unit mass and energy equal to $\frac{d}{2}$. 

\begin{rmq} Notice that, unfortunately, we are able to construct only radially symmetric solutions to \eqref{tauT}. Clearly, this relies on the restriction \eqref{HYP} for the control of $L^p$ norms. At first sight, it may seem possible to construct solutions to \eqref{tauT} with zero bulk velocity but it  is not known whether this property is preserved by the semi-group $(\mathcal{S}_t)_{t \geq 0}$. Since the property of being radially symmetric is preserved by $(\mathcal{S}_t)_{t \geq 0}$, we have to restrict our choice to that class of self-similar solutions.
\end{rmq}

\begin{rmq}
In the special case of hard spheres interactions in dimension $d=3$, i.e. whenever $\mathcal{B}(\xi-\xi_*,\sigma)=\frac{|\xi-\xi_*|}{4\pi}$, one has according to Remarks \ref{rmqalpha0}, \ref{rmqastar} and \ref{rmqabar} that $\alpha_0 =\frac{2}{7},$ $\overline{\alpha}=\frac{1}{4}$. Therefore, $\alpha_1\leq\frac{1}{4}$.
\end{rmq}

 \section{Conclusion and perspectives}\label{sec:discuss}

We derived in the present paper the existence of a self-similar profile $\psi_H$ associated to the probabilistic ballistic annihilation equation \eqref{BE}. Such a self-similar profile is actually the steady state of the rescaled equation \eqref{BEscaled} and the existence of such a steady state was taken for granted in various papers in the physics literature \cite{Maynar1, Maynar2, Trizac}. Our paper thus provides  a rigorous justification of some of the starting point of the analysis of the \textit{op. cit.}. The self-similar profile $\psi_H$ we constructed is isotropic, i.e.
$$\psi_H(\xi)=\overline{\psi_H}(|\xi|), \qquad \xi \in \mathbb{R}^d$$
and the existence is proven only in a given (explicit) range of the probability parameter $\alpha.$ Namely, we proved the existence of $\psi_H$ only whenever the probability parameter $\alpha$ lies in some interval $(0,\alpha_1)$ with some explicit $\alpha_1 >0$. Even if the parameter $\alpha_1 >0$ is certainly not optimal, this restriction arises naturally from our method of proof; in particular, it seems difficult to prove uniform in time estimates of the higher-order moments for all range of parameters $\alpha \in (0,1)$. However, our restriction on the initial datum (isotropy, $L^p$-integrability) and on the probability parameter $\alpha$ leaves several questions open. Let us list a few of them that can be seen as possible perspectives for future works.

\subsection{Uniqueness} A first natural question that should be addressed is of course  the uniqueness of the self-similar profile $\psi_H$. Clearly, since our existence result is based upon a compactness argument (via Tykhonov fixed point Theorem \ref{GPV}) it does not provide any clue for uniqueness. We believe that, as it is the case for the Boltzmann equation with inelastic hard spheres \cite{MiMo3, BCL}, a perturbation argument is likely to be adapted here. Such an approach consists in taking profit of the knowledge of  the stationary solution in the "pure collisional limit"  $\alpha=0$ (for which the steady state is clearly a uniquely determined Maxwellian distribution) and to prove  quantitative estimates of the convergence of stationary
solution  as the  parameter $\alpha$ goes to $0.$ It is likely that such a uniqueness result would require a good knowledge of some quantitative \textit{a posteriori estimates} for the self-similar profile $\psi_H$.

\subsection{A posteriori estimates for $\psi_H$} Typically, we may wonder what are the thickness of the tail of $\psi_H$; more precisely, one should try to find explicit $r >0$, $a >0$ - possibly independent of the parameter $\alpha$ - such that
$$\int_{\R^d} \psi_H(\xi)\exp(a |\xi|^r)\d\xi < \infty.$$
Besides such integral upper bound, one also may wonder if good $L^\infty$-bounds can be derived for $\psi_H$ (at least in the limit $\alpha \to 0$), i.e. is it possible to derive universal explicit functions $\underline{M}(\xi)$ and $\overline{M}(\xi)$ such that
$$\underline{M}(\xi) \leq \psi_H(\xi) \leq \overline{M}(\xi) \qquad \forall \xi \in \mathbb{R}^d \text{ and any } \alpha \in (0,\overline{\alpha}).$$

\subsection{Intermediate asymptotics}\label{interm} A fundamental problem, related to the original probability annihilation equation \eqref{BE}, is to understand the role of the self-similar profile $\psi_H$ (if unique). Indeed, we know that solutions to \eqref{BE} are vanishing as $t \to \infty$
$$\lim_{t \to \infty} f(t,v)=0$$
and physicists expect that the self-similar profile should play the role of an \textit{intermediate asymptotic} in the following sense. One expects to find suitable explicit scaling functions $a(\cdot), b(\cdot)$ a rescaled density $\psi=\psi(\tau,\xi)$ and a rescaled time $\tau(t)$ which are such that, if $f$ is a solution to \eqref{BE} in the form
\begin{equation*}f(t,v)=a(t)\psi(\tau(t),b(t)v)\end{equation*}
then the rescaled density $\psi$ is such that
$$\psi(\tau,\xi) \longrightarrow \psi_H(\xi) \qquad \text{ as } \quad \tau \to \infty.$$
The convergence, in rescaled variables, to a unique self-similar profile is a well-known feature of kinetic equation exhibiting a lack of collisional invariants. In particular, for granular flows described by inelastic hard-spheres, such a self-similar profile (known as the homogeneous cooling state) is known to attract all the solutions to the associated Boltzmann equation yielding a proof of the so-called Ernst-Brito conjecture (see \cite{MiMo3} for a proof and a complete discussion on this topic).

A related question is also the exact decay of the macroscopic quantities associated to solutions $f(t,v)$ to \eqref{BE}: it has already been observed that the number density
$$n(t)=\int_{\mathbb{R}^d} f(t,v)\d v$$
and the kinetic energy
$$E(t)=\int_{\mathbb{R}^d} f(t,v)|v|^2 \d v$$
are continuously decreasing if $\alpha \in (0,1)$ and converge to zero as $t \to \infty$. To determine the precise rate of convergence to zero for such quantities is a physically relevant problem. Notice that for the particular solution $f_H(t,v)$ (constructed in \eqref{scalingfpsi} through the self-similar profile)  the density $n_H(t)$ and energy $E_H(t)$ satisfy
 $$n_H(t)\,E_H(t) \simeq C t^{-2} \qquad \text{ as } t \to \infty$$
for some $C >0$ in the case of true-hard spheres (i.e. whenever $\gamma=1$) as can easily be deduced from \eqref{haff}. One may wonder if such a decay is \emph{universal}, i.e. does any solution $f(t,v)$ to \eqref{BE} is such that $n(t)\,E(t)$ behaves as $t^{-2}$ for large times ?  Partial answers, based upon heuristic and dimensional arguments, are provided by physicists \cite{Piasecki} and it would be interesting to provide a rigorous justification of these results.  Exploiting again the analogy with the Boltzmann description of granular flows, expliciting the decay rate of the number density and the kinetic energy would be the analogue of the so-called Haff's law for inelastic hard-spheres (see \cite{MiMouh06, AloLo1}).

\subsection{Improvement of our result: the special role of entropy} Besides the above cited fundamental questions, we may also discuss some possible improvements of the results we obtained in the present paper. First, one may try to extend the range of parameters $\alpha$ for which our result holds. Notice that, since we strongly believe that the self-similar profile $\psi_H$ is unique in some peculiar regime (at least whenever $\alpha \simeq 0$), getting rid of the isotropic assumption on $\psi_H$ is not particularly relevant. However, in both Theorems \ref{well_posedness} and Theorem \ref{existence}, the hypothesis of $L^p$-integrability does not have a clear physical meaning. It would be interesting to investigate if such an assumption can be relaxed: for instance, it would be more satisfactory to prove the well-posedness result Theorem \ref{well_posedness} under the sole assumption that the initial datum is of finite entropy. Unfortunately, we did not succeed in proving that the flow solution associated to \eqref{BEscaled} propagates suitable bounds of the entropy functional.

\appendix

\section{Well-posedness for the Boltzmann equation with ballistic annihilation}\label{appendixA}

In this appendix, we only give the main lines of the proof of Theorem \ref{exi_BE}. Indeed, the proof of Theorem \ref{exi_BE} may be easily adapted from that of Theorem \ref{well_posedness}.

\bigskip
 Let us denote by $f_0$ a nonnegative distribution function from $W^{1,\infty}(\R^d) \cap L^1_{2+\gamma}(\R^d)$.
Let $n\in\N$. We consider first the well-posedness of the following truncated
equation
\beq\label{trunc_BE}
\partial_t f (t,v) = \mathbb{B}^n(f,f)(t,v)
\eeq
where the collision operator $\mathbb{B}^n(f,f)$ is given by \eqref{Bn}.
Let $T>0$ and
$$h\in \C([0,T]; L^1(\R^d))\cap L^\infty((0,T); L^1(\R^d,|v|^{2+\gamma}\, dv))$$
be fixed.
We introduce the auxiliary equation:
\begin{equation}\label{annihi_lin_f}
\begin{cases}
& \partial_t f (t,v)+\LLn(t,v)\, f(t,v)= (1-\alpha)\, \Q^n_+(h,h)(t,v),\\
& f(0,v)=f_0(v).\end{cases}\end{equation}
Here, as in Section \ref{sec:cauchy},
$$\LLn(t,v):=\int_{\R^d \times \S^{d-1}} \B_n(v-v_*,\s)\,
h(t,v_*)\, \d v_*\, \d\sigma=\|b_n\|_{\L} \int_{\R^d} \Phi_n(|v-v_*|)\, h(t,v_*)
\, \d v_*.$$
The Cauchy problem \eqref{annihi_lin_f} admits a unique solution given by
\begin{eqnarray}\label{solf}
f (t,v) & =& f_0(v) \exp\left(-\int_0^t \LLn(\tau,v)\, \d\tau\right)  \nonumber \\
 & +&   (1-\alpha) \int_0^t \exp\left(-\int_s^t \LLn(\tau,v)\, \d\tau\right)\, \Q_+^n(h,h)(s,v)\, \d s.
\end{eqnarray}

For any $T >0$ and any $M_1,M_2,L, C_\gamma >0$ (to be fixed later on), we define $\H=\H_{T,M_1,M_2,L, C_\gamma}$ as the set of all nonnegative
$h\in\C([0,T];L^1(\R^d))$ such that
$$\sup_{t\in[0,T]} \int_{\R^d} h(t,v)\, \d v\leq M_1, \qquad
\sup_{t\in[0,T]} \int_{\R^d} h(t,v)\,|v|^2\,  \d v \leq M_2,$$
and
$$ \sup_{t\in[0,T]} \int_{\R^d} h(t,v)\,|v|^{2+\gamma}\,  \d v \leq C_\gamma,
\qquad  \sup_{t\in[0,T]} \| h(t)\|_{W^{1,\infty}} \leq L.$$
Define then the mapping $$\T\::\:\H\longrightarrow \C([0,T];L^1(\R^d))$$
 which, to any $h \in \H$, associates the solution $f=\T(h)$ to \eqref{annihi_lin_f} given by \eqref{solf}. We look for parameters $T, M_1, M_2, C_\gamma$ and $L$ that ensures $\T$ to map $\H$ into itself. 

\medskip
\paragraph{\textit{\textbf{Control of the density.}}}  One checks easily that  the solution $f(t,v)$ given by \eqref{solf} fulfills
\begin{equation}\label{density_f}
\sup_{t \in [0,T]} \int_{\R^d}f(t,v)\,\d
  v \leq \|f_0\|_{L^1} + (1-\alpha)\,n^\gamma \, \|b_n\|_{\L}\, M^2_1 \, T, \qquad \forall h \in \H.\end{equation}

\paragraph{\textit{\textbf{Control of the moments.}}}  Arguing as above and as in Section \ref{sec:cauchy}, we get
\bea
\hspace{-1cm}\sup_{t \in [0,T]}\int_{\R^d}f(t,v)\, |v|^2\, \d v & \leq & \int_{ \R^d} f_0(v)\,|v|^2\, \d v  + 4 \,(1-\alpha) \,n^\gamma \, \|b_n\|_{\L}\, M_1 \,M_2\,  T,\label{f_r=2}\\
\hspace{-1cm}\sup_{t \in [0,T]}\int_{\R^d}f(t,v)\, |v|^{2+\gamma}\, \d v&  \leq&
\int_{\R^d}f_0(v)\, |v|^{2+\gamma}\, \d v+ 2^{2+\gamma} \,(1-\alpha) \,n^\gamma \, \|b_n\|_{\L}\, M_1 \,C_\gamma\,  T,\label{f_r=2+d}
\eea
for any $h \in \H$.\\

\paragraph{\textit{\textbf{Control of the $W^{1,\infty}$ norm.}}}

Here again as in Section \ref{sec:cauchy}, we obtain,
\begin{equation}\label{Winfty_f}\begin{split}
\sup_{t \in [0,T]}\|f(t)\|_{W^{1,\infty}} &\leq  \|f_0\|_{W^{1,\infty}}\, (1+2\,n^\gamma \, \|b_n\|_{\L}\, M_1 \, T ) \\
&+ 2\, (1-\alpha) \, n^{1+\gamma} \, \|b_n\|_{\L}\, M_1 \, L\, T \,(2+n^\gamma \, \|b_n\|_{\L}\, M_1 \,  T).\end{split}\end{equation}

Now, from \eqref{density_f}-\eqref{Winfty_f}, one sees that, choosing for instance
 $M_1=2 \|f_0\|_{L^1}$,
$$M_2=2\int_{ \R^d} f_0(v)\,|v|^2 \, \d\xi, \qquad
C_\gamma= 2 \int_{\R^d} f_0(\xi) \, |\xi|^{2+\gamma}\, \d\xi, \qquad
L= 4\, \|f_0\|_{W^{1,\infty}} $$
and
$$T=\frac{1}{16\,\|b_n\|_{\L}\, M_1\, n^{1+\gamma} }\; \min\{1,\;2^{1-\gamma} \, n\},$$
we get that $f \in \H$, i.e. with the above choice of the parameters
$M_1,M_2,C_\gamma,L,T$, one has  $\T(\H) \subset \H$. On the other hand, given $h_1,h_2\in \H$, one deduces from \eqref{annihi_lin_f} and Lemma \ref{ABC} that there exists some constant $C>0$ such that 
\beq \label{stab_2}
\sup_{t \in [0,T]} \left\|\T(h_1)(t)-\T(h_2)(t)\right\|_{L^1_2} \leq C \sup_{t \in [0,T]}\left\|h_1(t)-h_2(t)\right\|_{L^1_2} .
\eeq 
Moreover, $\T(\H)$ is a relatively compact subset of  $\C([0,T],L^1_2(\R^d))$. Thus, the Schauder fixed point theorem ensures the existence of some fixed point $f^1$ of $\T$, i.e. there exists $f^1  \in \C([0,T]; L^1_2(\R^d))\cap L^\infty((0,T); L^1_{2+\gamma}(\R^d)\cap
W^{1,\infty}(\R^d))$ solution to \eqref{trunc_BE}. Integrating equation \eqref{trunc_BE} against $1$ and $|v|^2$ over $\R^d$, we get
$$\frac{\d}{\d t} \int_{ \R^d}  f^1(t, v)\, \d v \leq 0 \qquad \mbox{and}
\qquad  \frac{\d}{\d t}  \int_{ \R^d}  f^1(t, v)\, | v|^2\,  \d v\leq 0.$$
Consequently, $f^1$ satisfies \eqref{ineg} and $\|f^1(T,.)\|_{L^1}\leq  \|f_0\|_{L^1}.$
Since the time $T$ only depends on the inverse of $\|f_0\|_{L^1}$, by a standard continuation argument, we construct a global solution $f$ to \eqref{trunc_BE}. Uniqueness clearly follows from \eqref{stab_2}.

In order to prove Theorem \ref{exi_BE}, we now need to get rid
of the bound in $W^{1,\infty}(\R^d)$ for the initial condition and to pass to
the limit as $n\to +\infty$. Let $f_0\in L^1_{2+\gamma}(\R^d)$ be
a nonnegative distribution function. There
exists a sequence of nonnegative functions $(f_0^n)_{n\in\N}$ in
$W^{1,\infty}(\R^d)\cap L^1_{2+\gamma}(\R^d)$  that converges to $f_0$ in $L^1_2(\R^d)$ and that satisfies, for any $n\in\N$,
\beq\label{majof} \|f_0^n\|_{L^1}\leq \|f_0\|_{L^1}\quad \mbox{ and } \quad
\int_{\R^d} f_0^n(v)\,|v|^{2+\gamma}\, \d v \leq 2^{1+\gamma} \|f_0\|_{L^1}
+2^{1+\gamma} \int_{\R^d}f_0(v)\, |v|^{2+\gamma}\,\d v.
\eeq
We infer from the above properties of $(f_0^n)_{n\in\N}$ that there exists some $N_0\in\N$ such that for $n\geq N_0$,
\beq\label{minof1}
\frac{1}{2}\; \|f_0\|_{L^1} \leq \int_{\R^d}f_0^n(v)\, \d v \leq \|f_0\|_{L^1}
\eeq
 and
\beq\label{minof2}
\frac{1}{2} \int_{\R^d}f_0(v)\,|v|^2\, \d v \leq \int_{\R^d}f_0^n(v)\,|v|^2\, \d v \leq 2 \int_{\R^d}f_0(v)\,|v|^2\, \d v.
\eeq
For each $n\in\N$, we denote by $f_n$ the solution to \eqref{trunc_BE} with
initial condition $f_0^n$. Our purpose is to show that $(f_n)_{n\in\N}$
is a Cauchy sequence in $\C([0,T];L^1_2(\R^d))$ for any $T>0$. However, this
requires uniform estimates on $f_n$. So, we now show uniform bounds for moments of $f_n$.

\begin{lem}\label{lemCT_f}
Let $T>0$ and $s>2$. Assume that $\|f_0\|_{L^1_s}<\infty$.
Then, there exists some constant $C$ depending only on $\a$, $d$, $\g$,
$s$, $T$, $b(\cdot)$ and  $\|f_0\|_{L^1_s}$ such that, for $n\geq N_0$,
\beq\label{f_mom_s}
\sup_{t\in[0,T]} \int_{\R^d} f_n(t, v)\, | v|^{s} \, \d v \leq C \quad
\mbox{ and } \quad\int_0^T \|f_n(t)\|_{L^1} \int_{\R^d} f_n(t, v)\,  \Phi_n(| v|) \,
| v|^s\, \d v \, \d t \leq C.
\eeq
\end{lem}

\begin{proof}
Let $s>2$ and $n\geq N_0$. Our proof follows the same lines as the proof
of Lemma \ref{lemCT}.  As previously, we have
\bean
\frac{\d Y^n_s}{\d t}(t)
& = & \frac{1-\a}{2} \int_{\R^d}\int_{\R^d} f_n(t, v)\, f_n(t, v_*)
\, \Phi_n(| v- v_*|) \, K^n_s( v, v_*)\, \d v\, \d v_* \nonumber \\
& - & \a \int_{\R^d} \Q^n_-(f_n,f_n)(t, v)\, | v|^s\, \d v,
\eean
where $ Y^n_s(t)=\int_{\R^d} f_n(t,v) \, |v|^s\, dv$.
Now, arguing as in the proof of Lemma \ref{lemCT}, we obtain
\bean
& & \hspace{-1cm} \frac{\d}{\d t} Y^n_s(t) + \frac{(1-\a)\, c_2(n)}{2} \;
\|f_n(t)\|_{L^1} \int_{\R^d} f_n(t, v)\,  \Phi_n(| v|) \, | v|^s\, \d v \\
&  & \leq \frac{c_2(n)}{2} \; Y^n_s(t)\,   Y^n_\g(t)+ c_1 \left( Y^n_s(t)\, Y^n_1(t) + Y^n_{s-\g}(t)\, Y^n_{1+\g}(t) \right).
\eean
Finally,
$$ \frac{\d}{\d t} Y^n_s(t) + \frac{(1-\a)\, c_2(2)}{2}\; \|f_n(t)\|_{L^1}
\int_{\R^d} f_n(t, v)\,  \Phi_n(| v|) \, | v|^s\, \d v
\leq  C_3\,  Y^n_s(t) + 2 \, c_1\,\|f_0\|_{L^1_2},$$
where $C_3=(c_2^\infty+4c_1)\|f_0\|_{L^1_2}$.
Then, (\ref{f_mom_s}) follows easily from the Gronwall Lemma and \eqref{majof}.
\end{proof}

Observe that the second inequality of \eqref{mom_s} has to be modified in that case. Since the mass of the solution is decreasing, we do not recover, as previously, that moments of order $2+\gamma$ are integrable. This is the reason why we assume here that the initial condition lies in $L^1_{2+\gamma}$. Thanks to Lemma \ref{lemCT_f}, it then follows that moments of order $2+\gamma$ are uniformly bounded. We are thus in a position to prove that $(f_n)_{n\in\N}$ is a Cauchy sequence in $\C([0,T];L^1_2(\R^d))$ for any $T>0$. We omit the proof since it follows exactly the same lines as the proof \cite[Theorem 4.1]{MiWe99}. Then denoting by $f\in \C([0,T];L^1_2(\R^d))$ the limit of the sequence  $(f_n)_{n\in\N}$, it is easy to check that $f$ is a weak solution to \eqref{BE}. Performing the same calculations as in the proof of Proposition \ref{propostab} (with the $L^1_2$ norm instead of the $L^1_{2+\gamma}$ norm), we prove the uniqueness of such a solution.

\section{The case of Maxwellian molecules kernel}\label{appendixB}

We discuss in this appendix the particular case of Maxwellian molecules.
Notice that the Boltzmann equation for ballistic annihilation associated to Maxwellian molecules has been already studied in the mid-80's \cite{spiga,santos}, and was referred to \textit{as Boltzmann equation with removal}. Consider as above, the equation \begin{equation}\label{BEapp}
\partial_t f(t,v)=(1-\alpha)\Q(f,f)(t,v) -\alpha \Q_-(f,f)(t,v)=\mathbb{B}(f,f)(t,v), \qquad f(0,v)=f_0(v)
\end{equation}
where $\Q$ is the quadratic Boltzmann collision operator associated to the Maxwellian collision kernel
$$\B(v-v_*,\s)=b(\cos \theta)$$
For any solution $f(t,v)$ to \eqref{BEapp}, we denote
$$n(t)=\int_{\R^d}f(t,v)\d v, \qquad n(t)\u(t)=\int_{\R^d}vf(t,v)\d v,$$
and
$$\Theta(t)=\dfrac{1}{d\,n(t)}\int_{\R^d} |v-\u(t)|^2f(t,v)\d v.$$
Since, for Maxwellian molecules
$$\Q_-(f,f)(t,v)=\|b\|_{\L} f(t,v)\int_{\R^d}f(t,v_*)\d v_*=\|b\|_{\L} n(t)f(t,v)$$
one sees easily that the evolution of the density $n(t)$ is given by
\begin{equation}\label{dnt}\dfrac{\d}{\d t}n(t)=-\mu n^2(t),  \qquad \forall t \geq 0,\end{equation}
with $\mu=\alpha \|b\|_{\L}$. Thus
\begin{equation}\label{ntmu}
n(t)=\dfrac{n_0}{\mu \,n_0 t+1}, \qquad \forall t \geq 0.\end{equation}
In the same way,
\begin{equation}\label{dnut} \dfrac{\d}{\d t}(n(t)\u(t))=-\mu\,n^2(t)\,\u(t),\qquad \text{ and } \qquad \dfrac{\d}{\d t}(n(t)\Theta(t))=-\mu\,n^2(t)\Theta(t)\end{equation}
from which we deduce that
$$\u(t)=\u(0) \qquad \text{ and } \qquad \Theta(t)=\Theta(0) \qquad \forall t \geq 0.$$
One sees therefore that, for the special case of Maxwellian molecules, the evolution of the moments of $f(t,v)$ are explicit. Another striking property, very peculiar to Maxwellian molecules, has been noticed in \cite{santos}: if one defines
$$s(t)= \dfrac{1-\alpha}{n_0}\int_0^t n(\tau)\d\tau= \dfrac{1-\alpha}{\mu\,n_0}\log(1+\mu \,n_0 \,t), \qquad t\geq 0,$$
then, the change of unknown
\begin{equation}\label{change} f(t,v)=\dfrac{n(t)}{n_0}g(s(t),v) \qquad t \geq 0\end{equation}
shows that, $f(t,v)$ is a solution to \eqref{BEapp} if and only if $g(s,v)$ is a solution to the classical Boltzmann equation
\begin{equation}\label{BEclass}\partial_s g(s,v)=\Q(g,g)(s,v) \,\:(s >0)\qquad \text{ with } \qquad g(0,v)=f_0(v).\end{equation}
Moreover, one has
$$\int_{\R^d} g(s,v)\d v=n_0=\int_{\R^d} g(0,v)\d v\,,\,\quad \int_{\R^d} vg(s ,v)\d v=n_0 \u(0)$$
and
$$\int_{\R^d} |v-\u(0)|^2 g(s ,v)\d v=dn_0 \Theta(0) \qquad \forall s \geq 0.$$
 In other words, the ballistic annihilation equation \eqref{BEapp} is equivalent to the classical Boltzmann equation with Maxwellian molecules interactions. The mathematical theory of Eq. \eqref{BEclass} is by now completely understood (see e.g. \cite{villani}) and
it is well known that (under suitable conditions on the initial distribution $f_0$)  the solution $g(s,v)$ to \eqref{BEclass} converges (in suitable $L^1$-norm)  as $s \to \infty$ to the Maxwellian distribution
$$\mathcal{M}(v)=\dfrac{n_0}{\left(2\pi\Theta(0)\right)^{d/2}}\exp\left(-\dfrac{|v-\u(0)|^2}{2 \Theta(0)}\right) \qquad v \in \R^d$$
with an explicit rate (we do not wish to explicit the minimal assumption on $f_0$ nor the precise convergence result and rather refer the reader to \cite{villani} for details). Turning back to the original variable, this proves that
$$f(t,v) - \dfrac{n(t)}{n_0}\mathcal{M}(v) \longrightarrow 0\qquad \text{ as } \quad t \to \infty.$$
The long-time behavior of the solution to \eqref{BEapp} is therefore completely described by the evolution of the density $n(t)$ given by \eqref{ntmu} and the moments of the initial datum $f_0$ (through the Maxwellian $\mathcal{M}$). This gives a complete picture of the asymptotic behavior of \eqref{BEapp} and answers the problem stated in Section \ref{interm} for the special case of Maxwellian molecules.

\end{document}